\documentclass[10pt]{amsart}

\date{\today}

\usepackage{graphicx,ctable,booktabs,enumerate, color, hyperref}
\usepackage{amssymb,amsmath,amsthm} 
\usepackage{tikz, tikz-cd}

\usepackage{hyperref, thm-restate}

\usepackage{mathrsfs}
\newcommand{\ms}[1]{\mathscr{#1}}

\newcounter{prcounter}
\newcommand{\pause}{\setcounter{prcounter}{\value{enumi}}}
\newcommand{\resume}{\setcounter{enumi}{\value{prcounter}}}

\newcommand{\lf}{\left}
\newcommand{\ri}{\right}

\newcommand{\f}{\frac} 
 
\newcommand{\into}{\hookrightarrow}

\newcommand{\xto}[1]{\xrightarrow{#1}}
\newcommand{\iso}{\xrightarrow{\sim}}
\newcommand{\wh}{\widehat}
\DeclareMathOperator{\sgn}{sgn}

\DeclareMathOperator{\rank}{rank}

\DeclareMathOperator{\Supp}{Supp}
\DeclareMathOperator{\Aut}{Aut}
\DeclareMathOperator{\Int}{Int}
\DeclareMathOperator{\Out}{Out}
\DeclareMathOperator{\Stab}{Stab}
\DeclareMathOperator{\Gal}{Gal}

\DeclareMathOperator{\Res}{Res}

\DeclareMathOperator{\vol}{vol}

\DeclareMathOperator{\tr}{tr}

\newcommand{\id}{\mathrm{id}}

\DeclareMathOperator{\ch}{ch}

\DeclareMathOperator{\Lie}{\mathrm{Lie}}
\newcommand{\Ld}[1]{{}^L\!{#1}}

\newcommand{\m}[1]{\mathbf{#1}}
\newcommand{\mf}[1]{\mathfrak{#1}}
\newcommand{\mc}[1]{\mathcal{#1}}
\newcommand{\td}[1]{\tilde{#1}}

\newcommand{\C}{\mathbb C}
\newcommand{\R}{\mathbb R}
\newcommand{\Q}{\mathbb Q}
\newcommand{\N}{\mathbb N}
\newcommand{\Z}{\mathbb Z}

\newcommand{\F}{\mathbb F}
\newcommand{\A}{\mathbb A}

\newcommand{\SL}{\mathrm{SL}}
\newcommand{\Sp}{\mathrm{Sp}}
\newcommand{\GSp}{\mathrm{GSp}}

\newcommand{\SO}{\mathrm{SO}}

\newcommand{\GL}{\mathrm{GL}}

\newcommand{\Frob}{\mathrm{Frob}}

\newcommand{\Gm}{\mathbb G_m}
\newcommand{\eps}{\epsilon}
\newcommand{\om}{\omega}
\newcommand{\lb}{\lambda}
\newcommand{\Om}{\Omega}

\newcommand{\dif}[2]{\f{d#1}{d#2}}

\newcommand{\bs}{\backslash}
\newcommand{\1}{\m 1}

\newcommand{\ab}{\mathrm{ab}}
\newcommand{\adj}{\mathrm{ad}}
\newcommand{\scn}{\mathrm{sc}}
\newcommand{\rat}{\mathrm{rat}}
\newcommand{\disc}{\mathrm{disc}}
\newcommand{\cts}{\mathrm{cts}}
\newcommand{\spec}{\mathrm{spec}}
\newcommand{\geom}{\mathrm{geom}}
\newcommand{\el}{\mathrm{ell}}
\newcommand{\bad}{\mathrm{bad}}

\newcommand{\tam}{\mathrm{tam}}
\newcommand{\ur}{\mathrm{ur}}
\newcommand{\ssm}{\mathrm{ss}}

\newcommand{\can}{\mathrm{can}}
\newcommand{\reg}{\mathrm{reg}}
\newcommand{\pl}{\mathrm{pl}}
\newcommand{\wt}{\mathrm{wt}}
\newcommand{\lv}{\mathrm{lv}}
\newcommand{\der}{\mathrm{der}}
\newcommand{\qs}{\mathrm{qs}}
\newcommand{\temp}{\mathrm{temp}}
\newcommand{\ST}{\mathrm{ST}}
\DeclareMathOperator{\obs}{\mathrm{obs}}
\newcommand{\cusp}{\mathrm{cusp}}

\newtheorem{thm}{Theorem}[subsection]
\newtheorem{prop}[thm]{Proposition}

\newtheorem{cor}[thm]{Corollary}
\newtheorem{lem}[thm]{Lemma}

\theoremstyle{remark}
\newtheorem*{note}{Note}

\theoremstyle{definition}
\newtheorem*{dfn}{Definition}

\newcommand{\red}[1]{}

\title[Sato-Tate Equidistribution through the Stable Trace Formula]{Sato-Tate Equidistribution for Families of Automorphic Representations through the Stable Trace Formula}
\author{Rahul Dalal}

\begin{document}
\begin{abstract}
Shin and Templier studied families of automorphic representations with local restrictions: roughly, Archimedean components contained in a fixed $L$-packet of discrete series and non-Archimedean components ramified only up to a fixed level. They computed limiting statistics of local components as either the weight of the $L$-packet or level went to infinity. We extend their weight-aspect results to families where the Archimedean component is restricted to a single discrete-series representation instead of an entire $L$-packet. 

We do this by using a so-called ``hyperendoscopy'' version of the stable trace formula of Ferarri. The main technical difficulties are first, defining a version of hyperendoscopy that works for groups without simply connected derived subgroup and second, bounding the values of transfers of unramified functions. We also present an extension to non-cuspidal groups of Arthur's simple trace formula since it does not seem to appear elsewhere in the literature.
\end{abstract}

\maketitle
\tableofcontents

\section{Introduction}
\subsection{Context}
This writeup generalizes work in \cite{Shi12} and \cite{ST16} on equidistribution of local components of families of automorphic representations (see the summary next section). We roughly extend their weight-aspect to the case where the infinite component can be restricted to a single discrete series instead of an entire $L$-packet. Beyond the main result (theorem \ref{mainresult}), the methods used should also be of interest. In particular, we do somewhat explicit computations with Arthur's stable trace formula and develop techniques to deal with some practical issues that thereby arise. 

Generally, problems of statistics of families automorphic representations are interesting for a few potential reasons.
First, when interpreted classically, such statistics are information on the spectra of lattices in locally symmetric spaces. 

Second, they give so-called globalization results such as \cite[lem 6.2.2]{Art13} through probabilistic method-style arguments. These allow the construction of automorphic forms satisfying desired local conditions. This is important since a very standard technique in studying local representations is to find a global representation with the local representation as a component and then use global methods to study the global representation: see for example the classification in \cite{Art13} or the cohomology formula in \cite{Shi12a}. Globalization results were the motivation for \cite{Shi12}. 

Next, certain bounds on automorphic representations---in particular the generalized Ramanujan conjecture and what it says about the sizes of Fourier coefficients---have various bizarre, unexpected implications. These include some striking ones outside of number theory such as the original construction of expander graphs. See \cite{Sar05} for a review of this subject. As is common in analytic number theory, bounds on averages in families instead of bounds on individual representations are often good enough for these applications. Conveniently enough, average bounds over families are also directly provided by studying statistics. This seems to be the original motivation for studying the problem in \cite{ST16}.

As far as we know, this is the first work to apply the general stable trace formula to computing statistics of automorphic representations. A more common method seems to be using the non-invariant trace formula. This has the advantage of working for very general types of automorphic representations like Maas forms, but the disadvantage of requiring difficult explicit computations that create problems when dealing with general groups (as mentioned later, see \cite{FLM15} and \cite{FL18} for current progress removing this difficulty). One of the key insights of \cite{ST16} is that, for certain families, the nicer abstract properties of terms in the invariant trace formula simplify computations to the point where good error bounds can be derived even for very general groups. As a next step, the more powerful stable trace formula allows generalizing the class of more-easily-studied families. Here, we focus on a first example of distinguishing between elements of an $L$-packet at infinity. Another potential example could be families appearing in cohomologies of locally symmetric spaces---like the type studied in \cite{Ger19} but maybe coming from groups that are not anisotropic. The main trace formula term counting this family comes from endoscopic groups. 

While automorphic representations with components in the same $L$-packet are almost definitionally indistinguishable from the point of view of Galois representations and $L$-functions, they do differ in other important aspects. For example, a discrete series $L$-packet can contain both holomorphic and non-holomorphic discrete series as in the case of $\GSp_4$ (see \cite[\S3.2]{Sch17}). Breaking up $L$-packets is therefore useful in studying, for example, holomorphic Siegel modular forms. Breaking up $L$-packets can be similarly useful for accessing the forms corresponding specifically to the quaternionic discrete series from \cite{GW96}. 

We also hope that this work can serve as a blueprint for other statistics computations using the stable trace formula. We develop some practical methods to get around common difficulties that may arise:
\begin{itemize}
\item
In section \ref{Hformsection}, a version of the hyperendoscopy formula from \cite{Fer07} that works when groups without simply connected derived subgroup appear in hyperendoscopy 
\item
In section \ref{straceformsection}, a generalization of the simple trace formula in \cite{Art89} to non-cuspidal groups with fixed central character datum 
\item
In sections \ref{unramformulasection} and \ref{unramboundsection}, some computations and bounds on unramified transfers 
\end{itemize}
We also attempt to summarize the relevant endoscopy background, focusing mostly on computational practicalities.

We point out some relevant previous work: pseudocoefficients and their simplification of the trace formula were developed by Clozel and Delorme \cite{CD90} and Arthur \cite{Art89}. They were used to study statistics of families by Clozel \cite{Clo86}. The exact families studied and the setup to study them are of course a small modification from \cite{Shi12} and \cite{ST16}. The use of the stable trace formula is through the hyperendoscopy formula in \cite{Fer07}, although the results of \cite{Pen19} give a different potential strategy. The paper \cite{KWY18} solves this problem for $\GSp_4$ with far more explicit bounds through different methods. For a fuller history of this field of ``limit multiplicity''-type problems, see the introduction to \cite{FLM15}. As for using the theory of endoscopy to count automorphic representations, there are a few articles by Marshall and collaborators, such as \cite{Mar14}, \cite{MS19}, and \cite{Ger19}, that use endoscopic character identities to bound cohomology dimensions of symmetric spaces for certain unitary groups.

Finally, the results here should be compared to \cite{FLM15} and \cite{FL18} by Finis, Mueller and Lapid. These use the non-invariant trace formula to develop similar though much more general results. In particular, they show Shin and Templier's level aspect with the Archimedean component restricted to any set of positive measure in the unitary dual. The result is dependent on some technical estimates on intertwining operators that are satisfied for $\GL_n$ and $\SL_n$. A future work promises the estimates for most other groups. In addition, their methods do not currently deal with the weight aspect or give error bounds though they could presumably be pushed to do both.

\subsection{Summary}
\subsubsection{Shin-Templier's work} Let $G$ be a reductive group satisfying some technical conditions (described in section \ref{cndtns}). In \cite{ST16} building off \cite{Shi12}, Shin and Templier studied certain families of automorphic representations with level and weight restrictions:
\[
\mc F_{U, \xi} = \{\pi \in \mc{AR}_\disc(G) : \pi_\infty \in \Pi_\disc(\xi), \dim (\pi^\infty)^U \geq 1\}
\]
where $\mc{AR}_\disc(G)$ is the set of the discrete automorphic representations of $G$, $U$ is an open compact subgroup of $G(\A^{\infty, S_0})$ for some finite set of places $S_0$, $\xi$ is a regular weight of $G_\C$, and $\Pi_\disc(\xi)$ is the discrete series $L$-packet corresponding to $\xi$. Pick another finite set of places $S \supseteq S_0$ and consider the empirical distribution 
\[
\mu_{\mc F, S} = \sum_{\pi \in \mc F_{U, \xi}} a_\pi \delta_{\pi_S}
\]
of $S$-components of $\pi \in \mc F$ weighted by
\[
a_\pi = m_\disc(\pi)  \dim(\pi^{S, \infty})^U.
\]

Shin and Templier used Arthur's invariant trace formula to study the limits of these distributions under either increasing level ($U \to 1$) or increasing weight ($\xi \to \infty$). In both cases, the limits converged to the Plancherel measure. They furthermore provided bounds on how quickly the integrals $\mu_{\mc F, S}(f)$ converge in the case where both $f$ and the elements of $\mc F$ are unramified on $S \setminus S_0$. The increasing weight aspect required that the center of $G$ was trivial. 

Their method was in a few broad steps:
\begin{enumerate}
\item
Realize the empirical distribution $\mu_{\mc F, S}$ as the trace of a function with a special Archimedean component $\eta_\xi$ against the discrete automorphic spectrum. Here, $\eta_\xi$ is the Euler-Poincar\'e function from \cite{CD90}.
\item
Since the Archimedean component is an Euler-Poincar\'e function, Arthur's invariant trance formula reduces to the simple trace formula in \cite{Art89} giving a reasonably tractable expression for this trace.
\item
Bound the appropriate terms and take a limit. This is most of the work.
\end{enumerate}

The form of the error bound allowed the proving of Sato-Tate equidistribution limits of $\mu_{\mc F, v}$ for a single place $v$ as $v$ and $\xi$ jointly go to infinity. They also provided some results on the statistics of low-level zeros of $L$-functions over the entire family. 

\subsubsection{The extension}
Here, we extend Shin-Templier's weight aspect described in the previous section. First, instead of looking at a sequence of entire $L$-packets $\Pi_\disc(\xi_k)$, we fix a single representation $\rho_k \in \Pi_\disc(\xi_k)$ for each $k$. Second, we allow $G$ to have trivial center. 

Then we consider the limit as $k \to \infty$ of the empirical distribution 
\[
\mu_{\mc F_k, S} = \sum_{\pi \in \mc F_{U, \rho_k}} a_\pi \delta_{\pi_S}
\]
of representations with $\pi_\infty = \rho_k$ weighted by
\[
a_\pi = m_\disc(\pi)  \dim(\pi^{S, \infty})^U.
\]
and compute error bounds on its convergence to Plancherel measure. The precise definition of the family we study is in section \ref{cndtns} and the final result is theorem \ref{mainresult}. 

Here are the broad steps of the argument:
\begin{enumerate}
\item
Realize the empirical distribution $\mu_{\mc F, S}$ as the trace of a function with a special Archimedean component $\varphi_\pi$ against the discrete automorphic spectrum. The function $\varphi_\pi$ is the pseudocoefficient from \cite{CD90}.
\item
Notice that pseudocoefficients have the same stable orbital integrals as Euler-Poincar\'e functions.
\item
Use the stable trace formula to write this trace as a linear combination of traces of functions with Euler-Poincar\'e components at infinity on the smaller endoscopic groups.
\item
Proceed as before to bound each term in the sum. Showing that enough technical conditions are satisfied and that the bounds are uniform enough that you are allowed to do so is most of the new work.
\item
Redo the computations showing the versions of Plancherel and Sato-Tate equidistribution that the new main term gives. 
\end{enumerate}

It is worth discussing these in more detail. For step (3), the key difficulty is that Arthur's simple trace formula only works when the Archimedean component is Euler-Poincar\'e instead of a pseudocoefficient. However, the stable trace formula roughly gives the trace of a function as a linear combination of stable traces of transfers of the function on smaller endoscopic groups---we get an expansion of shape:
\[
I^G(f) = \sum_{H \in \mc E_\el(G)} S^H(f^H).
\]
Since pseudocoefficients have the same stable orbital integrals as their corresponding Euler-Poincar\'e functions, the $f^H$ can without loss of generality be chosen to have Euler-Poincar\'e components at infinity. See section \ref{archtransfersection} for details on these transfers.

The most direct way to proceed is to then repeat the work in \cite{Art89} on the stable distributions $S^H$ instead of the invariant distribution $I^G$. We choose to instead use the hyperendoscopy formula from \cite{Fer07} (see the remark at the beginning of section \ref{Hformsection}). 

It gives an expansion of shape
\[
I^G(f) = I^G(f^*) + \sum_{\mc H \in \mc{HE}_\el(G)} \iota(G, \mc H) I^{\mc H}((f - f^*)^{\mc H}).
\]
Here $f^*$ is a function with the same stable orbital integrals as $f$, $\mc{HE}_\el(G)$ is roughly the set of groups that can show up in a sequence of iteratively choosing an endoscopic group starting from $G$, and $\iota(G,\mc H)$ is a non-troublesome constant. See section \ref{Hformsection} for the full details. The distributions $I^{\mc H}$ can then be treated exactly as in \cite{ST16} provided technical conditions still hold.

There are also some complications in step (4). First, the distribution $I^G_\spec(f)$ is not obviously the trace of $f$ against the discrete automorphic spectrum like we want it to be. The paper \cite{Art89} shows this for Euler-Poincar\'e at infinity and an unpublished lemma of Vogan (appearing here as lemma \ref{pstrace}) is needed to extend to the pseudocoefficient case. Next, the groups appearing in $\mc{HE}_\el(G)$ do not satisfy the technical simplifying conditions of \cite{Art89}. We therefore need to slightly generalize the result, in particular to non-cuspidal groups. This is section \ref{straceformsection}. Finally, we need some bounds on endoscopic transfers of test functions so that Shin-Templier's orbital integral bounds apply. This takes some work in the non-Archimedean case and is sections \ref{unramformulasection} and \ref{unramboundsection}.

For step (5), as explained in section \ref{themainterm}, allowing a non-trivial center changes the main term in theorem \ref{mainresult} to something more complicated than originally in \cite{ST16}. We therefore have to redo the computations for Sato-Tate and Plancherel equidistribution. This produces slightly different limiting measures that can be roughly thought of as Sato-Tate or Plancherel measure conditioned to be on a certain subset of $\wh G_S$: representations with central character contained in a particular discrete set. The computations appear in section \ref{corollaries}. We do not do the computation for low-lying zeros of $L$-functions due to complexity. 

Finally, we save the level aspect computation for a future writeup. The main difficulty here is that as level gets larger, the test function $f$ becomes more and more ramified adding more and more non-zero terms to the sum over $\mc{HE}_\el(G)$. This necessitates proving much stronger uniformity of the bounds in \cite[\S8]{ST16} over endoscopic groups. 

\subsection{Acknowledgements}
This work was done under the support of NSF RTG grant DMS-1646385 and the UC Berkeley math department summer grants for 2017 and 2018. Part of it was done while under funding from the workshop ``On the Langlands Program: Endoscopy and Beyond" at the National University of Singapore. 

I would like to thank my advisor Sug Woo Shin for suggesting this problem and for many hours teaching relevant material, discussing strategies for solving the problem, pointing out helpful references, and checking arguments. I would also like to thank Jasmin Matz and Erez Lapid for showing me some tricks for dealing with central character issues, Tasho Kaletha for much help understanding intuitions behind endoscopy, Peter Sarnak for a helpful discussion on the context surrounding this type of problem, Julia Gordon and Silvain Rideau for help understanding motivic integration and its applicability, David Vogan for lemma \ref{pstrace}, Alex Youcis for teaching many useful facts about algebraic groups, Aaron Landesman for help with some cohomology arguments in proposition \ref{zextpath}, Jeremy Meza and Alexander Sherman for helpful discussions about the representation theory in sections \ref{sectionArchtransferbounds} and \ref{unramformulasection}, Ian Gleason and Ravi Fernando for help with the proof of lemma \ref{aboundgap}, and Alexander Bertoloni-Meli for many, many helpful discussions that touched almost every argument involving endoscopy. 

Finally, much thanks to the reviewers for a very thorough reading, invaluable help with making my writing clearer, and for catching a lot of errors I had in earlier drafts. 

\subsection{Notational Conventions} Here are some notational conventions we will use throughout:

\noindent Basics
\begin{itemize}
\item
$F$ is a fixed number field.
\item
$G$ is a fixed reductive group over $F$. In certain sections where we are working locally, $G$ will be the local component instead. 
\item
$\A$ is $\A_F$ for shorthand.
\item
$\A_\infty$, $\A^\infty$ are the at infinity and away from infinity parts of $\A$ respectively.
\item
$W_E$ is the Weil group of local or global field $E$. 
\item
$\mc O_E$ is the ring of integers of local field $E$.
\item
$k_E$ is the residue field of local field $E$. 
\item
$\1_X$ is the indicator function for set $X$.
\item
$\wh H$ is the reductive dual of reductive group $H$.
\item
$\wh S$ is the unitary dual of abstract group $S$.  
\item
$\wh S^\temp$ is the tempered part of $\wh S$. 
\item
$\wh f$ is the Fourier transform of function $f$ on an abstract group $S$ that should be clear from context.
\item
$\bar f$ is the Fourier transform of $f$ restricted to some subgroup of the center of $S$ with respect to that subgroup. The exact subgroup should be clear from context. 
\end{itemize}
\noindent Reductive Groups
\begin{itemize}
\item
$Z_H$ is the center of abstract or reductive group $H$. 
\item
$Z_H(G)$ is the centralizer of $H$ inside $G$. 
\item
$A_H$ is the maximum split component in the center of reductive group $H$. 
\item
$H_S$ for group $H$ over $F$ and finite set of places $S$ of $F$ is $H(\A_S)$. Use the standard conventions where an upper index means everything except $S$. 
\item
$H_\infty$ may be equivalently defined as $(\Res^F_\Q H)(\R)$ since $(\Res^F_\Q H)(\R) = H(F \otimes_\Q \R) = H(\A_\infty)$. It is in particular a real reductive group.
\item
$A_{H, \rat}$ for group $H$ over $F$ is $A_{\Res^F_\Q H}(\R)^0$ (the connected component is in the real topology).
\item
$A_{H,\infty} := A_{(\Res^F_\Q H)_\R}(\R)^0$.
\item
$H(\A)^1 := H(\A)/A_{H, \rat}$.
\item
$H_\infty^1 := H_\infty/A_{H, \infty}$. 
\item
$H_\gamma$ is the centralizer of $\gamma$ in $H$ for $H$ either an algebraic or abstract group.
\item
$I^H_\gamma$ is the connected component of the identity in the centralizer of $\gamma$ in $H$.
\item
$\iota^H(\gamma)$ is the set of connected components of $H_\gamma$ with an $F$-point.
\item
$[H]$, $[H]^\ssm$, $[H]^\el$ are the sets of (semisimple, elliptic) conjugacy classes in $H$.
\item
$D^H(\gamma)$ is the Weyl discriminant for $H$.
\item
$K_S$ where $S$ is a finite set of places of $F$ is a chosen hyperspecial of $G(\A_S)$.
\item
$M$ usually represents some Levi subgroup.
\item
$P$ usually represents some parabolic subgroup.
\item
$K_{S,H}$ for $S$ some finite set of places usually represents some kind of maximal compact of $H(\A_S)$. 
\end{itemize}
\noindent Lie Theory
\begin{itemize}
\item
$\Phi^*(H), \Phi^+(H), \Phi^*_F(H), \Phi^+_F(H)$ are the sets of (positive, rational) roots of $H$. 
\item
$\Phi_*(H), \Phi_{+}(H), \Phi_{*,F}(H), \Phi_{+,F}(H)$ are the sets of (positive, rational) coroots of $H$. 
\item
$\Delta^*(H), \Delta^*_{F}(H)$ are the sets of (rational) simple roots of $H$. 
\item
$\Delta_{*}(H), \Delta_{*,F}(H)$ are the sets of (rational) simple coroots of $H$. 
\item
$\Om_H$ is the Weyl group of $H_\C$ for $H$ a reductive group.
\item
$\Om_{H,E} = \Om_E$ for $H$ over $F$ and $E$ an extension of $F$ is the subset of $\Om_H$ generated by conjugating by elements of $H(E)$. Note that this depends on the maximal torus chosen to define $\Om$. 
\end{itemize}
\noindent Volumes
\begin{itemize}
\item
$\mu^\tam, \mu^\can, \mu^{EP}$ are the Tamagawa, Gross' canonical, or Euler-Poincar\'e measures on various groups.
\item
$\bar \mu^\star$ is the quotient of measure $\mu^\star$ by something that should be clear from context.
\item
$\tau(H)$ is the Tamagawa number of $H$.
\item
$\tau'(H)$ is the modified Tamagawa number using the canonical measure $\mu^{\can, EP}$. 
\end{itemize}
\noindent Endoscopy
\begin{itemize}
\item
$(H, \mc H, s, \eta)$ is an endoscopic quadruple for $G$. 
\item
$(\td H, \td \eta)$ is a $z$-pair for $(H, \mc H, s, \eta)$.
\item
$(H_1, \eta_1)$ will also sometimes be used to represent a $z$-pair to keep diacritics from stacking too much.
\item
$\mc E_\el(H)$ is the set of elliptic endoscopic quadruples of reductive group $H$.
\item
$\mc{HE}_\el(H)$ is the set of elliptic hyperendoscopic paths of reductive group $H$.
\item
$(\mf X, \chi)$ is a central character datum on some reductive group.
\item
$\mc H$ is further overloaded: when context is clear, it can also refer to either a hyperendoscopic path or the last group in the path.
\end{itemize}
\noindent Automorphic Representations and the Trace Formula
\begin{itemize}
\item
$\ms H(H, \chi) = \ms H(H, (\mf X, \chi))$ is the space of compactly supported functions on $H(\A)$ that transform according to character $\chi^{-1}$ on $\mf X \subseteq Z_G(\A)$. 
\item
$\ms H(H_S, \chi_S)$ for $S$ a finite set of places of $F$ is compactly supported functions on $H(\A_S)$ similarly transforming according to $\chi_S^{-1}$.
\item
$\ms H(H_S, K_S, \chi_S)$ if $K_S$ is a product of hyperspecial subgroups and $\chi_S$ is unramified is the Hecke algebra of $K_S$-biinvariant elements of $\ms H(H_S, \chi)$. 
\item
$\ms H(H_S, K_S, \chi_S)^{\leq \kappa}$ is the truncated Hecke algebra from section \ref{trunhalg}.
\item
$L^2(G(\Q) \bs G(\A), \chi)$ for $(\mf X, \chi)$ a central character datum is the unitary $G(\A)$-representation of $L^2$-up-to-$\mf X$ functions on $G(\Q) \bs G(\A)$ transforming according to $\chi^{-1}$. 
\item
$L^2_\disc( \cdot )$ is the discrete part of unitary representation $L^2(\cdot)$. 
\item
$\mc{AR}_\disc(H, \chi)$ is the set of discrete automorphic representations on $H$ with character $\chi$ on $A_{H, \infty}$.  
\item
$O^H_\gamma(f)$ is the integral of $f$ on the conjugacy orbit of $\gamma$. This can either be local or global; $f$ can be a function on $H(\A)$ or some $H(F_v)$. 
\item
$I_\spec^{G,\chi}, I_\disc^{G,\chi}, I_\geom^{G,\chi}$ are the distributions on $G$ defined by Arthur's invariant trace formula depending on central character datum $(\mf X, \chi)$.
\item
$S_\spec^{H, \chi}, S_\disc^{H, \chi}, S_\geom^{H, \chi}$ are the distributions on $H$ defined by Arthur's stable trace formula depending on central character datum $(\mf X, \chi)$.
\item
$\ms L$ is the set of rational Levi's of $G$ containing a fixed minimal Levi.
\item
$\ms L^\cusp$ is the $M \in \ms L$ such that $A_{M, \rat}/A_{G, \rat} = A_{M, \infty}/A_{G, \infty}$. This is a generalization of the definition of cuspidal Levi from \cite{Art89} to the case where $G$ isn't itself cuspidal.
\end{itemize}
\noindent Representation theory
\begin{itemize}
\item
$\pi(\lb, w_0)$, $\pi(w_0(\lb + \rho))$ are two different parametrizations for discrete series representations for $\lb$ a dominant weight. 
\item
$\Pi_\disc(\lb)$ is a discrete series $L$-packet where $\lb$ is a dominant weight. 
\item
$\Theta_\pi$ is the Harish-Chandra character for representation $\pi$. 
\item
$\om_\pi$ is the central character of representation $\pi$. 
\item
$\varphi_{\pi}$ is the pseudocoefficient for discrete series representation $\pi$.
\item
$\eta_\lb$ is the Euler-Poincar\'e function for the $L$-packet $\Pi_\disc(\lb)$. 
\end{itemize}
\noindent Families
\begin{itemize}
\item
$\varphi^\infty$ is a specific function defined in section \ref{cndtns}. 
\item
$\mc F$ is a specific family (as in \cite{ST16}) of automorphic representations defined in section \ref{cndtns}.
\item
$a_{\mc F}(\pi)$ are the coefficients defining $\mc F$.  
\item
$S_0$, $S_1$, $U^{S \cup \infty}$, $\varphi_{S_1}$, $f_{S_0}$ are data used to define $\varphi^\infty$ and $\mc F$ as explained in section \ref{cndtns}. 
\item
$S_{\bad, G}$ is the unknown finite set of bad places depending on reductive group $G$ defined in section \ref{final}. 
\item
$S_{\bad', G}$ is the version of $S_{\bad, G}$ needed for the results from \cite{ST16}.
\item
$L$ is the lattice $Z_G(F) \cap U^{S, \infty} \subseteq Z_{G_{S, \infty}}/A_{G, \rat}$. 
\item
$E^\pl(\wh \varphi | \om)$ is the expectation defined in section \ref{cft}.
\item
$E^\pl(\wh \varphi_S | \om_\xi, L, \chi_S)$ is defined in proposition \ref{wtbound}.
\end{itemize}
\subsubsection{Dimensional Analysis}
A lot of the formulas here depend on choices of Haar measure. Since we are explicitly bounding terms, it is sometimes helpful to have notation for how they depend on these choices. For example, if we say that a value has dimension $[G][H]^{-1}$, then it is proportional to a choice of Haar measure on $G$ and inversely proportional to a choice on $H$. 

In any formula, dimensions on both sides need to match. In addition, any quantity with dimension needs to be normalized by a formula expressing it in terms of just dimensionless quantities and Haar measures---for example, the formulas defining traces of Hecke algebra elements, orbital integrals, or pseudocoefficients.

\section{Trace Formula Background}
\subsection{Invariant Trace Formula}
Let $G$ be a connected reductive group over a number field $F$. Let $\A = \A_F$. Fix a central character $\chi$ on $A_{G, \rat}$. Let $\ms H(G, \chi)$ be the space of functions on $G(\A)$ that are smooth and compactly supported when restricted to $G(\A)^1$ and satisfy $f(ax) = \chi^{-1}(a) f(x)$ for all $a \in A_{G, \rat}$. 

Over a long series of papers that are summarized in \cite{Art05} Arthur defines two equal distributions on $\ms H(G, \chi)$:
\[
I_\geom^{G, \chi} = I_\spec^{G,\chi}.
\]
Intuitively, one should think of $I_\geom$ as a sum of modified orbital integrals of $f$ and $I_\spec$ as a sum of modified traces of $f$ against components of $L^2(G(\Q) \bs G(\A), \chi)$. The exact definitions of these distributions are impractically complicated to use directly. However, enough useful special cases and abstract properties have been worked out---the most relevant being the simple trace formula in \cite{Art89}.  The $\chi$ will often be suppressed in notation.

Both sides have dimension $[G(\A)^1]$. The individual terms in the expansions for both sides can have more complicated dimensions.

\subsubsection{Spectral side}
As a very rough description of the spectral side, Arthur defines components
\[
I_\spec^G = I_\cts^G + \sum_{t \geq 0} I_{\disc,t}^G.
\]
$I_{\disc,t}$ is $0$ except for countably many $t$ and is much easier to evaluate. Expanding further,
\[
I_{\disc, t} = \sum_{M \in \ms L} \f{|\Om_M|}{|\Om_G|} \sum_{w \in W(M)_\reg} |\det(w-1)|_{\mf a^G_M}|^{-1} \tr(M_{P,t}(\om)\mc I_{P,t}(f)).
\]
To describe the most relevant terms, $\ms L$ is the set of Levi's of $G$ containing a chosen minimal Levi, $P$ is a chosen parabolic for $M$, $W(M)_\reg$ is a particular set of elements of a relative Weyl group (this and the Weyl group factor are a combinatorial term roughly parametrizing parabolics containing the Levi), and $M_{P,t}(\om, \chi)$ is an intertwining operator between parabolic inductions through different parabolics containing $M$ from the theory of Eisenstein series.

The last term is the most important for us. The $\chi$ induces a character on $A_{M, \rat}$ by pullback. Then $\mc I_P(\chi)$ is the representation of $G(\A)$ produced from parabolically inducing $L^2_\disc(M(\Q) \bs M(\A), \chi)$. The term $\mc I_{P,t}$ is the subrepresentation of this with archimedean infinitesimal character having imaginary part of norm $t$. By lots of work, all these decompositions makes sense and the convolution operators $\mc I_{P,t}(f)$ for $f \in \ms H(G, \chi)$ are trace class. Finally, a much later result in \cite{FLM11} implies that the sum over $t$ converges absolutely.

There are well-known and simple sufficient conditions on $f$ such that $I_\cts(f) = 0$:
\begin{dfn}[{\cite[paragraph above cor. 23.6]{Art05}}]
If $v$ is a place of $F$, $f \in \ms H(G(F_v))$ is \emph{cuspidal} if for all Levi's $M_v$ of $G_v$ and $\pi_v$ tempered representations of $M_v$:
\[
\tr_{\pi_v^G}(f) = 0. 
\]
Here $\pi_v^G$ is (any) parabolic induction of $\pi_v$.
\end{dfn}
Note that this is an alternate definition to the original one from \cite{Art88}.
\begin{thm}[{\cite[thm 7.1]{Art88}}]
If $f$ factors as $f_v \otimes f^v$ for some place $v$ with $f_v$ cuspidal, then $I_\cts(f) = 0$. 
\end{thm}

\subsubsection{Geometric side}
The geometric side can be succinctly written as
\[
I_\geom(f)  = \sum_{M \in \ms L} \f{|\Om_M|}{|\Om_G|} \sum_{\gamma \in [M(\Q)]_{M,S}} a^M(S, \gamma) I^G_M(\gamma, f).
\]
Here $S$ is a large enough set of places in particular including those at which $f$ is not the characteristic function of a hyperspecial and $[M(\Q)]_{M,S}$ is the set of conjugacy classes mod a complicated equivalence relation involving the away-from-$S$ components of the unipotent parts. For $\gamma$ semisimple,
\[
a^M(S, \gamma) = |\iota^M(\gamma)|^{-1} \vol(I^M_\gamma(\Q) \bs I^M_\gamma(\A)^1)
\]
where $|\iota^M(\gamma)|$ is the number of connected components of $M_\gamma$ that have an $F$-point. In general, there is no explicit description of $a^M(S, \gamma)$. 

Next, $I^G_M$ is a weighted orbital integral of the $S$-components of $f$. If $M=G$, it is simply the orbital integral at $\gamma$. If $\gamma$ is semisimple, there is an explicit formula weighting the integral by a complicated combinatorial factor. Otherwise, it is only defined though some analytic continuations. The term $I^G_M$ satisfies some splitting formulas (\cite[23.8]{Art05} and \cite[23.9]{Art05}) that factor it into local components in terms of traces of $f$ against parabolic inductions. When $f$ is cuspidal at some place, these splitting formulas of course then greatly simplify. 

If $\gamma$ is semisimple, the $a^M$ have dimension $[I^M_\gamma(\A)^1]$ while the $I^G_M$ have dimension $[G(\A)^1][I^M_\gamma(\A)^1]^{-1}$. Otherwise the dimensions are more complicated. 

\subsubsection{A technicality}
Arthur actually defines two slightly different versions of his local distributions $I^G_{M,v}(\gamma, f)$. Looking at just the place at $\infty$ for notational ease, the key issue is that the weighting factor $v_M$ in his orbital integrals depends on a choice of the space $A_{M, \star}/A_{G, \star}$ where $\star \in \{\infty, \rat\}$. 

The version appearing in his splitting formula is $\star = \rat$ ,which we will denote by $I^G_{M,\infty}(\gamma, f)$.  The version in his descent formula is the purely local choice  $\star = \infty$, which we will denote by $\td I^G_{M,\infty}(\gamma, f)$.
\begin{lem}\label{centertechnicality}
If cuspidal $f \in \ms H(G_\infty, \chi)$ (so that $I^G_{M,\infty}(\gamma, f)$ is defined), then
\[
I^G_{M,\infty}(\gamma, f) = 
\begin{cases}
\td I^G_{M,\infty}(\gamma, f) & A_{M, \rat}/A_{G, \rat} = A_{M, \infty}/A_{G, \infty} \\
0 & \text{else}
\end{cases}.
\]
\end{lem}

\begin{proof}
If the two spaces are equal, then the weighting factors $v_M$ at the beginning of \cite{Art05}[\S18] and the sum over Levi's in \cite{Art05}[thm. 23.2] are equal. Note that while $\wh I^L_M(\gamma, \phi_L(f))$ in \cite{Art05}[thm. 23.2] ostensibly looks like it depends on the choice of $\star$, this is just based on different descriptions of certain spaces of functions to make conditions for containment in the two versions of $\mc I_{\text{ac}}$ easier to describe. In particular, the distinction does not matter as long as $f$ is in both versions of $\ms H_{\text{ac}}$. In total, stepping through the definitions of $I^G_M$ and $\td I^G_M$ shows that they are the same since the above are the only parts that depend on the various $A$'s. 

Otherwise, this follows from the generalized descent formula \cite{Art88I}[thm. 8.1] with $\mf b$ corresponding to $A_{M, \rat}A_{G, \infty}$. Since $f$ is cuspidal, the only possibly non-zero term is $L = G$. However, then $d^G_M(\mf b, G) = 0$.   
\end{proof}

\subsection{The Simple Trace Formula}
Whenever $G_\infty$ has discrete series, the trace formula can be simplified by setting the test function to have a special real component. 
\subsubsection{Parametrizing discrete series}\label{dseries}
The classification of discrete series is work of Harish-Chandra that can be found summarized in \cite[\S III.5]{Lab11}. They only exist when $G_\infty$ has an elliptic maximal torus or equivalently if $C \Om_G$ on any torus contains $-\id$ where $C$ is complex conjugation. 

Therefore, for this subsection and the next only, let $G$ be reductive group over $\R$ with fixed elliptic maximal torus $T$. Let $K$ be a maximal compact of $G(\R)$ containing $T(\R)$, $B_K$ a Borel of $K_\C$ containing $T$, and $B$ a Borel of $G_\C$. Let $\Om_G$ be the Weyl group of $(G_\C, T_\C)$ and $\Om_{G,\R}$ be the subgroup given by only conjugating by elements of $G(\R)$. 

The characters of $T(\R)$ are contained in $T(\C)$ so the root space of $K$ is contained in $G$. Let $\rho$ be half the sum of the positive roots of $G$. Finally, let $\Om(B_K)$ be a particular set of coset representatives of  $\Om_{G,\R} \bs \Om_G$: namely, $w$ such that $w\lb$ is $B_K$-dominant for any $\lb$ that is $B$-dominant.

The discrete series representations of $G$ are parametrized by $B$-dominant weights $\lb \in X^*(T)_\C$ and elements $w^* \in \Om(B_K)$. Call the representation parameterized by $\lb$ and $w_0$ either $\pi(\lb, w_0)$ or $\pi(w_0(\lb + \rho))$. It is the unique representation with trace character
\[
\Theta_{\pi(\lb, w_0)} = (-1)^{1/2\dim(G(\R)/K A_{G, \infty})} \frac{\sum_{w \in \Om_K} \sgn(ww_0) e^{ww_0(\lb + \rho)}}{\sum_{w \in \Om_G} \sgn(w) e^{w\rho}}
\]
on $T$. The infinitesimal character of $\pi(\lb, w)$ is $\lb + \rho$: the same as that of $V_\lb$, the finite dimensional representation with highest weight $\lb$. Therefore the $\pi(\lb, w)$ for a fixed $\lb$ are all in the same $L$-packet $\Pi_\disc(\lb)$. We call $\pi(\lb, w_0) = \pi(w_0(\lb + \rho))$ \emph{regular} if $\lb$ is. Finally, we call $\lb$ the \emph{weight} of $\pi(\lb, w_0) = \pi(w_0(\lb + \rho))$. 

\subsubsection{Pseudocoefficients and Euler-Poincar\'e functions}
Given a discrete series representation $\pi$ of a real reductive group $G(\R)$ with character $\chi$ on $A_{G, \infty}$, Clozel and Delorme in \cite{CD90} define a pseudocoefficient $\varphi_\pi \in C_c^\infty(\chi^{-1})$. The function $\varphi_\pi$ is compactly supported and has the property that for irreducible representations $\rho$ with character $\chi$, 
\[
\tr_\rho(\varphi_\pi) = \begin{cases} 1 & \pi = \rho \\0 & \pi \neq \rho, \rho \text{ basic} \\ ? & \text{else} \end{cases}.
\]
Here, a basic representation is a parabolic induction of a discrete series or limit of discrete series. The non-basic case is much more complicated. Pseudocoefficients have dimension $[G(\R)^1]^{-1}$. 

If $\Pi_\disc(\lb)$ is the discrete series $L$-packet for $\pi$, it is also useful to consider Euler-Poincar\'e functions:
\[
\eta_\lb = \f1{|\Pi_\disc(\lb)|} \sum_{\pi' \in \Pi_\disc(\lb)} \varphi_{\pi'}.
\]
Traces against Euler-Poincar\'e functions can be interepreted as Euler characteristics of certain cohomologies for basic representations and therefore all representations by the Langlands classification. If $\lb$ is regular, these Euler characteristics can be shown to be $0$ on non-tempered representations. Therefore, if $\lb$ is regular we get
\[
\tr_\rho(\eta_\lb) = \begin{cases} |\Pi_\disc(\lb)|^{-1} & \pi \in \Pi_\disc(\lb) \\0 & \text{else} \end{cases}
\]
for all irreducible representations $\rho$ (see sections 1 and 2 in \cite{Art89}). Beware that this normalization is different from the one in \cite{ST16}. It makes endoscopic computations easier.

Note that both pseudocoefficients and Euler-Poincar\'e functions are cuspidal since they have $0$ trace against any non-discrete series basic representation and therefore against all parabolic inductions of tempered representations.

\subsubsection{Simple trace formula}\label{simpletraceformula}
The simple trace formula is the main result of \cite{Art89}. A more textbook exposition is in \cite[\S24]{Art05}. We state it here. First, assume
\begin{itemize}
\item
$G$ is connected,
\item
$G$ is cuspidal over $\Q$: $\Res^F_\Q G/A_{G, \rat}$ has an $\R$-anisotropic maximal torus. 
\end{itemize}
The last condition in particular gives that $G_\infty$ has an elliptic maximal torus and therefore has discrete series mod center. In the case where $G_\infty$ has discrete series mod center, cuspidal is equivalent to $A_{G, \rat} = A_{G,\infty}$: in other words, taking infinite place points of the maximum split torus in the center is the same as base changing to $\R$, looking at the maximal split torus in the center, and taking $\R$-points.

 Consider a test function of the form $h = |\Pi_\disc(\xi)|\eta_\xi \otimes h^\infty$ for regular weight $\xi$ and $h^\infty \in \ms H(G(\A^\infty))$. Let $\chi$ be the character on $A_{G, \infty}$ determined by $\xi$. Then
\begin{equation}\label{spec}
I_\spec(h) = I_\disc(h) = \sum_{\pi : \pi_\infty \in \Pi_\disc(\xi)} m_\disc(\pi) \tr_{\pi^\infty}(h^\infty)
\end{equation}
where $m_\disc(\pi)$ is the multiplicity of $\pi$ in $\mc{AR}_\disc(G, \chi)$. 
Let $\ms L$ be the set of Levi's containing a chosen minimal Levi of $G$. For each $M \in \ms L$, choose $P_M$ a parabolic for $M$. Then
\begin{multline}\label{geom}
I_\geom(h) = \sum_{M \in \ms L^\cusp} (-1)^{\dim(A_M/A_G)} \f{|\Om_M|}{|\Om_G|} \\
 \sum_{\gamma \in [M(F)]^{\ssm}} \chi(I^M_\gamma) |\iota^M(\gamma)|^{-1} \Phi_M(\gamma_\infty, \xi) O^M_\gamma(h^\infty_M).
\end{multline}
Here $\iota^M(\gamma)$ is the set of connected components of the full centralizer $M_\gamma$ that have an $F$-point and
\[
\chi(H) = (-1)^{q(H)} \vol(H(F) A_{H, \infty}^0 \bs H(\A)) \vol(A_{H, \infty}^0 \bs \bar H_\infty)^{-1}|\Om(B_{K_{H_\infty}})|
\]
where $\bar H_\infty$ is an inner form of $H_\infty$ such that $H_\infty/A_{H, \infty}$ has anisotropic center, $\Om(B_{K_{H_\infty}})$ is the analog of $\Om(B_K)$ for $H_\infty$, and $q(H) = 1/2 \dim(H_\infty/K_{H,\infty} A_{H,\infty})$ is the Kottwitz sign. Also
\[
h^\infty_M(\gamma^\infty) = \delta_{P_M}(\gamma^\infty)^{1/2} \int_{K^\infty} \int_{N_M(\A^\infty)} h(k^{-1} \gamma^\infty n k) \, dn \, dk
\]
where $N_M$ is the unipotent group for $P_M$ and $K$ some chosen maximal compact. To make dimensions work out, the Haar measures choices should satisfy:(
\begin{itemize}
\item
The choices on $I^M_\gamma$, $M$, and in the orbital integral need to coincide,
\item
The measure on $\bar I^M_\gamma$ comes from that on $I^M_\gamma$ through them both coming from the same top form on $I^M_\C$,
\item
The choices on $N_P$, $K$, $M$, and $G$ need to coincide according to the Iwasawa decomposition.
\end{itemize}
Finally, 
\[
\Phi_M(\gamma_\infty, \xi) = \begin{cases}
\lf| \f{D^G(\gamma_\infty)}{D^M(\gamma_\infty)}\ri|^{1/2} \sum_{\pi \in \Pi^G_\disc(\xi)} \Theta_\pi(\gamma^\infty) & \gamma_\infty \text{ in an elliptic torus of $M$}\\
0 & \text{else}.
\end{cases}
\]
As written, this is only defined on regular elements, but Arthur proves it extends to a function that is continuous on every elliptic torus. 

As some notes for using this:
\begin{itemize}
\item
Comparing character formulas computes that $\Phi_G(\gamma_\infty,\xi) = \tr \xi(\gamma_\infty)$ where $\xi$ is overloaded to also denote the finite dimensional representation with highest weight $\xi$.
\item
If $M \neq G$, $\Phi_M$ cannot be evaluated through the standard Harish-Chandra character formula since it involves $\Theta_\pi$'s evaluated on tori that are not elliptic in $G$. See \cite[\S4]{Art89} for an algorithm to actually do so.
\item
The only $M$ that contribute to the outer sum are those in $L^\cusp$; in this case, those that are cuspidal over $\Q$. Arthur's original paper implicitly showed this for $M$ cuspidal over $\R$. There is a small correction using lemma \ref{centertechnicality} that the formula in \cite{Art89}[thm. 5.1]  is zero for $M$ not cuspidal over $\Q$ (Arthur was surely aware of this but seems to have forgotten to mention it). Alternatively, \cite{GKM97} shows vanishing using different methods. See section \ref{cuspvanishing} for more details. 
\item
Because of the dimensions on $\eta_\xi$, both sides of this formula have dimension $[G^\infty]$. However, explicitly computing the $\chi(I^M_\gamma)$ terms still requires choosing Haar measures at $\infty$. 
\end{itemize}

\subsection{Trace Formula with Central Character}
Stabilization requires a slightly different version of the trace formula where the fixed character $\chi$ is on a larger closed subgroup of $Z(\A)$. There is a full theory in \cite{Art02} that takes quite a bit of work to describe. We summarize the relevant parts here.
\begin{dfn}
A central character datum on $G$ is $(\mf X, \chi)$ where 
\begin{itemize}
\item
$\mf X \supseteq \A_{G, \infty}$ is closed inside $Z(\A)$ such that $Z(F) \mf X$ is also a closed subgroup. 
\item
$\chi : \mf X \cap Z(F) \bs \mf X \to \C^\times$ is a continuous character.
\end{itemize} 
Furthermore, $\ms H(G,(\mf X, \chi)) = \ms H(G, \chi)$ is the set of smooth functions $f$ on $G(\A)$ such that $f(gx) = \chi^{-1}(x) f(g)$ and $f$ is compactly supported mod $\mf X$. 
\end{dfn}

\begin{note}
For our purposes here, it suffices to consider $\mf X$ that are the product of the adelic points of some algebraic subtorus of $Z$ multiplied by some abstract subgroup of $Z_{G_\infty}(\R)$. 
\end{note}
Fix central character data $(\mf X, \chi)$. In \cite[\S3]{Art13}, Arthur defines $I_{\disc, t, \chi}$ as a distribution on $\ms H(G, \chi)$:
\begin{equation}\label{Idisc}
I_{\disc, t, \chi}(f) = \sum_{M \in \ms L} \f{|\Om_M|}{|\Om_G|} \sum_{w \in W(M)_\reg} |\det(w-1)|_{\mf a^G_M}|^{-1} \tr(M_{P,t}(\om, \chi)\mc I_{P,t}(\chi, f)).
\end{equation}
This is a generalization of $I_{\disc, t}$ and most of the terms are the same. The relevant part is how $\mc I_{P,t}$ changes. First, $\chi$ induces a character on $A_{M, \rat}\mf X$ by pullback and therefore lets us define $L^2_\disc(M(\Q) \bs M(\A), \chi)$ analogous to other $L^2$ spaces with character: as the discrete part of $\chi^{-1}$-invariant, $L^2$-up-to-$\mf X$ functions on $M(\Q) \bs M(\A)$ as an $M(\A)$-representation. Then, $\mc I_{P,t}(\chi, f)$ can be defined analogously to $\mc I_{P,t}$ from the trace formula without central character. Decompositions and traces making sense in this context requires some extra work summarized on \cite[pg 123]{Art13}. The dimensions change to $[G(\A)][\mf X]^{-1}$. 

For our work here, we only need to worry about the spectral side so we will not mention the geometric version. 

\section{Endoscopy and Stabilization Background}
The standard reference for this material, \cite{KS99}, is written for the more general case of twisted endoscopy. It is therefore easier to follow the summary in \cite[\S1.3]{Kal16}. The simpler summary in \cite[\S2]{Shi10} for the simply connected derived subgroup case is also helpful. Finally, \cite{Lab11} is a course-notes style writeup of this material and therefore more motivated albeit far less general.

For this section, allow $F$ to be a local or global number field. 
\subsection{Endoscopic Groups}
\subsubsection{Endoscopic quadruples}\label{quadruples}
\begin{dfn}[{\cite[pg 18]{KS99}}] 
An \emph{endoscopic quadruple} for $G$ is a tuple $(H,\mc H, s,\eta)$ with
\begin{itemize}
\item
$H$ a quasisplit connected reductive group over $F$,
\item
$\mc H$ is a split extension of $\wh H$ by $W_F$ such that action of $W_F$ on $\wh H$ determined by the splitting matches the one coming from $H$ (in $\Out(\wh H)$),
\item
$s \in Z_{\wh H}$ and semisimple in $\wh G$,
\item
$\eta : \mc H \to \Ld G$ an $L$-embedding
\end{itemize}
such that
\begin{enumerate}
\item
$\eta$ restricts to an isomorphism $\wh H \iso \wh G^0_{\eta(s)}$.
\item
There is then a $W_F$-equivariant sequence 
\[
1 \to Z_{\wh G} \to Z_{\wh H} \to Z_{\wh H}/Z_{\wh G} \to 0 
\]
which induces a map $(Z_{\wh H}/Z_{\wh G})^{W_F} \to H^1(F, Z_{\wh G})$. We require that $s$ maps to something locally trivial under this.
\pause
\end{enumerate}
It is furthermore \emph{elliptic} if
\begin{enumerate}
\resume
\item
$(Z_{\wh H}^{W_F})^0 \subseteq Z_{\wh G}$. 
\end{enumerate}
For future reference, we let $\mf K(s, \eta)$ be the elements that map to something locally trivial under $(Z_{\wh H}/Z_{\wh G})^{W_F} \to H^1(F, Z_{\wh G})$. 
\end{dfn}

\begin{dfn}
Two endoscopic quadruples $(H, \mc H, s, \eta), (H', \mc H', s', \eta')$ are isomorphic if there is an element $g \in \wh G$ such that 
\begin{enumerate}
\item
$\eta(\mc H)$ and $\eta'(\mc H')$ are conjugate by $g$,
\item
$s$ and $g s g^{-1}$ are equal in $Z_{\wh H}/Z_{\wh G}$. 
\end{enumerate}
Call the set of isomorphism classes of elliptic endoscopic quadruples $\mc E_\el(G)$. 
\end{dfn}
Note that the definition implicitly uses a fact which we state directly here to cite more easily later:
\begin{lem}
Let $G$ be a reductive group over global or local field $K$ and $(H, \mc H, \eta, s)$ an elliptic endoscopic quadruple. Then there is a map $Z_G \into Z_H$. 
\end{lem}

\begin{proof}
See \cite{KS99} pg. 53.
\end{proof}

\subsubsection{Endoscopic pairs}
Endoscopic quadruples actually contain a lot of redundant data. A more basic and easier to think about notion is the endoscopic pair defined in \cite[\S7]{Kot84}:
\begin{dfn}
An \emph{endoscopic pair} for group $G$ is $(s, \rho)$ where
\begin{itemize}
\item
$s$ is a semisimple element of $\wh G/Z_{\wh G}$,
\item
$\rho$ is a map $W_F \to \Out(\wh H)$ where $\wh H = \wh G^0_s$
\end{itemize}
satisfying 
\begin{enumerate}
\item
$\rho(\sigma)$ for $\sigma \in W_F$ is conjugation by an element in the normalizer of $\wh H$ in $\Ld G$ that projects to $\sigma$.
\item
Then, $\rho$ induces a $W_F$-action on $Z_{\wh G^0_s}$ which fits into $W_F$-equivariant sequence 
\[
1 \to Z_{\wh G} \to Z_{\wh H} \to Z_{\wh H}/Z_{\wh G} \to 0 
\]
which induces a map $(Z_{\wh H}/Z_{\wh G})^{W_F} \to H^1(F, Z_{\wh G})$. We require that $s$ maps to something locally trivial under this.
\pause
\end{enumerate}
It is furthermore \emph{elliptic} if 
\begin{enumerate}
\resume
\item
$(Z_{\wh H}^{W_F})^0 \subseteq  Z_{\wh G}$.  
\end{enumerate}
For future reference, we let $\mf K(s, \rho)$ be the elements that map to something locally trivial under $(Z_{\wh H}/Z_{\wh G})^{W_F} \to H^1(F, Z_{\wh G})$. 
\end{dfn}

The $\rho$ action can be further clarified: if $a \rtimes \gamma \in \Ld G$ and $(b,1) \in \wh G \subset \Ld G$, 
\begin{multline*}
(a \rtimes \gamma)(b \rtimes 1)(a \rtimes \gamma)^{-1} = (a \rtimes \gamma)(b \rtimes 1)(\gamma^{-1}(a^{-1}) \rtimes \gamma^{-1}) \\
= (a \gamma(b) \rtimes \gamma)(\gamma^{-1}(a^{-1}) \rtimes \gamma^{-1}) = (a \gamma(b) a^{-1} \rtimes1)
\end{multline*}
so if $\rho$ is part of an endoscopic pair, any $\rho(\gamma)$ is of the form $b \mapsto a_\gamma \gamma_{\wh G} (b) a^{-1}_\gamma$ for some $a_\gamma \in \wh G$ where the subscript $\wh G$ denotes that the $\gamma$ action is as it is on $\wh G$. The choices of $a_\gamma$ are unique up to 
\[
a_\gamma \in \Int \wh H \bs \wh G / Z_{\gamma_{\wh G} \wh H}(\wh G) =  {\wh H}_\adj \bs \wh G /  Z_{\gamma_{\wh G} \wh H}=  {\wh H}_\adj  \bs \wh G / Z_{\wh H}= \wh H \bs \wh G
\]
since $\gamma_{\wh G} \wh H$ is the centralizer of $\gamma_{\wh G} s$.
 
 \begin{dfn}
 An isomorphism of endoscopic pairs $(s, \rho)$ and $(s', \rho')$ is an element $g \in \wh G$ such that
 \begin{itemize}
 \item
 $\wh G^0_s, \wh G^0_{s'}$ and $\rho, \rho'$ are $g$-conjugate,
 \item
 $s,s'$ have the same image in $\mf K(s, \rho)$. 
 \end{itemize}
 \end{dfn}
 
 As explained in \cite[pg 630-631]{Kot84}, $\rho$ determines a quasisplit group $H$ from $\wh H$ and therefore the $(H, s, \eta)$ part of an endoscopic quadruple. Given $H$ and $G$, we can define $\mc H$ as follows: $\wh H$ embeds into both $\Ld H$ and $\Ld G$. Let $\mc H$ be the set of $x \in \Ld G$ such that there exists $y \in \Ld H$ such that conjugation by $x,y$ are the same on $\wh H$ and $x,y$ project to the same element of $W_F$. In terms of  the $a_\gamma$ from above, we can realize
 \[
 \mc H = \bigcup_{\gamma \in W_F} \wh H a_\gamma \rtimes \gamma
 \] 
 where we can choose representatives for $a_\gamma$ so that conjugation by $a_\gamma \rtimes \gamma$ fixes a pinning of $\wh H$. Isomorphisms are also the same on each side, so in summary:
 \begin{lem}[{\cite[\S7]{Kot84}}]
 The set of elliptic endoscopic pairs of $G$ up to isomorphism are in bijection with $\mc E_\el(G)$ where the bijection is as described above.
 \end{lem}

 \subsubsection{Motivation and the group $\mf K$}
There are two motivations for this definition, either spectral or geometric. We briefly and very roughly describe the geometric explanation since it is somewhat relevant later. We ignore many, many Galois cohomology details. In increasing generality and detail, more information can be found in \cite[\S III.3]{Lab11}, \cite[\S9]{Kot86}, and \cite[\S6-7]{KS99}. 

Let semisimple $\gamma \in G(F_v)$ be contained in maximal torus $T$.  If $\gamma$ is strongly regular, then we can write its stable orbit as $(T \bs G)(F_v)$ and its orbit as $T(F_v) \bs G(F_v)$. Therefore, the fibers of the map from $(T \bs G)(F_v)$ onto 
\[
\mf D(F_v, T \bs G) = \ker(H^1(F_v, T) \to H^1(F_v, G))
\]
are exactly the unstable conjugacy classes making up $(T \bs G)(F_v)$.  Let
\[
\mf E(F_v, T \bs G) = \ker(H^1(F_v, T) \to H^1_\ab(F_v, G))
\]
be the abelian group version of this and
\[
\mf K(F_v, T \bs G) = \mf E(F_v, T \bs G)^\vee. 
\]
Elements $\kappa \in \mf K$ are called endoscopic characters. 

If $v$ is a place of $F$ and $\kappa \in \mf K(F_v, T \bs G)$, this allows the definition of twisted orbital integrals
\[
O^\kappa_\gamma(f) = \int_{(T \bs G)(F_v)} \kappa(g) f(g^{-1} \gamma g) \, dg
\]
using the map $(T \bs G)(F) \to \mf E(F, T \bs G)$. 

We can also define adelic versions of these groups $\mf D(\A, T \bs G), \mf E(\A, T \bs G)$, and $\mf K(\A, T \bs G)$ using corresponding cohomology groups $H^1(\A, \cdot)$.  If $\gamma \in G(\A)$ is strongly regular, $\mf D(\A, T \bs G)$ parametrizes the $\gamma'$ that have every component stably conjugate to $\gamma$. It is a restricted direct product of the $\mf D(F_v, T \bs G)$ by $\mf D(\mc O_v, T \bs G)$ which happens to be trivial. Define a measure on it by taking the product of the counting measures on $\mf D(F_v, T \bs G)$. Then for $\kappa \in \mf K(\A, T\bs G)$ we can define global twisted orbital integral
\[
O^\kappa_\gamma(f)  = \sum_{e \in \mf D(\A, T \bs G)} \kappa(\obs(\gamma_e)) O_{\gamma_e}(f)
\] 
where $\gamma_e$ is the conjugacy class corresponding to $e$ with base point $\gamma$ and $\obs$ is the obstruction defined in \cite{Kot86} and \cite{KS99}. 

Stabilization of the trace formula first produces sums of $O^\kappa_\gamma(f)$'s over triples of these $(T, \gamma, \kappa)$ over $F$. The result (\cite[lem 7.2.A]{KS99}) shows that such triples are in bijection with quintuples $(H, \mc H, s, \eta, \gamma_H)$: endoscopic quadruples with a choice of strongly regular element $\gamma_H \in H$ up to appropriately defined equivalence. Through this equivalence, the group $\mf K$ for $T$ ends up being the same as the group $\mf K$ defined above for $(s, \eta)$ (see \cite[pg 105-106]{KS99}).

\subsection{$z$-Extensions}
Our next goal is to define transfers of functions. This na\"ively needs an embedding $\Ld H \into \Ld G$, but in general $\Ld H \not\cong \mc H$ so we do not have one. There are two possible strategies for dealing with this: the original in \cite{LS87} is to take a nice enough central extension of $G$. This works for the standard endoscopy described here but not for the more general twisted endoscopy, so more modern sources prefer to take central extensions of $H$ as described in \cite{KS99}. As we will remark after proposition \ref{zextpath}, these methods are more or less interchangeable in the standard endoscopy case. 

We describe the second method in detail:
\begin{dfn}
A \emph{$z$-pair} $(\td H, \td \eta)$ for endoscopic quadruple $(H, \mc H, s, \eta)$ is an extension $\td H$ by a central induced torus such that
\begin{enumerate}
\item
$\td H_\der$ is simply connected (we call such an $\td H$ a \emph{$z$-extension}). 
\item
$\td \eta : \mc H \to \Ld{\td H}$ is an $L$-embedding that restricts to the map $\wh H \to \wh {\td H}$ dual to the projection $\td H \to H$. 
\end{enumerate}
\end{dfn}
By Lemma 2.2.A in \cite{KS99}, as long as (1) is satisfied, a valid $\eta $ satisfying (2) always exists.
\begin{lem}\label{zextbound}
Let $H$ be a reductive group that splits over $K'$. Then there exists a $z$-extension of $H$ splitting over $K'$. Furthermore, the dimension of the extending torus is bounded by $[K' : \Q](\rank_\ssm H)$. 
\end{lem}

\begin{proof}
We just go through the construction in \cite{Lan79} or \cite[pg 299]{MS82} explicitly seeing how big things get at each step. Let  $T^\scn$ be the maximal torus in the simply connected cover of $H^\der$. Let $P = X_*(T)/X_*(T^\scn)$ as a Galois module. A $z$-extension would correspond to an extension of $X_*(T)$ making this quotient have no torsion. The torsion part has less than $\rank_\ssm G$ generators.

The argument starts with a lemma writing $P$ as a quotient of Galois modules
\[
0 \to M \to Q \to P \to 0
\]
with $M$ free over $Z[G]$ and $Q$ free over $P$. The construction is \cite[prop 3.1]{MS82} and bounds $\rank_\Z M$ by $\dim K'$ times the number of generators of the torsion of $P/\Z$ which we can further bound by $(\dim K')(\rank_\ssm H)$. 

Some work with reductive groups shows that $M$ can be chosen to be the cocharacter space of the extending torus, thereby finishing the argument.
\end{proof}

In the case where $G$ has simply connected derived subgroup, the $Z$-extension can be chosen to be trivial and $\mc H \simeq \Ld H$. In this case, an endoscopic triple $(H, s, \eta)$ contains all the needed data.

\subsubsection{$z$-extensions and central character datum}
If $(\mf X, \chi)$ is a central character datum for $G$, any $(H, \mc H, s, \eta)$ and $(\td H, \td \eta)$ quadruple and extension determine a central character datum $(\mf X_{\td H}, \chi_{\td H})$ on $\td H$. The central subgroup $\mf X_{\td H}$ is produced from $\mf X$ by first taking the image under the map $Z_G \into Z_H$ and then taking the preimage under $\td H \to H$.

To get $\chi_{\td H}$, pick a section $c$ for $\mc H \to W_F$. Then if $T$ is the extending torus defining $\td H$, the composition
\[
W_F \xto c \mc H \xto{\td \eta} \Ld{\td H} \to \Ld T
\]
is an $L$-parameter for $T$. This determines a character $\lb_{\td \eta}^{-1}$ on $T(F)$ if $F$ is local or $T(F)\bs T(\A)$ if $F$ is global through the Langlands correspondence for Tori. The inverse is to match our convention for defining Hecke algebras. 

Through considerations of transfer factors (see section \ref{transfer}), $\lb_{\td \eta}$ can be extended to the preimage of $Z_G$ in $Z_{\td H}$. Therefore, we can set $\chi_{\td H}$ to be $\chi \lb_{\td \eta}$ (where $\chi$ is defined on $\mf X_H$ by pullback). We will discuss this and more properties of $\lb_{\td \eta}$ when we discuss transfer. In particular, we will show that in the relevant cases, $\lb_{\td \eta,v}$ at a place $v$ can be extended to a character on $\td H_v$. 

\subsubsection{$z$-extensions do not change much}
There is a vague intuition that taking a $z$ extension should not change a groups endoscopy:
\begin{prop}\label{zextpath}
Let $G$ be a group over $F$. 
\begin{enumerate}[(a)]
\item
If $G_1$ is a central extension of $G$ by induced torus $T$, then the (elliptic) endoscopic tuples for $G$ are in bijection with those of $G_1$. This bijection takes a group $H$ to a central extension $H_1$ by $T$. 
\item
If $H$ is an (elliptic) endoscopic group of $G$ and $H_1$ is a central extension of $H$ by induced torus $T$, then there is a central extension $G_1$ of $G$ by $T$ such that $H_1$ is an (elliptic) endoscopic group of $G_1$. Furthermore, the endoscopic tuples determining $H$ and $H_1$ correspond under the bijection from (a). 
\end{enumerate}
\end{prop}

\begin{proof}
\noindent\underline{Part (a):}

\noindent\underline{The $s$:} The map $\wh G \to \wh G_1$ gives a canonical $W_F$-equivariant isomorphism $\wh G/Z_{\wh G} \to  \wh G_1/Z_{\wh G_1}$ so choices for $s$ are the same. Given such an $s$, set $\wh H_1 = (\wh G_1)_s^0$. Then we have the diagram
\[
\begin{tikzcd}
\wh H \arrow[r, hook] \arrow[d, hook] & \wh G \arrow[d, hook] \\
\wh H_1 \arrow[r, hook] \arrow[d, two heads] & \wh G_1 \arrow[d, two heads] \\
\wh T \arrow[r, "\sim"] & \wh T
\end{tikzcd}
\]

\noindent\underline{The $\rho$ and H:} This gives a canonical isomorphism $\wh H_1 \bs \wh G_1 \to \wh H \bs \wh G$ so assignments $\gamma \to a_\gamma$ as in the comment after the definition of endoscopic pair are the same for $G$ and $G_1$. There are two conditions for this assignment to give a valid $\rho$: The first is that $\gamma \mapsto \Int a_\gamma \circ \gamma$ is a homomorphism up to $\Int \wh H = \Int \wh H_1$. This condition is clearly the same with respect to either $\wh H$ or $\wh H_1$.

The second condition is that $\Int a_\gamma \circ \gamma$ needs to fix the appropriate group: $\wh H$ or $\wh H_1$. By construction, $\wh H = \wh G \cap \wh H_1$. Therefore, since $\wh G$ is $W_F$ and $\Int \wh G_1$-invariant, if such a map fixes $\wh H_1$, it fixes $\wh H$. For the other direction, since these are all complex groups and $\wh G \supseteq (\wh G_1)^\der$, all elements of $\wh G_1$ can be written as $zg$ for $z \in Z_{\wh G_1}^0$ and $g \in \wh G$. This is an element of $\wh H_1$ if and only if $g \in \wh H$. In total, $\wh H_1 = Z^0_{\wh G_1} \wh H$ so we are done since $Z^0_{\wh G_1}$ is fixed by $W_F$ and $\Int_G$. Therefore, this second condition is true for $\wh G$ if and only if it is true for $\wh G_1$. 

Note that for any such $\rho$, the columns of the above diagram and the isomorphism between $\wh T$'s are $\Gamma$-equivariant. Undoing this dual, this will give that $H_1$ is an extension of $H$ by $T$.

\noindent\underline{The cohomology condition:} In total, the possible pairs $(s, \rho)$ ignoring the cohomology condition are the same for $G$ and $G_1$. It remains to show that the cohomology condition holds with respect to $G$ if and only if it does for $G_1$. We have $W_F$-equivariant diagram where the first two rows are exact sequences (note that the actions on $Z_{\wh G}$ from $\rho$ and $\wh G$ coincide so the action here is according to $\rho$):
\[
\begin{tikzcd}
1 \arrow[r] & Z_{\wh G} \arrow[r]\arrow[d] & Z_{\wh H} \arrow[r]\arrow[d] & Z_{\wh H}/Z_{\wh G} \arrow[r]\arrow[d, "\sim"] & 1 \\
1 \arrow[r] & Z_{\wh G_1} \arrow[r]\arrow[d] & Z_{\wh H_1} \arrow[r]\arrow[d] & Z_{\wh H_1}/Z_{\wh G_1} \arrow[r] & 1 \\
& \wh T \arrow[r, "\sim"] & \wh T &&
\end{tikzcd}
\]
This gives a corresponding diagram in cohomology:
\[
\begin{tikzcd}
(Z_{\wh H}/Z_{\wh G})^\Gamma \arrow[r, "\varphi_1"] \arrow[d, "\sim"] & H^1(\Gamma, Z_{\wh G}) \arrow[d, "\psi"] \\
(Z_{\wh H_1}/Z_{\wh G_1})^\Gamma \arrow[r, "\varphi_2"] & H^1(\Gamma, Z_{\wh G_1})
\end{tikzcd}
\]
Here $\Gamma \subseteq W_F$ is some local Galois group. The cohomology conditions for $H$ and $H_1$ matching at $\Gamma$ is equivalent to $\ker \varphi_1 = \ker \varphi_2$. To show this, consider the sequence
\[
\pi_0(\wh T^\Gamma) \to  H^1(\Gamma, Z_{\wh G}) \xto \psi H^1(\Gamma, Z_{\wh G_1}) \to H^1(\Gamma, \wh T).
\]
Since $T$ is an induced torus, $\wh T$ is a power of $\Gm$ with a $\Gamma$ action by permuting coordinates. This gives first that $\wh T^\Gamma$ is connected and second that $\wh T$ is induced so $H^1(\Gamma, \wh T) = 0$. Therefore $\psi$ is an isomorphism and the cohomology conditions are equivalent at every place.

\noindent\underline{Ellipticity:} The elliptic condition is that  $(Z_{\wh H_*}^{W_F})^0 \subseteq Z_{\wh G_*}^{W_F}$. As before, $Z_{\wh H_1} = Z_{\wh H} Z_{\wh G_1}$ and $Z_{\wh H} \cap Z_{\wh G_1} = Z_{\wh G}$. Then we get the sequence
\[
1 \to Z_{\wh G} \to Z_{\wh H} \times Z_{\wh G_1} \to Z_{\wh H_1} \to 1
\]
where the first map is the antidiagonal. This gives a map in cohomology:
\[
Z_{\wh H}^{W_F} \times Z_{\wh G_1}^{W_F} \to Z_{\wh H_1}^{W_F} \to H^1({W_F}, Z_{\wh G})\to H^1({W_F}, Z_{\wh G_1}) \oplus H^1({W_F}, Z_{\wh H}).
\]
From previous arguments, $T$ being induced gives that the last map in injective into the first coordinate. Therefore the middle is $0$ and the first is surjective. Therefore $Z_{\wh H_1}^{W_F} = Z_{\wh H}^{W_F} \times Z_{\wh G_1}^{W_F}/Z_{\wh G}^{W_F}$ and $(Z_{\wh H_1}^{W_F})^0 \subseteq (Z_{\wh G_1}^{W_F})^0(Z_{\wh H}^{W_F})^0 Z_{\wh G}^{W_F}$. This gives that the elliptic condition on $H$ implies that on $H_1$. 

For the other direction, $Z_{\wh H} = Z_{\wh H_1} \cap \wh G$ gives that $Z_{\wh H}^{W_F} = Z_{\wh H_1}^{W_F} \cap \wh G^{W_F}$ which gives $(Z_{\wh H}^{W_F})^0 \subseteq (Z_{\wh H_1}^{W_F})^0 \cap \wh G^{W_F}$. Assuming $(Z_{\wh H_1}^{W_F})^0 \subseteq Z_{\wh G_1}^{W_F}$ and further using that $Z_{\wh G_1} \cap \wh G = Z_{\wh G}$ implies $Z_{\wh G_1}^{W_F} \cap \wh G^{W_F} = Z_{\wh G}^{W_F}$ finally giving that $(Z_{\wh H}^{W_F})^0 \subseteq Z_{\wh G}^{W_F}$.

\noindent\underline{Part (b):}

We are given $G$, endoscopic group $H$, and extension $H_1$ by $T$. There is a map $Z_G \into Z_H$ (see \cite{KS99} pg. 53) so we can pullback the extension $Z_{H_1}$ to an extension $Z_{G_1}$ of $Z_G$ by $T$. 

Set $G_1 = Z_{G_1} \times G_\der/Z_{G^\der}$ as an algebraic group where the $Z_{G^\der}$ is embedded antidiagonally. Then, since $G = Z_G \times G_\der/Z_{G^\der}$, $G_1$ is an extension of $G$ by $T$. If $H$ comes from data $(s, \rho)$, then through the construction of the bijection in  (a), $(s,\rho)$ gives data for $H_1$ and is elliptic if and only if $(s,\rho)$ is.
\end{proof}

Consider $H$ an endoscopic group of $G$ and $H_1$ a $z$-extension (so it has simply connected derived subgroup). Let $(H_1, \mc H_1, s, \eta)$ be the quadruple for $G_1$ produced by part (b). Then the map $\Ld H_1 \to \mc H_1$ is an isomorphism, so we actually do have an embedding $\Ld H_1 \into \Ld G_1$. This is the $z$-extension construction described in \cite{LS87}.

\subsection{Transfer}\label{transfer}
Consider quadruple $(H, \mc H, s, \eta)$ for $G$ over local or global $K$ and associated $z$-extension $(H_1, \eta_1)$. There is a transfer map
\begin{multline*}
\mc T : \{\text{strongly $G$-regular semisimple conjugacy classes in $H(K)$}\} \\
\to \{\text{strongly regular stable conjugacy classes $G(K)$}\} \cup \{*\}
\end{multline*}
where the $*$ is a dummy variable to allow maps that are not necessarily defined everywhere. We say that $\gamma_H \in H(K)$ is a norm of $\gamma_G \in G(K)$ if $\mc T$ takes the conjugacy class of $\gamma_H$ to that of $\gamma_G$. Respectively, $\gamma_{H_1} \in H_1(K)$ is a norm of something if its projection to $H(F)$ is.

\subsubsection{Local Transfer}
Now, consider local $F_v$. If strongly $G$-regular $\gamma_{H_1}$ is a norm of strongly regular $\gamma_G$, a transfer factor $\Delta(\gamma_{H_1}, \gamma_G) = \Delta^{H_1}_G(\gamma_{H_1}, \gamma_G)$ can be defined (this is the content of sections $4.1-5.1$ in \cite{KS99}). The factor is non-canonical up to a uniform constant. We recall some useful properties from \cite[\S5.1]{KS99}:
\begin{itemize}
\item
$\Delta(\gamma_{H_1}, \gamma_G)$ is $0$ unless $\gamma_{H_1}$ is a norm of $\gamma_G$. 
\item
$\Delta(\gamma_{H_1}, \gamma_G)$ is constant over the stable conjugacy class of $\gamma_{H_1}$.
\item
Let $Z_{G_1} = Z_{H_1} \times_{Z_H} Z_G$. There exists a character $\lb_{\eta_1}$ on $Z_{G_1}(F_v)$ such that if $(z_1, z) \in Z_{G_1}(F_v)$,
\[
\Delta^{H_1}_G(z_1 \gamma_{H_1}, z \gamma_G) = \lb_{\eta_1}^{-1}(z_1,z) \Delta^{H_1}_G(\gamma_{H_1}, \gamma_G).
\]
In fact, $\lb_{\eta_1}$ even extends to a character on $G_1(F_v)$ (see the construction on pg. 53 in  \cite{KS99} or pg. 55 in \cite{LS87}).
\item
Let the quadruple $(H, \mc H, s, \eta, \gamma_{H_1})$ correspond to the triple $(T, \gamma_G, \kappa)$. Then $\gamma_{H_1}$ is a norm of $\gamma_G$. If $\gamma'_G$ is a stable conjugate of $\gamma_G$, 
\[
\kappa(\gamma'_G) \Delta(\gamma_{H_1}, \gamma_G) = \Delta(\gamma_{H_1}, \gamma'_G).
\]
\end{itemize}

Fix central character datum $(\mf X, \chi)$ for $G$. Let $f \in \ms H(G(F_v), \chi_v)$. We say that $f^{H_1} \in \ms H(H_1(F_v), \chi_{H_1,v})$ matches $f$ if 
\[
SO_{\gamma_{H_1}}(f^{H_1}) = \sum_{\gamma_G} \Delta(\gamma_{H_1}, \gamma_G) O_{\gamma_G}(f)
\]
for all strongly $G$-regular $\gamma_{H_1} \in H_1(F_v)$ where $\gamma_G$ ranges over representatives of unstable conjugacy classes such that $\gamma_{H_1}$ is a norm of $\gamma_G$. Note that the right-hand side is a twisted orbital integral multiplied by an appropriate constant. 

Since $\gamma_{H_1}$ and $\gamma_G$ are strongly regular, if $T$ is a maximal torus for $G$ and $Z$ is the extending torus defining $H_1$ from $H$, the orbital integrals have dimension $[G(F_v)][T(F_v)]^{-1}$ and $[H_1(F_v)][T(F_v)]^{-1}[Z(F_v)]^{-1}  = [H(F_v)][T(F_v)]^{-1}$. Therefore, $f^{H_1}$ needs to have dimensions $[G(F_v)][H(F_v)]^{-1}$. 

A big theorem is that such an $f^H$ always exists. The Archimedean case is from Shelstad in \cite{She82} while the non-Archimedean case was reduced to the fundamental lemma (which will be discussed later) by Waldspurger in \cite{Wal97}. Call such an $f^H$ a transfer of $f$.

\subsubsection{Global Transfer} 
If $F$ is global, then the endoscopic datum determine local endoscopic datum at each place $v$. This lets us define a global transfer factor $\Delta_\A(\gamma_{H_1}, \delta_G)$ as the product of all the local transfer factors. \cite[cor 7.3.B]{KS99} gives that all the choices defining the local factors can be made consistently giving a canonical choice of global factor. 

If $f \in \ms H(G, \chi)$ factors into local factors at each place, then transferring each of the local factors gives a transfer $f^H$ satisfying a similar identity. By the fundamental lemma, this is unramified almost everywhere and is therefore an element of $\ms H(H_1, \chi_{H_1})$. 

After lots of cohomology work, $f^H$ can be shown to satisfy a global identity
\[
SO_{\gamma_{H_1}}(f^H) = O^\kappa_{\gamma_G}(f)
\]
where $(H, \mc H, s, \eta, \gamma_H)$ corresponds to $(T, \gamma_G, \kappa)$. This is \cite[lem 7.3.C]{KS99}.

\subsubsection{Characters from Transfer}\label{transferchar}
By the above, endoscopy always defines a character on $Z_{H_1}(F_v)$. However, for $v$ non-Archimedean, this actually extends to a character on $H_1(F_v)$. We will need this to state some bounds on non-Archimedean transfers later. 

Fix such a $v$ and assume without loss of generality that $G$ has simply connected derived subgroup (possibly by taking a $z$-extension and using proposition \ref{zextpath}). Take the extension $G_1$ of $G$ as in proposition \ref{zextpath}(b). Then $G_1^\der$ is an isogenous cover of $G^\der$, so the two are equal. The map $\eta$ determines a character $\lb_{\eta_1}$ on $Z_{G_1}(F_v) = Z_{H_1}(F_v) \times_{Z_H(F_v)} Z_G(F_v)$. Since this lifts to a character on $G_1(F_v)$, it is actually a character on $G_1(F_v)/G_1^\der(F_v)$. If $F$ is local then $H^1(F_v, G_1^\der) = 0$ since $G_1^\der$ is semisimple and simply connected. Therefore this is a character on $(G_1)_\ab(F_v)$ so let it correspond to the $L$-parameter $\alpha: W_{F_v} \into \Ld(G_1)_\ab$.

Next
\begin{lem}
Let $G$ be a reductive group over $F_v$. Then $Z_{\wh G}^0 = \wh{G_{\ab}}$ as groups with $W_{F_v}$-action. 
\end{lem} 

\begin{proof}
Let $G$ have maximal torus $T$. As $W_F$-modules, $X_*(\wh{G_\ab})= X^*(G_\ab) = X^*(T)^\Om$ and $X_*(Z_{\wh G}^0) = X_*(\wh T)^\Om = X^*(T)^\Om$. This equality of cocharacters induces an equality of torii.  
\end{proof}
Since $\wh H_1$ is a connected centralizer in $\wh G_1$, we get a map $Z_{\wh G_1}^0 \into Z^0_{\wh H_1}$. Since $H_1$ is endoscopic, the map is Galois-equivariant so it extends to a map $\Ld (G_{1,\ab}) \to \Ld (H_{1,\ab})$. Therefore $\alpha$ can be pushed forward and determines a character $\lb'_{H_1}$ on $H_1$. 

Note that $\lb_{H_1}$ and $\lb'_{H_1}$ are equal on $Z_{G_1}(F_v)$ since they correspond to the same parameter of $Z_{G_1}(F_v)$. This common value is the character $\lb_{\eta_1}$ from before that determined which Hecke algebra transfers landed in. The discussion here simply shows that it extends to a character on $H_v$. 

\subsubsection{A trick for computing transfers with $z$-extensions}\label{zextransfertrick}
Most formulas for transfers in the literature only apply in the case when $\Ld H \cong \mc H$. To use these in the general case, consider the same quadruple and $z$-extension as before with $T \into H_1$ as the extending torus. Proposition \ref{zextpath}(b) lets us find $G_1$ such that $(H_1, \mc H_1, s, \eta)$ is an endoscopic quadruple for $G_1$ with $\Ld H_1 \cong \mc H_1$. Let $\pi : G_1 \to G$ be the projection. 

The key property we use is that 
\[
\Delta^{H_1}_{G_1}(\gamma_1, \delta_1) = \Delta^{H_1}_G(\gamma_1, \delta)
\]
whenever $\delta_1 \in G_1(F)$ projects to $\delta \in G(F)$ and $\gamma_1$ is a norm of $\delta_1$ (see \cite{LS87} pg. 55). Therefore, given $f \in \ms H(G(F), \chi)$, let 
\[
f_1(g) = f \circ \pi(g)
\]
for $g \in G_1(F)$. If $f_1$ and ${f_1}^{H_1}$ match, then for all appropriate $\gamma_1, \delta$,
\[
SO_{\gamma_1}(f_1^{H_1}) = \sum_{\delta_1} \Delta^{H_1}_{G_1}(\gamma_1, \delta_1) O_{\delta_1}(f_1) = \sum_{\delta_1} \Delta^{H_1}_G(\gamma_1, \pi(\delta_1)) O_{\pi(\delta_1)}(f),
\]
which is the condition for $f$ and ${f_1}^{H_1}$ matching. Therefore we can compute $f^{H_1}$ by transferring $f_1$. 

As a sanity check, note that $\gamma_1$ being a norm of $\delta_1$ is true if and only if $z\gamma_1$ is a norm of $z\delta_1$ for all $z \in Z_{G_1}$. In particular, if $x= (z_1,z) \in Z_{G_1}$ then
\begin{multline*}
\Delta^{H_1}_{G_1}(x\gamma_1, x\delta_1) = \Delta^{H_1}_{G_1}(z_1 \gamma_1, x \delta_1) = \Delta^{H_1}_G(z_1 \gamma_1, z \pi(\delta_1))\\
 = \lb_{\eta_1}(x)^{-1} \Delta^{H_1}_G(\gamma_1, \pi(\delta_1)) = \lb_{\eta_1}(x)^{-1} \Delta^{H_1}_{G_1}(\gamma_1, \delta_1).
\end{multline*}
Therefore, the transfer factor transforms appropriately so that this transfer will be in $\ms H(H_1(F), \chi_H)$. 

Beware that there is a small technical issue here. Theorems in the literature only give the existence of transfers of compactly supported functions. We get around this by finding a compactly supported function $f'$ that averages to $f \circ \phi$ along the central character datum (see lemma \ref{functiontruncation} for example) and then transferring $f'$. We then average $(f')^{H_1}$ against the central character datum.

\subsection{Stabilization}
Using all the above and with much work, $I_{\disc, t}^{G,\chi}(f)$ can be stabilized. In other words, it can be expanded as
\[
I_{\disc,t}^G(f)  = \sum_{H \in \mc E_\el(G)} \iota(G, H) \wh S_{\disc, t}^{\td H, \chi_{\td H}}(f^{\td H})
\]
for some choice of $z$-extensions. Here $\wh S^{\td H, \chi_{\td H}}_{\disc, t}$ is a stable distribution on $\ms H(\td H, \chi_{\td H})$ depending only on $t,\td H$. We will not use any properties of $S$ except that it is stable. There is no explicit construction of $f^H$ in general, so its known properties will be cited as needed.

The constant $\iota$ has an explicit formula. Recall the definition in section \ref{quadruples} of automorphisms of quadruples $(H, \mc H, s, \eta)$ by elements $g \in \wh G$. Let $\Lambda(H, \mc H, s, \eta)$ be the image of $\Aut(H, \mc H, s, \eta) \to \Out(\wh H)$. Then
\[
\iota(G,H) = |\Lambda(H, \mc H, s, \eta)|^{-1} \tau(G) \tau(H)^{-1}
\]
where $\tau$ is the Tamagawa number.

\subsection{Some Properties}
\subsubsection{Endoscopy and root data}
The following is a summary of the relation between roots data of endoscopic groups and the original group: 
\begin{lem}\label{roots}
Let $G$ be a reductive group over global or local field $K$, $(H, \mc H, \eta, s)$ an elliptic endoscopic quadruple and $(\td H, \td \eta)$ a $z$-extension. Let $T_H$ be a maximal torus for $H_{\overline K}$. Then there is a maximal torus $T$ of $G_{\overline K}$ and an isomorphism $T_H \to T$. The choice of $T$ and the map are unique up to $G_{\overline K}$-conjugacy. Let $T_{\td H}$ be the pullback of $T_H$ to $\td H$. 

Then the following also hold:
\begin{enumerate} 
\item
The positive (co)roots of $(H, T_{H})$ can be chosen to be a subset of those of $(G, T)$ through $T_H \to T$. 
\item
For any root of $\alpha$ of $(H,T_H)$, $s_\alpha \in \Om_H$ is the same as $s_\alpha \in \Om_G$ through the isomorphism $T_H \to T$.
\item
The positive roots of $(\td H, T_{\td H})$ can be chosen to be a subset of those of $(G,T)$ through $T_{\td H} \to T_H \to T$. 
\item
The Weyl action on the roots of $(\td H, T_{\td H})$ restricts to that on $(H, T_H)$ through $X^*(T_H) \into X^*(T_{\td H})$.  
\end{enumerate}
\end{lem}

\begin{proof}
The construction of $T_H$ and (1),(2) are done in \cite[\S3.1]{Kot86} and \cite[\S1.3]{LS87}. 

To deal with $\td H$, let the extension be $1 \to Z \to \td H \to H \to 1$. Every maximal torus of $\td H$ is the preimage of one of $H$ so $X^*(T_H)$ maps into the corresponding $X^*(T_{\td H})$. Since in the sequence
\[
0 \to \Lie Z \to \Lie \td H \to \Lie H \to 0,
\]
 $\Lie Z$ maps into the center, the roots of $\td H$ have to be the images of those of $H$. Choose a Borel $\td B$ containing $B_H$ to get containment of positive roots. The last statement on Weyl groups comes from $\Om_{(H, T_H)} \cong N_H(T)/Z_H(T)$. 
\end{proof}

Be careful that this lemma says nothing about the Galois actions on the roots. We will not need that information and getting it requires $G$ to be quasisplit. Also beware that this does not give that the simple roots of $H$ are a subset of the simple roots of $G$ or that the coroots of $\td H$ are a subset of the coroots of $G$.

\subsubsection{Real endoscopic characters}\label{realendoscopiccharacters}
As another computational tool, the character $\kappa = \kappa_{G,H}$ for elliptic elements has a nice form in the real case. If $G$ is a real group and $T$ is elliptic, there is an isomorphism
\[
\Om_{\C,G} / \Om_{\R,G} \to \mf D(\R, T \bs G).
\]
An endoscopic character $\kappa$ can therefore be extended to $\Om_\C(G)$. \cite[\S IV.1]{Lab11} gives that the extension is left-$\Om_{\C,H}$ invariant.  

In addition, the composition
\[
\Om(B_K) \to \Om_{\C, G} \to \Om_{\C, G} / \Om_{\R, G}
\]
is a bijection. 
This gives a bijection between any regular $\Pi_\disc(\xi)$ and $\mf D(\R, T \bs G)$ that depends on the choice of $B_K$. 

This interpretation of $\kappa$ will be used when computing transfers of pseudocoefficients.

\section{The Hyperendoscopy Formula}\label{Hformsection}
Here we will describe Ferrari's hyperendoscopy formula with some modifications in the case where groups without simply connected derived subgroup appear in the hyperendoscopic paths. Using this formula may appear a little bizarre since it may seem more reasonable to try to directly mimic the work of \cite{Art89} on the stable distributions $SO^H(f^H)$ like the main result of \cite{Pen19}. 

The advantage of using hyperendoscopy is that we can directly apply the work already done in \cite{ST16} instead of proving slightly different bounds for the slightly different terms appearing in the stable trace formula. There are two disadvantages: first, it gives worse constants in bounds, but the constants were already not explicit due to the model theory bounds that go into them. Second, hyperendoscopy requires extending Shin-Templier's results to groups with fixed central character datum, but this is interesting in its own right.  In addition, the hyperendsocopic formula itself may be a useful tool for studying future forms of the invariant trace formula that, unlike \cite{Pen19}, do not have a reasonable stabilization. 

\subsection{Raw Formula}
Recalling the key trick from \cite{Fer07}, rearrange the stabilized trace formula:
\[
\wh S_{\disc, t}^{G^\qs}(f^{G^\qs}) =  I_{\disc,t}^G(f) +  \sum_{\substack{H \in \mc E_\el(G) \\ H \neq G^\qs}} (-\iota(G, H)) \wh S_{\disc, t}^{\td H}(f^{\td H})
\]
where $G^\qs$ is the quasisplit form of $G$. We want to continue this expansion inductively to get a formula in terms of $I_\disc$ for the various groups. The result in \cite{Fer07} uses endoscopic triples, seemingly assuming that if a group has simply connected derived subgroup, then so do all its endoscopic groups. This is not true as there can be $\SO_{2k}$ factors in endoscopic groups of $\Sp_{2n}$ (see \cite[\S1.8]{Wal10}). Nevertheless, with a little more work, a formula more-or-less equivalent to Ferrari's can be derived.  

Inductively substituting in the expansions for $\wh S_{\disc, t}^{\td H}(f^{\td H})$ since the $\td H$ are all quasisplit gives something like
\[
\wh S_{\disc, t}^{G^\qs}(f^{G^\qs}) =  I_{\disc,t}^G(f) +  \sum_{\mc H \in \mc{HE}^0_\el(G)} \iota(G, \mc H) I_{\disc,t}^{H_{n_\mc H}}(f^{\mc H}).
\]

Because of the non-canonical $z$-extensions, the notation defining the indexing set becomes somewhat painful. We will later find a nicer set to index over.
\begin{dfn}
A consistent choice of length-$1$ raw endoscopic paths for $G$ is a set $\mc{HE}^0_\el(G)_1$ consisting of pairs $(H, z)$ where $H$ ranges over the proper isomorphism classes in $\mc E_\el(G)$ and $z$ is a choice of $z$-pair for $H$. 

Given a consistent choice of length-$(n-1)$ raw hyperendoscopic paths $\mc{HE}^0_\el(G)_{n-1}$, a consistent choice of length-$n$ hyperendoscopic paths is a set $\mc{HE}^0_\el(G)_n$ consisting of tuples $(\mc H, H, z)$ where $\mc H \in \mc{HE}^0_\el(G)_{n-1}$, $H$ ranges over proper isomorphism classes in $\mc H_\el(\mc H)$ (overloading notation so that $\mc H$ also refers to the group in the last $z$-pair of $\mc H$), and $z$ is a choice of $z$-pair for $H$. 

A consistent choice of raw hyperendoscopic paths $\mc{HE}^0_\el(G)$ is the union of an (inductively-chosen) consistent choice of $\mc{HE}^0_\el(G)_n$ for all $n > 0$. 
\end{dfn}
The sum is over a choice of $\mc{HE}^0_\el(G)$. If $\mc H \in \mc{HE}^0_\el(G)$, let $n_\mc H$ be its length. As shorthand, we will sometimes write
\[
\mc H = (H_1, H_2, \cdots, H_{n_\mc H})
\]
where $H_n$ is the group in the $z$-pair for the $n$th step in the path. As further shorthand, $\mc H$ will sometimes be overloaded to refer to $H_{n_\mc H}$. For indexing purposes, $H_0 = G$. Similarly define:
\[
\iota(G, \mc H) = (-1)^{n_\mc H}\prod_{i =1}^{n_\mc H} \iota(H_{i-1}, H_i) \qquad f^{\mc H} = (\cdots (f^{H_1})^{H_2\cdots})^{H_{n_\mc H}}.
\]
Note that $f^{\mc H}$ is not canonical and the choice of $f^{\mc H}$ needs to be consistent with the choice of $f^{\mc H'}$ where $\mc H'$ is $\mc H$ truncated by removing the last step. Finally, a hyperendoscopic path $\mc H$ determines central character datum $(\mf X_n, \chi_n)$ for each $H_n$. 

This expansion of course only works if the paths are all finite. This holds:
\begin{lem}
 Every element of $\mc{HE}^0_\el(G)$ has $n_\mc H \leq \rank_\ssm G$.
\end{lem}

\begin{proof}
Consider the quadruple $(H_i, \mc H, s_i, \eta_i)$ of $H_{i-1}$. The group $\wh H_i$ is a centralizer of $s_i \in \wh H_{i-1}$ that is not $\wh H_{i-1}$ since $H_{i-1} \neq H_i^{\qs}$. This has semisimple rank smaller than $\wh H_{i-1}$ from which the result follows.
\end{proof}

The key point then is that
\[
I_{\disc,t}^G(f) +  \sum_{\mc H \in \mc{HE}^0_\el(G)} \iota(G, \mc H) I_{\disc,t}^{H_{n_\mc H}}(f^{\mc H})
\]
is a stable distribution in $f^{G^\qs}$. Finally, since $G^\qs$ corresponds to the trivial endoscopic character, if $f^{G^\qs} = f_1^{G^\qs}$, then $f$, $f_1$ have the same stable orbital integrals. Setting this equal for two such functions:
\begin{prop}[{\cite[prop 3.4.3]{Fer07} corrected}]
Let $f$ and $f_1$ be functions on $G(\A)$ that have the same stable orbital integrals. Then
\[
I_{\disc,t}^G(f) = I_{\disc,t}^G(f_1) + \sum_{\mc H \in \mc{HE}^0_\el(G)} \iota(G, \mc H)I_{\disc,t}^{H_{n_\mc H}}((f_1 - f)^{\mc H}).
\]
\end{prop}

\subsection{Simplifying Hyperendoscopic Paths}
To control which groups appear, it is nice to have an easier definition of hyperendoscopic path. 
\begin{dfn}
An endoscopic path for $G$ is a sequence $(Q_1, \dotsc, Q_n)$ where $Q_1 \in \mc E_\el(G)$ and $Q_i \in \mc E_\el(H_{i-1})$ for $i > 1$ where $H_{i-1}$ is the group in $Q_{i-1}$. Note that if two endoscopic quadruples are isomorphic, then so are their groups.  
\end{dfn}
We use the same notation for endoscopic paths as for raw endoscopic paths. The set of endoscopic paths for $G$ will be called $\mc{HE}_\el(G)$. 

\begin{dfn}
A $z$-pair path for an endoscopic path $(Q_1, \dotsc, Q_n)$ is a sequence of $z$-pairs $(\td Q_1, \dotsc, \td Q_n)$ where
\begin{itemize}
\item
 $\td Q_1 = (\td H_1, \td \eta_1)$ is a choice of $z$-pair for $Q_1$.
 \item
 For $i > 1$ assume we have already chosen $Q_1, \dotsc, Q_{i-1}$. We get a quadruple $Q'_i$ for $H_{i-1}$ through repeated applications of the bijection from lemma \ref{zextpath}(a) down through the $Q_i$ (it will be clear that $H_{i-1}$ can be produced from the group in $Q_{i-1}$ by a sequence of central extensions by induced torii). Then $\td Q_i = (\td H_i, \td \eta_i)$ should be a $z$-pair for $Q'_i$. 
 \end{itemize}
\end{dfn}
If $\mc H \in \mc{HE}_\el(G)$ with $z$-pair path $\td{\mc H}$, we will sometimes overload notation and use $\td{\mc H}$ to denote that last group $\td H_n$ in the path. If $(\mf X, \chi)$ is a central character datum for $G$, we will also let $(\mf X_{\td{\mc H}}, \chi_{\td{\mc H}})$ be the induced datum on $\td{\mc H}$. We can also define $\iota(G, \td{\mc H})$ and transfers $f^{\td{\mc H}}$ similarly. 

As in the definition of raw hyperendoscopic paths, we can similarly inductively define a consistent choice of $z$-pair paths for all elements of $\mc{HE}_\el(G)$. 
\begin{lem}
Choose a consistent set of $z$-pair paths $\td{\mc H}$ for $\mc H \in \mc{HE}_\el(G)$. Then the set of combined data $\{[\mc H, \td{\mc H}] : \mc H \in \mc{HE}_\el(G)\}$ concatenated properly form a consistent set of raw hyperendoscopic paths for $G$. 
\end{lem}

\begin{proof}
We show this inductively on length. For length $1$, this works by definition. For longer length, we use lemma \ref{zextpath}(a): if we know this for length $i$ and $H_i$ is the $i$th group in $\mc H$, the corresponding $H_i'$ in the corresponding raw endoscopic path has the same possible ``next steps''---the elliptic quadruples of the two are in bijection.
\end{proof}
Finally,
\begin{lem}
Let $\td{\mc H}$, $\td{\mc H}'$ be two different $z$-extensions for the hyperendoscopic path $H$. Let $f \in \ms H(G, \chi)$ for some central character datum $(\mf X, \chi)$. Then the two terms $S^{\td{\mc H}}_{\chi_{\td{\mc H}}}(f^{\mc{\td H}})$ and $S^{\td{\mc H'}}_{\chi_{\td{\mc H'}}}(f^{\mc{\td H'}})$ are equal. In addition, $\iota(G, \td{\mc H})=\iota(G, \td{\mc H}')$.
\end{lem}

\begin{proof}
First, let $G$ be a group, $H$ an endoscopic group, and $f \in \ms H(G, \chi)$ for some $\chi$. Let $(\td H, \td \eta)$ and $(\td H', \td \eta')$ be two $z$-pairs. Then part of the formalism of the stable trace formula gives that $S^{\td H}_{\chi_{\td H}}(f^{\td H}) = S^{\td H'}_{\chi_{\td H'}}(f^{\td H'})$. By definition, $\iota(G, \td H) = \iota(G, H) = \iota(G, \td H')$. 

Second, if $G_1$ is a $z$-extension of $G$ and $f_1$ the pullback of $f$ to some $\ms H(G_1, \chi_1)$ where $\chi_1$ is the pullback of $\chi$, it induces extension $H_0$ of $H$ according lemma \ref{zextpath}(a). We can find a $z$-pair $(H_1, \eta_1)$ of $H$ such that $H_1$ is a $z$-extension of $H_0$. By a similar argument to section \ref{zextransfertrick}, $f^{H_1} = f_1^{H_1}$ and $\chi_{H_1} = (\chi_1)_{H_1}$. Therefore $S^{H_1}_{ (\chi_1)_{H_1}}(f_1^{H_1}) = S^{H_1}_{\chi_{H_1}}(f^{H_1})$. Since Tamagawa measures are products of Tamagawa measures of factors, $\iota(G, H_1) = \iota(G, H) = \iota(G_1, H_1)$ by the explicit formula.

The result follows from an induction alternating on these two steps. 
\end{proof}

Define $\iota(G, \mc H)$ to be the common value of all the $\iota(G, \td{\mc H})$. In total, we can choose whichever $z$-extensions we want and ignore the consistency condition:
\begin{thm}[The Hyperendoscopy Formula]\label{Hform}
Let $f$ and $f_1$ be functions on $G(\A)$ that have the same stable orbital integrals. Then
\[
I_{\disc,t}^G(f) = I_{\disc,t}^G(f_1) + \sum_{\mc H \in \mc{HE}_\el(G)} \iota(G, \mc H)I_{\disc,t}^{\td{\mc H}}((f_1 - f)^{\td{\mc H}})
\]
where $\td{\mc H}$ is a choice of z-extension path for $\mc H$ and where we suppress the central character datum.
\end{thm}

\subsection{Central Characters from Hyperendoscopy}\label{hyptransferchar}
Let $\mc H$ be a hyperendoscopic path for $G$ with $z$-extension $\td{\mc H}$ corresponding to the sequence of groups and embeddings $(\td H_i, \eta_i)$. We can, without loss of generality, assume that $H_0 = G$ has simply connected derived subgroup by taking further extensions. Then we can inductively define character on each $(H_i)_v$:
\begin{itemize}
\item
$\chi_1$ is the character $\lb_{\eta_1}$ on $(\td H_1)_v$ defined by $\eta_1$ as in section \ref{transferchar}.
\item
Let $\chi'_i$ be the character on $(\td H_{i+1})_v$ coming from character $\chi_i$ on $(\td H_i)_v$ as in section \ref{transferchar}. Let $\lb_{i+1}$ be the character on $(H_i)_v$ determined by $\eta_{i+1}$. Then set $\chi_{i+1} = \chi'_i \lb_{i+1}$. 
\end{itemize}
From all the previous discussion, we know $\chi_i$ are the characters such that given central character datum $(\mf X, \chi)$ and $f \in \ms H(G, \chi)$, the successive transfers $f^{\td H_i}$ lie in $\ms H(G, (\mf X_{\td H_i}, \chi \chi_i))$.

\subsection{Remarks on Usage}
Some notes for using this:
\begin{itemize}
\item
Beware that the transfers $(f_1 - f)^{\mc H}$ must be chosen explicitly, since the stable orbital integrals of $(f^{H_1})^{H_2}$ depend on the standard orbital integrals of $f^{H_1}$. Care should be taken in these choices since the ease of evaluating $I_\disc$ depends much on properties of $f^{H_1}$ that are not determined by stable orbital integrals.
\item
As a sum of distributions, the sum over $\mc E_\el(H_i)$ can be infinite. However, for any particular $f$ only finitely many terms are non-zero. Nevertheless, the number of such terms depends on the choices of $f^\mc H$ and can be arbitrarily large. Thankfully, if we choose the $f^\mc H$ so that they stay unramified outside of a finite set of places $S$, then there is a finite set of terms depending only on $S$ that are non-zero. See lemma \ref{conditions}.
\item
If we can choose the $f^H$ to be cuspidal, we do not need to worry that this formula is only in terms of $I_\disc$ instead of $I_\spec$.
\item
If each of the $H_i$ in path $\mc H$ are unramified, we can choose $\td{\mc H}$ to only have unramified groups since $z$-extensions can be chosen to have the same splitting field as the original group.  
\end{itemize}

\section{Lemmas on transfers}
\subsection{Formulas for Archimedean Transfer}\label{archtransfersection}
This section will compute transfers of pseudocoeffcients. We take the Whittaker normalization of transfer factors as in \cite{She10} and \cite{Lab11}. Because pseudocoefficients already have the correct dimensions, we do not need to fix Haar measures.  

Recall the parametrization of discrete series in section \ref{dseries}. We first make a basic remark:
\begin{lem}
Let $\pi \in \Pi_\disc(\xi)$ be a discrete series representation. Then for any $\gamma \in G_\infty$,
\[
SO_\gamma(f_\pi) = SO_\gamma(\eta_\xi).
\]
\end{lem}

\begin{proof}
Transfers from $G$ to $G^\qs$ are determined by the identities
\[
SO^G_\gamma(f_\pi) = SO^{G^\qs}_\gamma(f_\pi^{G^\qs}), \qquad SO^G_\gamma(\eta_\xi) = SO^{G^\qs}_\gamma(\eta_\xi^{G^\qs}).
\]
By \cite{She10}, transfers of pseudocoefficients can be chosen to be linear combinations of Euler-Poincar\'e functions. Such linear combinations are determined by evaluations on an elliptic torus so both of the transfers will be equal if we can show the lemma statement for just elliptic elements . The transfers being equal will suffice to prove the lemma. 

We therefore just need the computation from \cite[\S IV.3]{Lab11} with $\kappa$ trivial. The key point is that the $\Om(B_K)$ parametrizing $\pi \in \Pi_\disc(\xi)$ also parametrizes conjugacy classes in an elliptic stable class by section \ref{realendoscopiccharacters}.
\end{proof}
Now let $(H_\infty, \mc H, \eta, s)$ be an endoscopic quadruple of $G_\infty$. Fix an elliptic maximal torus $T$ and let and $\kappa$ be the corresponding endoscopic character on $\Om_G$.

\subsubsection{Trivial $z$-Extension case}
We will first work out the formula for transfers in the case where $\mc H \cong \Ld H$ where we do not need a $z$-extension. To start,
\begin{lem}\label{etorii}
Unless all elliptic tori $G_\infty$ are transfers of elliptic torii of $H_\infty$, transfers of pseudocoefficients can be taken to be $0$. 
\end{lem}
\begin{proof}
See lemma 3.2 in \cite{She10} or the computation of $\kappa$-orbital integrals on page 186 of \cite{Kot90}. 
\end{proof}
Therefore, we can choose isomorphic maximal torii $T_H$ and $T$ of $H_\C$ and $G_\C$ respectively that are both elliptic over $\R$. The Weyl chambers of $(H,T_H)$ are a coarser partition than those of $(G, T)$ by lemma \ref{roots}. Therefore, we can choose a positive Weyl chambers for $H$ that contains a chosen one for $G$. Let $B_H$ and $B_G$ be the corresponding Borel subgorups. Let $\rho' = \rho_G - \rho_H$ be the half-sum of positive roots of $G$ that are not roots of $H$.

The transfer of pseudocoefficients is worked out in \cite[\S7]{Kot90}. Special cases are worked out in terms of roots in \cite[\S IV.3]{Lab11}. For full generality when $\rho'$ is not a character of $T$, we have to use a corrected transfer factor from \cite[pg 396]{She82} as worked out in \cite{Fer07}. This involves an $\Om_H$-invariant $\mu^* = \mu^*_{G,H}$ such that $\mu^* - \rho'$ is a character of $T$. The $\mu^*$ is determined by the exact chosen isomorphism $\Ld H \to \mc H$. Finally, recall the endoscopic character $\kappa := \kappa_{G,H}$ on $\Om_G$  defined in section \ref{realendoscopiccharacters}.
\begin{prop}[{\cite[prop. 4.3.1]{Fer07}}]\label{onesteppseudo}
We can take
\[
(f_{\pi_G(\lb)})^H = \sum_{\om_* \in \Om_*} \kappa(\om_*) \eps(\om^* \om_0) f_{\pi_H(\om_* \lb - \mu^*)}
\]
where $\om_0^{-1} \lb$ is $B$-dominant and $\Om_*$ is the set of representatives $w$ of $\Om_H \bs \Om_G$ such that $w\lb$ is $B_H$-dominant.
\end{prop}
As a sanity check, note that if $A_{G, \infty} \in \mf X$, then ellipticity forces $A_{H, \infty} \in \mf X_H$. 

\cite{Fer07} explicitly computes the extension to hyperendoscopy: let $\Om(G,H)$ be a set of representatives $w$ of $\Om_H \bs \Om_G$ such that $w\mu$ is $B_H$ dominant for any $\mu$ that is $B_G$ dominant. Reindexing $\om_* = \om_1 \om_0^{-1}$
\[
(f_{\pi_G(\lb)})^H = \sum_{\om_1 \in \Om(G,H)} \kappa(\om_1 \om_0^{-1}) \eps(\om_1) f_{\pi_H(\om_1 \om_0^{-1} \lb - \mu^*)}.
\]
Next, note that the Euler-Poincar\'e function $\varphi_\lb$ has the same stable orbital integrals as the pseudocoefficient $f_{\pi_H(\lb + \rho_H)}$. Let $\mu  = \om_0^{-1}\lb - \rho_G$ so that $\pi_G(\lb)$ becomes $\pi_G(\mu, \om_0)$. Then
\begin{cor}
We can take
\[
(f_{\pi_G(\mu, \om_0)})^H = \sum_{\om_1 \in \Om(G,H)} \kappa(\om_1 \om_0^{-1}) \eps(\om_1) \varphi_{\om_1(\mu + \rho_G) -\rho_H -  \mu^*}.
\] 
\end{cor}

Next, since $\kappa$ is $\Om_\R$-right invariant,
\[
\sum_{\om_0 \in \Om(B_K)} \kappa(\om_1 \om_0^{-1}) = \sum_{[\om] \in \Om_\R \bs \Om_\C}  \kappa(\om_1 \om^{-1}) = \sum_{[\om] \in \Om_\C/\Om_\R} \kappa(\om_1 \om) = \sum_{[\om] \in \Om_\C/\Om_\R} \kappa(\om)
\]
where it does not matter which representatives $\om$ we choose. Therefore, averaging over $\om_0 \in \Om(B_K)$,
\begin{cor}[see {\cite[prop. 4.3.2]{Fer07}}]\label{eptransfer}
We can take
\[
(\varphi_{\mu})^H = \bar \kappa \sum_{\om_1 \in \Om(G,H)} \eps(\om_1) \varphi_{\om_1(\mu + \rho_G) - \rho_H - \mu^*}
\]
where $\bar \kappa = \bar \kappa_{G,H}$ is the average value of $\kappa$ over $\Om_\C/\Om_\R$. 
\end{cor}

\subsubsection{General case}\label{zexttransfer}
For $\mc H \not\cong \Ld H$, we use the trick in section \ref{zextransfertrick}. Let $\varphi : (G_1)_\infty \to G_\infty$ be the surjective map coming from the $z$-extension $G_1 \to G$: if $f$ is a function on $G_\infty$, we choose $f^{H_1} = (f \circ \phi)^{H_1}$. 

Given elliptic torii $T_{G_1}$ and $T_{H_1}$ as before, we can also get elliptic torus $T_G$ by taking images under the $z$-extensions. The function $\varphi : (G_1)_\infty \to G_\infty$ gives a map $\phi^* : X^*(G_\infty, T_G) \into X^*((G_1)_\infty, T_{G_1})$. Then $f_\pi(\lb) \circ \phi = f_{\pi(\phi^* \lb)}$ so we can still use the above formulas in the general case as long as we treat $\lb$ as an element of $X^*(G_1, T_{G_1})$. 

Note that the character $\lb_{H_1}$ shows up through the weight $\mu^*$---each may be used to compute the other (not that we've explicitly described either here).

\subsubsection{Hyperendoscopic Transfers}
To simplify notation, for any weight $\mu$ of a group $G$, endoscopic group $H$, and $\om \in \Om(G,H)$ as before, let
\[
T_{G,H}(\mu, \om) = \om(\mu+\rho_G) - \rho_H - \mu^*_{G,H}.
\]
As in the previous section, we interpret $\mu$ as an character of $G_1$ corresponding to the chosen $z$-extension $H_1$. 

For any hyperendoscopic path $\mc H = (H_i)_{0 \leq i \leq n}$, let
\[
\Om(\mc H) = \prod_{i=1}^n \Om(H_{i-1}, H_i) \qquad \bar \kappa_\mc H = \prod_{i=2}^n \bar \kappa_{H_{i-1}, H_i}.
\]
For $\om = (\om_i)_{i \leq i \leq n} \in \Om(\mc H)$ let
\[
\eps(\om) = \prod_{i=1}^n \eps(\om_i)
\]
and let
\[
T_\mc H(\mu, \om) = T_{H_{n-1}, H_n}( \cdots T_{G, H_1}(\mu_1, \om_1) \cdots, \om_n )
\]
be the composition of all the $T_{H_{i-1}, H_i}$. Inductively applying propositions \ref{onesteppseudo} and \ref{eptransfer} while keeping in mind section \ref{zexttransfer} then gives:
\begin{prop}[see {\cite[prop. 4.4.2]{Fer07}}]\label{archtransfer}
We can take
\[
(f_{\pi_G(\mu, \om_0)})^\mc H = \bar \kappa_{\mc H} \sum_{\om \in \Om_\mc H} \kappa_{G, H_1}(\om_1 \om_0^{-1}) \eps(\om) \varphi_{T_\mc H(\mu, \om)}
\]
with the terms defined as in the above paragraph.
\end{prop}
Note that all the coefficients in the sum have norm $1$ and define $\Xi_{\mu, \mc H}$ to be the set of $T_{\mc H}(\mu, \om)$ for $\om \in \Om(\mc H)$.

\subsection{Bounds on Archimedean Transfers}\label{sectionArchtransferbounds}
Here are few lemmas on the terms that appear in proposition \ref{archtransfer}. For $\mu$ a weight of $G$ define
\begin{itemize}
\item
$m(\mu) = m_G(\mu) = \min_{\alpha \in \Phi^+(G)} \langle \alpha, \mu + \rho_G \rangle$
\item
$n(\mu) = n_G(\mu) = \min_{\alpha \in \Phi^+(G)} \langle \alpha, \mu \rangle$
\item
$\dim \mu = \dim_G(\mu)$ is the dimension of the finite dimensional representation with highest weight $\mu$. 
\end{itemize}
\begin{lem}\label{regbound}
If $\mu$ is a weight of $G$ and $\mc H$ as before, then for all $\mu' \in \Xi_{\mu, \mc H}$, $n_G(\mu') \geq n_{\mc H}(\mu)$. In particular, $\mu'$ is regular if $\mu$ is. 
\end{lem}
\begin{proof}
In the situation where $H$ is just an endoscopic group, consider $\om \in \Om_G$ such that $\mu' = \om(\mu + \rho_G) - \rho_H - \mu^* \in \Xi_{\mu, H}$. Consider $\alpha \in \Phi^+(H)$. Since $\mu^*$ is invariant under $\Om_H$, $\langle \mu^*, \alpha \rangle = 0$ so
\[
\langle \mu', \alpha \rangle =  \langle \om \mu, \alpha \rangle + \langle \om \rho_G - \rho_H, \alpha \rangle.
\]
Next, $\rho_G$ is the sum of the fundamental weights so it is a regular weight. This implies that $\om \rho_G$ is too. Therefore, for all $\beta \in \Phi^+(G)$, $\beta^\vee(\om \rho_G) \in \Z \setminus \{0\}$. In particular, since $\om \rho_G$ is $B_H$-dominant, for $\alpha \in \Phi^+(H)$, $\alpha^\vee(\om \rho_G) \geq 1$. If $\alpha$ is in addition simple, we can compute
\[
\alpha^\vee(\om \rho_G - \rho_H) \geq 1 - \alpha^\vee(\rho_H) = 0,
\]
so $\om \rho_G - \rho_H$ is $B_H$-dominant. This gives
\[
\langle \mu', \alpha \rangle \geq \langle \om \mu, \alpha \rangle.
\]
To finish this one-step case,
\[
n_H(\mu') = \min_{\alpha \in \Phi^+(H)} \langle \mu', \alpha \rangle \geq \min_{\alpha \in \Phi^+(H)} \langle \om \mu, \alpha \rangle  = \min_{\alpha \in \Phi^+(H)} \langle \mu, \om^{-1} \alpha \rangle.
\]
All the terms in the last two minimums have to be positive. However, $\mu$ is $B_G$-dominant so this means the $\om^{-1}\alpha$ are all in $\Phi^+(G)$ giving
\[
n_H(\mu') \geq \min_{\alpha \in \Phi^+(G)} \langle \mu, \alpha \rangle = n_G(\mu).
\]

Finally, for an arbitrary endoscopic path, inductively continue this argument through each step.
\end{proof}

\begin{lem}\label{dimbound}
If $\mu$ is a weight of $G$ and $\mc H$ as before, then for all $\mu' \in \Xi_{\mu, \mc H}$
\[
\f{\dim_{\mc H}(\mu')}{\dim_G(\mu)} = O(m_G(\mu)^{-1})
\]
with the implied constant only depending on $G$ and $\mc H$. 
\end{lem}
\begin{proof}
This follows from the Weyl character formula. If $H$ is just an endoscopic group, let $\mu' = \om(\mu + \rho_G) - \rho_H - \mu^*$ for appropriate $\om \in \Om_G$. Using that $\mu^*$ pairs to zero with any root of $H$,
\[
\f{\dim_H(\mu')}{\dim_G(\mu)} = \f{\prod_{\alpha \in \Phi^+(G)} \langle \alpha, \rho_H \rangle}{\prod_{\alpha \in \Phi^+(H)} \langle \alpha, \rho_G \rangle}\f{\prod_{\alpha \in \Phi^+(H)} \lf( \langle \alpha, \om \mu \rangle  + \langle \alpha,  \om\rho_G \rangle \ri)}{\prod_{\alpha \in \Phi^+(G)} \lf( \langle \alpha, \mu \rangle  + \langle \alpha,  \rho_G \rangle\ri)}.
\]
The first fraction is a constant depending only on $G$ and $H$. The second terms in the products in the second fraction are also. A priori, the $\langle \alpha, \om \mu \rangle = \langle \om^{-1} \alpha, \mu \rangle$ are a subset of the $\langle \pm \beta, \mu \rangle$ for $\beta \in \Phi^+(G)$. However, since they all have to be positive since $\om \mu$ is $B_H$-dominant, they are actually a subset of the $\langle \beta, \mu \rangle$. Denote by $A$ the subset of such $\beta$. Then
\[
\f{\dim_H(\mu')}{\dim_G(\mu)} = C \f{\prod_{\alpha \in A} (\langle \alpha, \mu \rangle  + O(1))}{\prod_{\alpha \in \Phi^+(G)} (\langle \alpha, \mu \rangle  + O(1))} = O\lf( \prod_{\alpha \in \Phi^+(G) \setminus A} \langle \alpha, \mu \rangle^{-1} \ri)
\]
using that the pairings are bounded below by a constant. Bounding the pairings again by $m_G(\mu)$, this is $O(m_G(\mu)^{|\Phi^+(H)| - |\Phi^+(G)|})$. Finally, since endoscopic groups have smaller rank, they do not have the same root data as the original group so this difference has to be negative. 

After a quick check that the $m_{\mc H_i}(\mu') = O(m_G(\mu))$, inducting on this argument for each step of the hyperendoscopic path $\mc H$ finishes the proof . 
\end{proof}

\subsection{Truncated Hecke algebras}\label{trunhalg}
We now move on to the unramified finite places. Fix a place $v$ at which $G_v$ is quasisplit. Since we are only working at $v$, for this subsection $G$ will always mean $G_v$ to simplify notation. 

Choose $(B,T)$ to be  a Borel and maximal torus defined over $F_v$. By $G$ being quasisplit, all such choices are conjugate and $T$ automatically contains a maximal split torus $A$. Furthermore, $\Om_F$ can be identified with the fixed points $\Om^{W_F}$ and therefore the Weyl group of the relative root system of rational roots in $X^*(A)$. Let $K$ be a hyperspecial subgroup from a hyperspecial point in the apartment corresponding to $A$. 

Eventually, we will evaluate $I_\geom(f)$ up to some error bounds which depend on how big the support of the finite part of $f$ is. To precisely measure this size, we slightly modify the notion of truncated Hecke algebras as in \cite[\S2]{ST16}. 

Recall then that the elements $\tau^G_\lb = \1_{K \lb(\varpi) K}$ for a chosen uniformizer $\varpi$ and $\lb \in X_*(A)^+$ generate $\ms H(G,K)$.  
Pick a basis $\mc B$ for the $X_*(A)$ and define norm
\[
\|\lb\|_\mc B = \max_{\om \in \Om} (\text{biggest }|\mc B\text{-coordinate of }\om \lb|)
\]
for $\lb \in X_*(A)$. Define truncated Hecke algebra
\[
\ms H(G,K)^{\leq \kappa, \mc B} = \langle \tau^G_\lb : \|\lb\|_\mc B \leq \kappa \rangle.
\]
It turns out (see \cite[\S2]{ST16}) that for any two $\mc B,\mc B'$, $\|\lb\|_\mc B = \Theta(\|\lb\|_{\mc B'})$. All the bounds we use will depend on $\kappa$ only up to an unspecified constant. Therefore we can suppress the $\mc B$.

There is also a truncated Hecke algebra with central character data: choose an $(\mf X, \chi)$ such that $\chi$ is unramified. In the case we care about, $\mf X$ is a subtorus of $Z_{G}$. Let $A_{\mf X}$ be its split part. Define
\[
\ms H(G, K, \chi)^{\leq \kappa, \mc B} = \langle \tau^G_\lb : \|\lb+\zeta\|_\mc B \leq \kappa \text{ for some } \zeta \in X_*(A_{\mf X}) \rangle \cap \ms H(G, K, \chi).
\]
Note that for $x \in K\lb(\varpi)K$ and $z \in \mf X$, then there is $k \in K$ and $\zeta \in X_*(A_{\mf X})$ such that $z = \zeta(\varpi) k$, implying $zx \in K(\lb + \zeta)(\varpi)K$. Therefore, this is a reasonable, non-empty intersection.

\subsubsection{A useful projection}
Working with the basis of $\tau^G_\lb$, it is sometimes useful to consider the following maps. First, there is a map $Q : \chi \mapsto \sum_{\om \in \Om_G} \om \chi$ on $X_*(T)$. This sends every coroot of $G$ to $0$. Normalizing $Q$ by $|\Om_G|^{-1}$ gives a projection $P$ on $X_*(T)\otimes\Q$. Note that this projection is onto $X_*(Z_G)\otimes\Q$ since Weyl-invariant cocharacters are the same as central cocharacters (they pair to zero with every root). 

Recall $X_*(A)$ embeds into $X_*(T)$ as the $W_F$ invariants. 
\begin{lem}
Let $\lb \in X_*(A)$. Then, $Q \lb \in X_*(A)$. 
\end{lem}

\begin{proof}
It suffices to show this for $P\lb$. The map $P$ is an orthogonal projection onto $W_F$-invariant $X_*(Z_G)\otimes\Q$ with respect to a $W_F$-invariant inner product. Therefore it commutes with $W_F$ and sends $W_F$ invariants to $W_F$ invariants. 
\end{proof}
Therefore, we can consider $Q$ and $P$ as maps of $X_*(A)$ and $X_*(A)\otimes\Q$ respectively. The kernel of $P$ is the span of the roots of $G$ so the kernel in $X_*(A)\otimes\Q$ is $V_F$ where $V_F$ is the span of $\{\alpha^\vee|\alpha \in \Phi^*_F\}$ inside $X_*(A)\otimes\Q$.

\subsection{Formulas for Unramified Non-Archimedean Transfers}\label{unramformulasection}
Fix a place $v$ at which $G_v$ is quasisplit. Since we are only working at $v$, for this subsection $G$ will always mean $G_v$ to simplify notation. 

\subsubsection{The Fundamental Lemma}\label{funlem}
The fundamental lemma allows for computation of unramified non-Archimedean transfers (the lemma is actually enough to show the existence of all non-Archimedean transfers). We will eventually use this to control which $\ms H(H_v, K_{H,v}, \chi_{H,v})^{\leq \kappa}$ transfers end up being in. Use the notation $T,A$, and $K$ analogous to the last section. 

As explained in \cite[\S2.2]{ST16}, the Satake transform gives two isomorphisms
\[
\varphi_G : \ms H(G,K) \to \ms H(A, A\cap K)^{\Om_F} \to \C[X_*(A)]^{\Om_F}.
\]
We mention that this implies:
\begin{lem}
The space $\wh G^\ur$ can be identified with $\Om_F \bs \wh A$. The tempered part is $\Om_F \bs \wh A_c$ where $\wh A_c$ is the maximum compact torus in $\wh A$. 
\end{lem}
\begin{proof}
A result in representation theory of $p$-adic groups says that unramified representations of $G$ are the same as characters of $\ms H(G, K)$ and therefore characters on $ \C[X_*(A)]^{\Om_F}$ (see \cite[\S10]{Bor79}). These are the same as elements of $\Om_F \bs \wh A$. Tempered representations need to correspond to tempered characters of $\ms H(G, K)$ which forces the element to be in $\wh A_c$. 
\end{proof}

There are more implications: let $\Ld G^\ur := \Ld G^\ur$ be defined like $\Ld G$ except that the semidirect product is only with $W_{F_v}^\ur$. Define $\C[\ch(\Ld G^\ur)]$ to be the algebra of trace characters of representations of $\Ld G^\ur$ restricted to $(\wh G \rtimes \Frob)_{\ssm}$. There is a third isomorphism
\[
\mc T : \C[\ch(\Ld G^\ur)] \to  \C[X_*(A)]^{\Om_F}
\]
that takes a representation $\pi$ to a function on $\wh T$ given by $a \mapsto \tr_\pi (a \rtimes \Frob)$. This function can be shown to factor through $\wh A$ (see \cite[prop 6.7]{Bor79}).

If we have a map $\eta : \Ld H^\ur \into \Ld G^\ur$, we get a pullback map $b_\eta : \C[\ch(\Ld G^\ur)] \to \C[\ch(\Ld H^\ur)]$. We pick the Whittaker normalization for transfer factors and choose the measures $\mu^\can$ on $H$ and $G$ that give $K$ and $K_{H}$ volume $1$. 
\begin{thm}[Full Fundamental Lemma] 
Let $G$ be an unramified reductive group over the local field $F_v$. Let $(H, \mc H, \eta, s)$ be an elliptic endoscopic quadruple for $G$ such that $\mc H \cong \Ld H$. Then, for $f \in \ms H(G,K)$ we can take
\[
f^H = \begin{cases} \varphi_H^{-1} \circ b_\eta \circ \varphi_G(f) & H \text{ unramified} \\
0 & H \text{ ramified}
\end{cases}.
\]
Here we recall that if $H$ and $G$ are unramified, then the embedding $\mc H \into \Ld G$ descends to one $\mc H^\ur \into \Ld G^\ur$. In addition, $H$ being unramified allows us to pick an $\eta : \Ld H \iso \mc H$ that also descends to unramfied $L$-groups. The pullback $b_\eta$ is defined through such an $\eta$. 
\end{thm}

\begin{proof}
The statements defining $\eta$ come from the construction of $\mc H$ and the proof of 7.2A in \cite{KS99}. 

The ramified $H$ case is by \cite[\S7.5]{Kot86}. Otherwise, it is reduced in \cite{Hal95} to proving the result for just $\1_K$. This was further reduced to a fundamental lemma for Lie algebras in \cite{Wal97} which was finally proven in \cite{Ngo10}. \cite{Hal95} removes a restriction on the size of the residue field of $F_v$. 
\end{proof}

\subsubsection{Representations of $\Ld G^\ur$}
To compute with the fundamental lemma, we need to describe representations of $\Ld G^\ur$. As a start:

\begin{lem}
Let $\pi$ be a representation of $\Ld T^\ur$. Then there exists $\lb$ a character of $\wh T$ up to $W_F^\ur$-action and $\alpha \in \C^\times$ such that $\pi = \chi_{\lb, \alpha}$ where
\[
\chi_{\lb, \alpha} = \bigoplus_{\gamma \in W_F/\Stab \lb} V_{\gamma \lb}
\]
and each $V_{\mu}$ is a $1$-dimensional space with a chosen generator $v_{\mu}$ on which $\wh T$ acts through $\mu$. Let $\Stab \lb$ be generated by $\Frob^{i(\lb)}$. Then $\Frob^{i(\lb)}$ acts by $v_\lb \mapsto \alpha v_\lb$. Finally, $\Frob(v_\lb) = \beta_\lb v_{\Frob(\lb)}$ for some constants $\beta_\lb$. (Note that by scaling $v_\mu$, without loss of generality all the $\beta_\lb$ are $1$ except one that is $\alpha$). 
\end{lem}

\begin{proof}
Decompose $\pi$ into eigenspaces $V_{\mu}$ for $\wh T$. We can compute that, $\gamma V_{\mu} \subseteq V_{\gamma \mu}$ for $\gamma \in W_F^\ur$. Let $\gamma_0$ generate $\Stab \lb$ for some non-empty $V_\lb$. Then $\gamma_0$ acts as an element of $\GL(V_\lb)$. Let $v_\lb$ be a chosen eigenvector of $\gamma_0$ with eigenvalue $\alpha$. The vectors $v_\lb$ generates a $\chi_{\lb, \alpha}$ inside $\pi$.
\end{proof}
Beware that this parametrization depends on the splitting $W_F \into \Ld T$. Next
\begin{prop}
Representations of $\Ld G^\ur$ are parametrized by $\chi_{\lb, \alpha}$ of $\Ld T^\ur$ for $\alpha$ dominant. Call the one corresponding to $\chi_{\lb, \alpha}$ by $\pi_{\lb, \alpha} := \pi^{\Ld G}_{\lb, \alpha}$.
\end{prop}

\begin{proof}
This is by \cite[pg 375-376]{Kos61} . We have that $\Ld T^\ur$ is the same as $H^+$ in the reference because the action of $W_F$ fixes the Borel $B$ used to define $\Ld G$. The construction is similar to that for connected complex Lie groups: $\pi_{\lb, \alpha}$ forms a highest weight space on which the actions of the root subgroups of $\wh G$ are determined. Together $\wh G$ and $\Ld T^\ur$ generate $\Ld G^\ur$.
\end{proof}

 In fact, if $\pi^{\wh G}_\lb$ is the representation corresponding to highest weight $\lb$ of $\wh G$, then each of the $V_{\gamma \lb} \subseteq V_{\lb, \alpha}$ generates a copy of $\pi^{\wh G}_{\gamma \lb}$ under the action of $\wh G$. The representation $\pi_{\lb, \alpha}|_{\wh G}$ therefore decomposes as a direct sum of the $\pi^{\wh G}_{\gamma \lb}$ and any $\gamma \in W_F$ sends $\pi^{\wh G}_\mu$ to $\pi^{\wh G}_{\gamma \mu}$. The exact description of this map in complicated but can be computed by the following trick: For any $\gamma \in \Gamma$, the $\mu$ coefficient of $\tr_\pi$ restricted to $\wh T \rtimes \gamma$ is the trace of $1 \rtimes \gamma$ acting on the $\mu$-weight space $V^{\lb, \alpha}_\mu$ of $\pi_{\lb, \alpha}$. This trace can be computed by Kostant's character formula \cite[thm 7.5]{Kos61}.
 
 As an easier way to think about this parametrization, let $F_n$ be the splitting field for $G$. The groups $\Gal(F_n/F_v)$ and $\Om_\C$ together generate a group $C$ in automorphisms of the set of roots. Inside this, $\Gal(F_n/F_v)$ is the stabilizer of the positive Weyl chamber and $\Om_\C$ acts simply on the Weyl chambers so $\Gal(F_n/F_v) \cap \Om_\C = 1$. In addition, $\Om_\C$ is normal since $T$ is fixed by Galois. Therefore, $C = \Om_\C \rtimes \Gal(F_n/F_v)$. The $\lb$ parametrizing $\pi_{\lb, \alpha}$ can be thought of as a $C$-orbit. This decomposes into $\Om_\C$ orbits representing the constituent $\pi^{\wh G}_{\gamma \lb}$.

 \subsubsection{Some Bases}
 We also need to describe some bases of the various spaces. 
 
If $\varpi$ is a chosen uniformizer for $\mc O_F$ and $X_*(A)^+$ a chosen Weyl chamber, then the functions
\[
\tau_\lb^G = \1_{K \lb(\varpi) K} \text{ for all } \lb \in X_*(A)^+
\]
form a basis for $\ms H(G,K)$ (the corresponding double cosets partition $G$ by the Cartan decomposition).

$\C[X_*(A)]^{\Om_F}$ contains functions
\[
\chi_\lb = \f{\sum_{\sigma \in \Om_F} \sgn_F(\sigma) \sigma(\lb \cdot \rho)}{\sum_{\sigma \in \Om_F} \sgn_F(\sigma) \sigma(\rho)} \in \C[X_*(A)]^{\Om_F}
\]
for $\lb \in X_*(A)^+$. We write the addition in $X_*(A)$ multiplicatively for clarity. Here, $\rho = \rho_F$ is the half-sum of the positive roots of $\wh G$ over $F_v$, which is the same as the half-sum of all positive roots since rational roots are sums over orbits of roots. We recall that $\Om_F$ is the same as the Weyl group for the relative root system of rational roots of $G_v$ by quasiplitness (See \cite[\S6.1]{Bor79}). The $\sgn_F$ here are $-1$ to the power of the number of positive rational roots sent to negative roots. If the rational roots form a reduced root system, this is just the standard $\sgn$ on $\Om_F$.

If the relative root system is reduced, these are the standard characters from Weyl's character formula and are studied in \cite{Kat82}. In the non-reduced case, these are the twisted characters from \cite[thm 1.4.1]{CCH19} or \cite[thm 7.9]{Hai18}. Either way, $\chi_\lb$ for dominant weighs $\lb$ form a basis for $\C[X_*(A)]^{\Om_F}$. 

Finally, 
\begin{lem}\label{ldreptrace}
\[
\mc T(\pi_{\lb,\alpha}) = \begin{cases}
\alpha \chi_\lb & \lb \in X_*(A) \\
0 & \text{else}
\end{cases}.
\]
\end{lem}

\begin{proof}
This is just stated in the proof of \cite{ST16} lemma 2.1. We give details here since there seems to be a minor mistake (that is irrelevant to all the work there and here) when $\lb$ is not in $X_*(A)$. This is also proven as \cite[thm 1.4.1]{CCH19} and as \cite[thm 7.9]{Hai18} in a slightly different form. 

We use Kostant's character formula \cite[thm 7.5]{Kos61}. Using the notation there, $a = t \rtimes \Frob$ for some $t \in \wh T$ and $W_a$ is the $W_F^\ur$ invariants in $\Om_\C$ which is $\Om_F$. Also, let $\Phi_\sigma = \Phi^+_\C \cap \sigma(-\Phi^+_\C)$ for $\sigma \in \Om_\C$ where $\Phi^+_\C$ is the set of positive roots. Since $\Frob$ preserves a pinning, it acts by a permutation on some diagonal basis of $\bigoplus_{\phi \in \Phi_\sigma} \mf g_{-\phi}$. Therefore, the determinant of the action of $a$ is
\[
\chi_1^\sigma(a) = \sgn(\Frob|_{\Phi_\sigma})\prod_{\varphi \in \Phi_\sigma}  \varphi^{-1}(t).
\]
In addition $\chi^\delta_1(a)$ for $\delta$ the representation of $\Ld T$ parametrized by $(\lb, \alpha)$ is $\alpha \lb(t)$ if $\lb$ is fixed by $\Frob$ and $0$ otherwise (the $0$ otherwise case is what is missing in \cite{ST16}). By a \cite[pg 15]{LS87}, we can find representations of $\sigma \in W_a$ fixed by $\Frob$ so we get that $\chi^\delta_\sigma(a) = \alpha \sigma\lb(t)$. 

In total, the trace in the non-zero case is
\begin{multline*}
\alpha \f{\sum_{\sigma \in \Om_F}  \sgn_\C(\sigma)\sgn(\Frob|_{\Phi_\sigma})  \sigma\lb(t)\prod_{\varphi \in \Phi_\sigma} \varphi^{-1}(t)}{\sum_{\sigma \in \Om_F}\sgn_\C(\sigma)\sgn(\Frob|_{\Phi_\sigma}) \prod_{\varphi \in \Phi_\sigma} \varphi^{-1}(t)} \\
= \alpha \f{\rho(t)^{-1} \sum_{\sigma \in \Om_F}  \sgn_\C(\sigma)\sgn(\Frob|_{\Phi_\sigma})  \sigma\lb(t) \sigma \rho(t) }{\rho (t)^{-1} \sum_{\sigma \in \Om_F}\sgn_\C(\sigma)\sgn(\Frob|_{\Phi_\sigma}) \sigma \rho(t)}.
\end{multline*}
The $\sgn_\C$ here is the sign character for $\Om_\C$: the number of all positive roots sent to negative roots. This differs from the $\sgn_F$ in the formula for $\chi_\lb$ by a factor of $\sgn(\Frob|_{\Phi_\sigma})$ through an argument breaking up $\Phi_\sigma$ into $\Frob$-orbits and noting that each rational root is a sum over an orbit. Therefore, we are done. 

Note that the $0$ case can be done more easily by thinking about the action in block matrix form with respect to the subspaces $\pi^{\wh G}_{\gamma \lb}$ and noticing that all diagonal blocks are $0$. 
\end{proof}
The key consequence of this is that the $\mc T(\pi_{\lb, 1})$ for $\lb \in X_*(A)$ form a basis for $\C[\ch(\Ld G^\ur)]$.

\subsection{Bounds on Unramified Transfers}\label{unramboundsection} 
\red{LEVEL ASPECT: all polynomials here uniform over endoscopic groups??}
\subsubsection{Trivial $z$-extension case}
As in the Archimedean case, we consider the trivial $z$-extension case first. 

Recall the notation for various bases of spaces related to the Satake isomorphism. From \cite{Gro98} and \cite{Kat82} (again, see \cite[\S7]{Hai18} or \cite[\S1]{CCH19} for the non-split case), we can write
\begin{gather*}
\varphi_G(\tau_\lb^G) = \chi_\lb +  \sum_{\substack{\mu \in X^*(\wh A)^+ \\ 0 \leq \mu < \lb}} b^G_\lb(\mu) \chi^G_\mu, \\
\varphi_H^{-1}(\chi^H_\nu) =   q^{-\langle \nu, \rho_H \rangle } \tau^H_\nu +  \sum_{\substack{\xi \in X^*(\wh A_H)^+ \\ 0 \leq \xi < \nu}} q^{-\langle \xi, \rho_H \rangle } d^H_\nu(\xi) \tau^H_\xi.
\end{gather*}
for some constants $b$ and $d$. Here $\mu \leq \lb$ means that there is some non-negative integer linear combination of roots $\alpha^\vee$ for $\alpha \in \Phi^+$ equal to $\lb - \mu$. 
\begin{lem}\label{dmu}
The $d^G_\lb(\mu)$ and $q^{-\langle \lb, \rho_H \rangle } b^G_\lb(\mu)$ are bounded by a polynomial in the norm $\|\mu\|$ that is independent of $q$ and $\lb$. 
\end{lem}

\begin{proof}
First, let's show this for $d^G_\lb(\mu)$. By the above, we can ignore the $\lb=\mu$ case. Otherwise, we apply \cite[lem 2.2]{ST16}. There is a small issue here: this lemma depends on the main result of \cite{Kat82} which only works when the root system is reduced. Nevertheless, \cite[thm 7.10]{Hai18} and \cite[thm 1.9.1]{CCH19} provide an appropriate substitute in the non-reduced case.

\cite[lem 2.2]{ST16} bounds $d^G_\lb(\mu)$ by $|\Om_{G, F_v}|$ times the size of the set of tuples $(c_{\alpha^\vee})$ for $\alpha$ a positive root such that $\sum_{\alpha^\vee} c_{\alpha^\vee} \alpha^\vee = \mu - \lb$ (since both $\mu$ and $\lb$ are in the positive Weyl chamber, the max in the lemma is achieved for the trivial element of the Weyl group). Looking at the coordinate of $\mu$ in the direction used to define positivity, every $\alpha^\vee$ is positive in this coordinate, so some weighted sum of the $c_{\alpha^\vee}$ is bounded. This implies that the number of tuples is only polynomial in this coordinate of $\mu$. The result follows.

For $b^G_\lb(\mu)$, note that the $q^{-\langle \beta, \rho_H \rangle } d^G_\alpha(\beta)$ for $\alpha, \beta \leq \lb$ form an upper-triangular matrix with dimension polynomial in the size of $\lb$. Then, $b^G_\beta(\alpha)$ are coordinates of the inverse of this matrix. Making a change of variables, the $q^{-\langle \beta, \rho_H \rangle } b^G_\beta(\alpha)$ are the coordinates of the inverse of the matrix with coordinates $d^G_\alpha(\beta)$ so these are bounded by a polynomial in $\|\mu\|$ by solving through back substitution. 
\end{proof}

It remains to understand the map $b_\eta$. This is computed exactly in terms of certain partition functions in \cite[\S2.3]{CCH19}, but we only need bounds so we do something slightly different and much simpler.
For $\mu \in X_*(A)$, define coefficients $c_\mu(\nu)$ by
\[
\pi^{\wh G}_\mu|_{\wh H} = \bigoplus_{\substack{\nu \in X_*(T_H)^+ \\0 \leq \nu \leq \mu}} c_\mu(\nu) \pi^{\wh H}_\nu.
\]
The $c_\nu(\mu)$ are in particular bounded by the dimension of $\pi^{\wh G}_\mu$ so they are polynomial in the size of $\mu$ by the Weyl character formula.
\begin{prop}\label{transferform}
As elements of $\C[\ch(\Ld H^\ur)]$,
\[
b_\eta(\pi^{\Ld G}_{\mu,1}) = \bigoplus_{\substack{\nu \in X_*(A_H)^+ \\0 \leq \nu \leq \mu}} \alpha_\mu(\nu) c_\mu(\nu) \pi^{\Ld H}_{\nu,1}
\]
where $A_H$ is the maximal split torus of $H$ contained in some maximal $T_H$ contained in a rational Borel $B_H$ and we consider $\mu \in X_*(T_H) = X_*(T)$ as dominant element by taking its Weyl-translate in the positive Weyl chamber.  

For notational convenience, let $\Gamma = W_{F_v}^\ur$. There exists $t_\eta \in (Z_{\wh G}^{\Gamma})^0$ depending only on $\eta$ such that the constants $\alpha_\mu(\nu)$ satisfy two properties:
\begin{itemize}
\item
$|\alpha_\mu(\nu)| \leq |\nu(t_\eta)|$.  
\item
Let $Y_G$ be the maximal split torus in $Z_G^0$. If $\zeta \in X_*(Y_G)$, then $\alpha_{\mu + \zeta}(\nu+\zeta) = \zeta(t_\eta)\alpha_\mu(\nu)$. 
\end{itemize}
\end{prop}

Before starting the proof, note that all such $T_H$ are isomorphic and that the map $X_*(T_H) \to X_*(T)$ is unique up to Weyl element. Therefore, this is well defined. 

\begin{proof}
\underline{Decomposition:} To avoid confusion, $\Gamma_{\wh G}$ is $\Gamma$ acting on $\wh G$ and visa versa for $\wh H$ when it is not clear from context. First, 
\[
b_\eta(\pi^{\Ld G}_{\mu,1})|_{\wh H} = (\pi^{\Ld G}_{\mu,1}|_{\wh G})|_{\wh H} = \bigoplus_\gamma \pi^{\wh G}_{\gamma \mu}|_{\wh H} = \bigoplus_\gamma\bigoplus_{\substack{\nu \in X_*(T_H)^+ \\0 \leq \nu \leq \mu}} c_\mu(\nu) \pi^{\wh H}_{\gamma_G \nu} 
\]
where the $\gamma \mu$ index the $\Gamma_{\wh G}$-orbit of $\mu$ in $X_*(T)$. Note that $c_\mu(\nu)$ is constant on $\Gamma_{\wh G}$ orbits and $\Om_\C(\wh G)$ orbits. 

The $\Gamma_{\wh H}$-action is the composition of the action of $\Gamma_{\wh G}$ with conjugation by elements of $N_{\wh G}(T)$, so since $G$ is quasisplit, $\Gamma_{\wh H}$ acts on $\wh T_H$ through a subgroup $W'$ with $\Gal(F_n/F) \subseteq W' \subseteq C_H \subseteq C_G$ (recall notation $C_G = \Gamma \rtimes \Om_{\wh G, \C}$). This implies that $c_\mu(\nu)$ is constant on $\Gamma_{\wh H}$-orbits. 

Therefore, the sum over such an orbit of the $c_\mu(\nu) \pi^{\wh H}_\nu$ decomposes into $c_\mu(\nu)$ different $\pi^{\Ld H}_{\nu, \alpha_{i,\mu}}$ for possibly different $\alpha_{i,\mu}$. In total
\[
b_\eta(\pi^{\Ld G}_{\mu,1}) = \bigoplus_{\substack{\nu \in X_*(T_H)^+ \\0 \leq \nu \leq \mu}} \bigoplus_{i=1}^{c_\mu(\nu)} \pi^{\Ld H}_{\nu, \alpha_{i,\mu}(\nu)} = \bigoplus_{\substack{\nu \in X_*(A_H)^+ \\0 \leq \nu \leq \mu}} \lf( \sum_{i=1}^{c_\mu(\nu)} \alpha_{i,\mu}(\nu) \ri) \pi^{\Ld H}_{\nu, 1}
\]
as elements of $\C[\ch(\Ld H^\ur)]$ and for some $\alpha_{i,\mu}(\nu) \in \C^\times$. Let $\alpha_\mu(\nu)$ be the average of the $\alpha_{i,\mu}(\nu)$. 

\underline{Properties of $\alpha_\mu(\nu)$:} It remains to show the two properties of  $\alpha_\mu(\nu)$. Since all $(B,T)$-pairs in $\wh G$ are conjugate, without loss of generality take an inner automorphism of $\Ld G$ so that $(\wh B_H, \wh T_H)$ is the pullback of $(\wh B,\wh T)$. The map $\eta$ determines a cocycle $c_\gamma \in C^1(\Gamma_{\wh G}, \wh G)$ by $\eta(1 \rtimes \gamma) = c_\gamma \rtimes \gamma$. We then have that $\alpha_i(\nu)$ is the factor by which $c_\Frob \rtimes \Frob$ acts on the highest weight space $V$ of the $i$th $\pi^{\wh H}_\nu$. 

There exists $n$ such that the conjugation action of $(c_\Frob \rtimes \Frob)^n$ on $X^*(\wh T)$ is trivial. Since this action also fixes a pinning of $H$, we must have
\[
(c_\Frob \rtimes \Frob)^n = z_0 \rtimes \Frob^n
\] 
for some $z_0 \in Z_{\wh H}$. By the lemma below, we know $1 \rtimes \Frob$ acts trivially on $V$. Therefore, $\alpha_{i,\mu}(\nu)^n = \nu(z_0)$. 

Next, note that the $\Gamma_{\wh H}$-action is generated by conjugation by $c_\Frob \rtimes \Frob$. This fixes $z_0$ so $z_0 \in Z_{\wh H}^{\Gamma}$. We can without loss of generality make $n$ bigger so that $z_0$ is trivial in the finite group $\pi_0(Z_{\wh H}^{\Gamma})$---in other words, we may without loss of generality assume $z_0 \in (Z_{\wh H}^{\Gamma})^0$. Then by ellipticity of $H$, $z_0 \in (Z_{\wh G}^{\Gamma})^0$. Since this a complex torus, there then exists $t_\eta \in Z_{\wh G}^0$ such that $t_\eta^n = z_0$, so taking $n$th roots, $|\alpha_{i,\mu}(\nu)| = |\nu(t_\eta)|$. Summing over $i$ then produces the bound on the $\alpha_\mu(\nu)$. 

To get the central character transformation, $\zeta \in X_*(Y_G)$ if and only if it is a $\Gamma_G$ and $\Om_G$-invariant element of $X_*(T) = X^*(\wh T)$. Such characters lift to $\Gamma$-invariant characters of $\wh G$ and therefore characters on $\Ld G$. For such $\zeta$, $\pi_{\mu + \zeta, 1} = \zeta \otimes \pi_{\mu, 1}$ so
\[
b_\eta(\pi_{\mu + \zeta, 1}) = b_\eta(\zeta) \otimes b_\eta(\pi_{\mu, 1}) = \zeta(c_\Frob) \zeta|_{\wh H} b_\eta(\pi_{\mu, 1}).
\]
Since $c_\mu(\nu)$ is $0$ unless $\mu$ and $\nu$ have the same central character and since $c_{\mu+\zeta}(\nu+\zeta) = c_\mu(\nu)$, this implies that $\alpha_{\mu + \zeta}(\nu+\zeta) = \zeta(c_\Frob)\alpha_\mu(\nu)$. Therefore, we are done if all the choices defining $t_\eta$ above are such that $t_\eta$ has the same image in $\wh G_\ab$ as $c_\Frob$. 
\end{proof}
The lemma used in this proof follows:
\begin{lem}
Let $V_\nu$ for $\nu \in X_*(A)$ be a weight space for $\pi^{\Ld G}_{\mu, \alpha}$ for $\mu \in X_*(A)$. Then $1 \rtimes \Frob$ acts as multiplication by $\alpha$ on $V_\nu$. 
\end{lem}

\begin{proof}
For any $\gamma \in W_F^\ur$, the trace of $\gamma$ acting on $V_\nu$ is the coefficient of $\nu$ in $\tr \pi^{\Ld G}_{\mu, \alpha}$ restricted to $\wh T \rtimes \gamma$. Let $n$ be the splitting degree of $G$. The same computation as lemma \ref{ldreptrace} gives that this is $\alpha^{ni+1} \dim V_\nu$ for any $\gamma = \Frob^{ni + 1}$. The only representation of $W_F^\ur \cong \Z$ with these traces sends $1$ to scaling by $\alpha$. 
\end{proof}

The element $t_\eta$ defines a function $\chi_\eta^{-1}$ on $G$ by $K \lb(\varpi) K \mapsto \lb(t_\eta)$ for $\lb \in X_*(A)$. Since $t_\eta$ is central, if $Q$ is the map on $X_*(A)$ summing over $\Om_G$-orbits, this is constant on fibers of $Q$. In particular, since products of basis elements $\tau_\lb^G \in \ms H(G,K)$ are a linear combination of $\tau_{\lb'}^G$ for $\lb'$ in a single fiber, $\chi$ is a character of $G$. This is the character that corresponds to $t_\eta$ considered as a Weyl-orbit in $\wh A$ through the Satake isomorphism. 

Furthermore, the relation $\alpha_{\mu + \zeta}(\nu+\zeta) = \zeta(t_\eta)\alpha_\mu(\nu)$ forces $\chi_\eta$ to be the character associated to $\eta$ through transfer factors as in section \ref{transferchar}. This all finally gives that the character on $H$ determined by $K_H \lb(\varpi) K_H \mapsto \lb(t_\eta)$ for $\lb \in X_*(A_H)$ is the same as the one from transfer factors. 

In summary, we get
\[
(\tau^G_\lb)^H =  \delta_H^G(\lb) \tau^H_\lb +  \sum_{\substack{\xi \in X^*(\wh A_H) \\ 0 \leq \xi < \lb}} a_\lb(\xi) \tau^H_\xi
\]
where
\[
a_\lb(\xi) = \sum_{\substack{\mu \in X^*(\wh A) \\ \nu \in X^*(\wh A_H) \\ \xi \leq_H \nu \leq_H \mu \leq_G \lb}} \alpha_\mu(\nu) b^G_\lb(\mu) c_\mu(\nu) q^{-\langle \xi, \rho_H \rangle } d^H_\nu(\xi),
\]
setting terms of the form $*_\mu(\mu) = 1$ here for ease of indexing. We also know that the $\alpha_\mu(\nu)$ can be bounded in terms of the character on $H$ determined by $\eta$. 

Going back to the global context, this finally allows us to compute:

\begin{prop}\label{valuebound}
Let $G$ be a reductive group over a global field and $(H, \mc H, \eta, s)$ an endoscopic quadruple that has a trivial $z$-extension. Let $S$ be a finite set of places $v$ such that:
\begin{itemize}
\item
$G_v,H_v$ are unramified.
\item
$|k_v|$ does not divide $|\Om_G|$.
\end{itemize} 
Let $\chi_{\eta, S}$ be the product of the characters $\chi_{\eta,v}$ on $H_v$ for $v \in S$ determined by $\eta$. 

If $f \in \ms H(G(F_S), K_S)^{\leq \kappa}$ with $\|f\|_\infty \leq 1$, we can take $f^H \in \ms H(H(F_S), K_S)^{\leq \kappa}$ such that $\|\chi_{\eta, S} f^H_S\|_\infty = O(q_{S_1}^{E \kappa} \kappa^{C|S|})$ for constants $C,E$ independent of $f_S$ and $q_S$. In addition, $E$ can be chosen uniformly over all $G$ in endoscopic paths from a fixed $G'$. 
\end{prop}

\begin{proof}
Use the notation from the previous discussion. For $s \in S$, $f_s$ is then a linear combination of some of $\tau^G_\lb$. If $\tau^G_\lb$ has a $\tau^H_\xi$ component then $\lb-\xi$ is in particular a non-negative sum of roots of $G$. The number of such $\lb$ is polynomial in $\kappa$. Therefore, if $f^H_s$ is written as a linear combination of $\tau^H_\xi$, the coefficient for $\tau^H_\xi$ is bounded by a sum of polynomially many $a_\lb(\xi)$. Furthermore, all these $\xi$ are smaller than $\lb$. 

Moving to what we are actually bounding, if $t_\eta$ is as in the previous discussion, the corresponding coefficient in $\chi_{\eta, s}^{-1} f^H_s$ is bounded by a sum of polynomially many $ \xi(t_\eta)^{-1} a_\lb(\xi)$. For all $\alpha_\mu(\nu)$ appearing in the sum defining $a_\lb(\xi)$,
\[
|\xi(t_\eta)^{-1}\alpha_\mu(\nu)| \leq |\xi(t_\eta)^{-1}\nu(t_\eta)| = 1 
\]
since $\xi$ and $\nu$ have the same $\Om_G$-orbit sum. In particular, if we define
\[
a'_\lb(\xi) = \sum_{\substack{\mu \in X^*(\wh A) \\ \nu \in X^*(\wh A_H) \\ \xi \leq_H \nu \leq_H \mu \leq_G \lb}} b^G_\lb(\mu) c_\mu(\nu) q^{-\langle \xi, \rho_H \rangle } d^H_\nu(\xi),
\]
then $| \xi(t_\eta)^{-1} a_\lb(\xi)| \leq | a'_\lb(\xi)|$. 

It remains to bound the polynomially many summands in $a'_\lb(\xi)$. Bounding each of these terms, the $c_\mu(\nu)$ are polynomial in how big $\mu$ is. By lemma \ref{dmu}, the term
\[
b^G_\lb(\mu)q^{-\langle \xi, \rho_H \rangle } d^H_\nu(\xi)
\]
is a polynomial in the size of $\lb$ times a factor of $q^{-\langle \xi, \rho_H \rangle + \langle \lb, \rho_G \rangle }$. Therefore, we roughly bound the entire product, $a_\lb(\xi)$, by a polynomial in $\kappa$ times a factor of $q^{-\langle \lb, \rho_G \rangle}$ 

Finally, note that $\langle \lb, \rho_G \rangle \leq \rank_{\ssm}(G) \kappa$. Taking the product of $f^H_s$ over $s \in S$ and setting $E =  \rank_{\ssm}(G)$ gives the result. 
\end{proof}
Note that this lemma can be inductively applied through a hyperendoscopic path by letting $\chi$ at each step be the character defined from the hyperendoscopic path as in section \ref{hyptransferchar}.  

\subsubsection{General case}\label{zextarchtransfer}
Starting as in the Archimedean case argument in section \ref{zexttransfer}, consider $z$-pair $(H_1, \eta_1)$ for $H$. The extension $H_1$ induces an extension $G_1$ such that $H_1$ is an endoscopic group for $G_1$ by proposition \ref{zextpath}. If $\varphi : (G_1)_v \to G_v$ is the projection, we have that $f^{H_1} = (f \circ \varphi)^{H_1}$ for any $H_1$ on $G$ (interpreted as before).

If $H$ is ramified, then all $\kappa$-orbital integrals are still $0$ so this transfer is $0$.  

If $H$ is unramified, $T$ can be pulled back to a maximal torus $T_1$ of $G_1$ and $A$ can be pulled back to $A_1$. By lemma \ref{zextbound} the extending torus $Z$ is without loss of generality unramified so $G_1$ is too. As explained in \cite[\S7]{Kot86}, the reductive model of $G$ corresponding to the chosen hyperspecial $K_{G,v}$ gives a reductive model of $G_1$ so we can find a hyperspecial $K_{G_1,v}$ that surjects onto $K_{G,v}$. The map $\varphi$ induces $\varphi_*: X_*(A_1) \to X_*(A)$ so
\[
\varphi(K_{G_1, v} \lb(\varpi) K_{G_1,v}) = K_{G, v} \varphi_*\lb(v) K_{G,v}.
\]
Therefore,
\[
\tau^G_\lb \circ \varphi = \sum_{\lb' \in \varphi_*^{-1}(\lb)} \tau^{G_1}_{\lb'} 
\]
and the transfer can be computed by the fundamental lemma. 

We describe the transfer of $\tau^G_0$ as an example computation:
\begin{lem}\label{zextfunlemma}
Use the notation above. Then we can take
\[
(\tau^G_0)^{H_1} = \sum_{\lb \in X_*(A_Z)} \chi_{\eta_1}(\lb(\varpi)) \tau^{H_1}_\lb.
\]
Here $A_Z$ is the split part of the extending torus $Z$ and $\chi_{\eta_1}$ is the character on $Z_{G_1}$ determined by $\eta_1$. 
\end{lem}

Finally, we get an extension of proposition \ref{valuebound}: that transfers from $\ms H(G_v, K_v, \chi)$ land in $\ms H(H^1_v, K_{H,v}, \chi \chi_{\eta_1})$ with the same bound.

\red{\subsection{Transfer of Some Ramified Functions}
LEVEL ASPECT}

\subsection{Controlling Endoscopic Groups Appearing}
\begin{lem}\label{conditions}
Let $G$ be a reductive group over global field $F$ that is cuspidal at infinity together with central character datum $(\mf X, \chi)$ such that $\mf X$ contains $A_{G, \infty}$. Let $f = \eta_\xi \otimes f^\infty$ be a function on $G(\A)$ where $\eta_\xi$ is some EP-function with central character matching $\chi$. Let $R$ be a finite set of places containing those on which $f^\infty$ or $G$ are ramified. Then there are a finite number of elliptic endoscopic quadruples $(H, \mc H, \eta, s)$ up to equivalence for which $I_\disc(f^{H_1}) \neq 0$ for (all) $z$-extensions $H_1$. For each such $H_1$:
\begin{itemize}
\item
$H_1$ is cuspidal at infinity and $\mf X_{H_1}$ contains $A_{H_1, \infty}$. 
\item
$f^{H_1}$ is unramified outside of $R$ and $H_1$ can be chosen to be. 
\item
$\chi_{H_1}$ is unramified outisde of $R$.
\end{itemize}
\end{lem}

\begin{proof}
If $H_1$ is not cuspidal at infinity, then $I_\disc(g) = 0$ for any $g$ with infinite part that is a EP function by the previous section. By corollary \ref{eptransfer} and lemma \ref{etorii}, $f^H$ is either a linear combination of such functions or $0$. As before, we remark that $\mf X_{H_1} \supseteq A_{H_1, \infty}$ due to ellipticity.

If $H$ is ramified outside of $R$, then by the full fundamental lemma together with the trick to compute transfers on $z$-extensions, $f^{H_1} = 0$. Otherwise, by lemma \ref{zextbound}, $H_1$ can be chosen to be unramified outside $R$ so $f^{H_1}$ is unramified outside of $R$ by the full fundamental lemma again. The group $H_1$ being unramified outside of $R$ further implies that $\chi_{H_1}$ is too. 

Finiteness of the sum is implicit in the stabilization of the trace formula. Repeating the argument here, note that the roots of $H_{\overline K}$ are a subset of those of $G_{\overline K}$. Therefore, there are a finite number of possibilities for $H_{\overline K}$ and the splitting field of $H$ has degree $\leq \Om_G$.  Since the splitting field is also unramified outside of $R$, there are a finite number of choices for it. This leaves only a finite number of choices for $H$.

To get finitely many quadruples it then suffices to show there are finitely many choices for $s \in (Z_{\wh H}/Z_{\wh G})^{W_F}$. For this, $Z_H^{W_F}/Z_G^{W_F}$ is finite by ellipticity and $Z_H^{W_F}$ having finitely many connected components. Therefore $(Z_{\wh H}/Z_{\wh G})^{W_F}$ is finite by finiteness of a cohomology group. 
\end{proof}
Note that this lemma can be inductively applied through a hyperendoscopic path.

\section{Simple Trace Formula with Central Character}\label{straceformsection}
\subsection{Set-up}
To apply the hyperendoscopy formula, we will need two generalizations of the simple trace formula: first, allowing central characters and second, allowing pseudocoefficients at infinite places on the spectral side.  We use a slightly convoluted and indirect argument to avoid having to go into too many technicalities of Arthur's distributions $I(f, \gamma)$ and $I(f, \pi)$:

Fix central character datum $(\mf X, \chi)$ and let $\chi_0$ be the restriction of $\chi$ to $A_{G, \rat}$. We first define a variant of $I_{\disc, \chi}$ that can be more easily related to $I_{\geom, \chi_0}$. Let $\mf X_F = \mf X \cap Z(F)$. There is a map
\[
\ms H(G, \chi_0) \to \ms H(G, \chi) : f(g) \mapsto \bar f_\chi(g) := \int_{\mf X/A_{G, \rat}} f(gz) \chi(z) \, dz.
\] 
\begin{lem}\label{functiontruncation}
$f \mapsto \bar f_\chi$ is surjective.
\end{lem}

\begin{proof}
Let $h \in \ms H(G, \chi)$. There exists compact $U \subseteq G(\A)/A_{G, \rat}$ such that $U\mf X$ contains the support of $h$. Let $c$ be a cutoff function: compactly supported, continuous, non-negative real valued, and positive on $U$. Then the function 
\[
m(g) = \int_{\mf X/A_{G, \rat}} c(gz) \, dz
\]
is continuous and non-zero on the support of $h$. If we take $f = m^{-1}ch$, then $\bar f_\chi = h$. 
\end{proof}

We follow a strategy from \cite{KSZ}. For any $\star \in \{\geom, \disc, \spec\}$, also define distributions on $\ms H(G, \chi_0)$:
\[
I'_{\star, \chi}(f) = \f1{\vol(\mf X_F \bs \mf X/ A_{G, \rat})} \int_{\mf X_\Q \bs \mf X/ A_{G, \rat}} \chi(z) I_{\star, \chi_0} (f_z) \, dz
\]
where $f_z : g \mapsto f(gz)$. We of course have that
\[
I'_{\geom, \chi} = I'_{\spec, \chi}.
\]
In addition, if $f$ is cuspidal, then so is $f_z$ for any central $z$ so 
\[
I'_{\spec, \chi}(f) = I'_{\disc, \chi}(f).
\]

For our case, we can only consider central character datum where $A_{G, \infty} \subseteq \mf X$. Fix $(\mf X, \chi)$ for the rest of this section and let $\chi_0$ be the restriction of $\chi$ to $A_{G, \rat}$. The generalized simple trace formula can then be developed in three steps:
\begin{enumerate}
\item
Find a generalized pseudocoefficient $\varphi$ so that $\bar \varphi_\chi$ is the pseudocoefficient $\varphi_\pi$ and traces against $\varphi$ can be computed easily
\item
Compute $I'_{\spec, \chi}(\varphi \otimes f^\infty)$ and show this equals $I_{\spec, \chi}(\varphi_\pi \otimes \overline{(f^\infty)}_\chi)$. Both these are small modifications of Arthur's original spectral side argument together with an extra lemma of Vogan. 
\item
Sum over $\varphi$ to get a generalized Euler-Poincar\'e function $\eta$. Evaluate $I_{\geom, \chi_0}(\eta \otimes f^\infty)$ and average to get a formula for $I'_{\geom, \chi}(\eta \otimes f^\infty)$.
\end{enumerate}
To see how everything depends on Haar measures, $\varphi$ will have dimension $[G_\infty/A_{G, \rat}]^{-1}$ and $f^\infty$ will have dimension $[\mf X^\infty]^{-1}$ so that both sides of our final formula will have dimension $[G^\infty][\mf X^\infty]^{-1}$. 

\subsection{Generalized Pseudocoefficients}\label{genps}
We first need to define a version of truncated/generalized pseudocoefficients from \cite[\S1.9]{HL04} in the real case. This actually can be done slightly more explicitly than the $p$-adic case. A lot of this section is probably implicit somewhere in \cite{CD90}.

For this section only, let $G = G(\R)$ be a group over $\R$ with discrete series mod center. All other variables ($\mf a$, $A_G$, etc.) will refer to real versions. There is a map
\[
H_G : G(\R) \to \mf a_*^G : \lb(H_G(\gamma)) = \log |\lb(\gamma)| \text{ for all } \lb \in \mf a^*_G.
\]
It is well known that this maps $A^0 = A_G(\R)^0$ isomorphically to $a_*^G$ so since $A^0$ is central, we get a splitting $G(\R) = G(\R)^1 \times A^0$, where $G(\R^1)$ is the kernel of $H_G$. 

 Any character $\lb \in (\mf a^*_G)_\C$ of $\mf a_*^G$ corresponds to the character $e^{\lb(H_G(\gamma))}$ on $A^0$ and therefore $G$ through this isomorphism. The unitary characters correspond to $\lb \in i\mf a^*_G$. Finally, if $\pi$ is a representation of $G(\R)$, let $\pi_\lb = \pi \otimes e^{\lb(H_G(\gamma))}$. 
 
 Let $f$ be any smooth, compactly supported function on $\mf a_*^G$ and $\pi$ a discrete series representation. The main theorem \cite{CD90} also allows us to construct a (again not-necessarily unique) compactly supported $\varphi_{\pi,f}$ such that for any unitary $\rho$
 \[
 \tr_\rho(\varphi_{\pi,f}) = \begin{cases} \wh f(\lb) & \rho = \pi_\lb \text{ for some } \lb \in (\mf a^*_G)_\C\\0 & \rho \text{ basic}, \rho \neq \pi_\lb \text{ for all } \lb \in (\mf a^*_G)_\C \\ ? & \text{else} \end{cases}.
 \]
Call such a $\varphi_{\pi,f}$ a generalized pseudocoefficient. For any character $\om$ on $A^0$, we can define
\[
\varphi_{\pi, f, \om}(g) = \int_{A^0} \om(a) \varphi_{\pi,f}(ag) \, da.
\]
This is compactly supported mod center and transforms according to $\om^{-1}$ on $A^0$. Therefore, if $\rho$ has character $\om$ on $A^0$, we can define
\begin{multline}\label{genpstops}
\tr_\rho(\varphi_{\pi, f, \om}) = \int_{G/A^0} \varphi_{f, \pi, \om}(g) \Theta_\rho(g) \, dg = \int_{G/A^0} \int_{A^0} \varphi_{\pi, f}(ag) \om(a) \Theta_\rho(g) \, da \, dg \\
  = \int_{G/A^0} \int_{A^0} \varphi_{\pi, f}(ag) \Theta_\rho(ag) \, da \, dg = \int_G \varphi_{f, \pi}(g) \Theta_\rho(g) \, dg = \tr_\rho(\varphi_{\pi, f}).
\end{multline}
where $\Theta$ is the Harish-Chandra trace character. In particular, $\varphi_{\pi, f, \om}$ appropriately scaled is a pseudocoefficient. 

Averaging $\varphi_{\pi, f}$ over an $L$-packet $\Pi_\disc(\tau)$ for fixed $f$ produces a generalized Euler-Poincar\'e function $\eta_{\tau,f}$. Since the $\eta_{\tau, f, \om}$ are averages of pseudocoeffecients over $L$-packets, they are actually standard Euler-Poincar\'e functions. Therefore, computation \eqref{genpstops} gives that whenever $\tau$ is regular:
\[
\tr_\rho(\eta_{\tau, f}) = \begin{cases} \wh f(\lb)|\Pi_\disc(\tau)|^{-1} & \rho = \pi_\lb \text{ for some } \pi \in \Pi_\disc(\tau), \lb \in (\mf a^*_G)_\C \\0 & \text{else} \end{cases}.
\]
Generalized pseudocoefficients and Euler-Poincar\'e functions are cuspidal for the same reason as the normal versions.

Finally, as a useful lemma relating our notion to the one in \cite{HL04}, 
\begin{lem}
Let $\pi$ be a discrete series representation with character $e^{\lb(H_G(a))}$ on $A^0$ for $\lb \in (\mf a^*_G)_\C$. Let $f$ on $\mf a_*^G$ be smooth and compactly supported. Then we can make choices for $\varphi_{\pi}$ and $\varphi_{\pi, f}$ such that $\varphi_{\pi,f} = f\varphi_\pi$. 
\end{lem}

\begin{proof}
Make a preliminary choice for $\varphi_{\pi,f}$. Then $\wh f(0)^{-1} \varphi_{\pi, f, \lb}$ is a valid choice of $\varphi_\pi$ We evaluate
\begin{align*}
\tr_\rho(f \varphi_{\pi, f, \lb}) &= \int_G f(g) \varphi_{\pi, f, \lb}(g) \Theta_\rho(g) \, dg \\
&= \int_{A^0} \int_{G/A_0} f(ag) \varphi_{\pi, f, \lb}(ag) \Theta_\rho(ag) \, dg \, da \\
&=  \int_{A^0} f(a) e^{(\mu - \lb)(H_G(a))} \int_{G/A_0} \varphi_{\pi, f, \lb}(g) \Theta_{\rho_{\lb - \mu}}(g) \, dg \, da
\end{align*}
where we choose $\mu \in (\mf a^*_G)_\C$ so that $e^{\mu(H_G(g))}$ is the central character of $\rho$ on $A^0$. By previous properties, the inner integral therefore becomes $\tr_{\rho_{\lb-\mu}}(\varphi_{\pi, f})$ and we get
\[
\tr_\rho(f \varphi_{\pi, f, \lb}) = \wh f(\mu - \lb) \tr_{\rho_{\lb-\mu}}(\varphi_{\pi, f}).
\]
Checking each of the three cases in its definition, $f \wh f(0)^{-1} \varphi_{\pi, f, \lb}$ is then a valid alternative choice for $\varphi_{\pi, f}$. 
\end{proof}
A similar property also therefore holds for Euler-Poincar\'e functions. 

\subsubsection{A small modification}
Generalized pseudocoefficients are in $C_c^\infty(G_\infty)$. We instead want functions in some $C_c^\infty(G_\infty, \chi_0)$ so we make a small modification. 

Return to the previous notation where $G$ is a group over $F$. Let $\chi_0$ be a character on $A_{G, \rat}$ and $\pi_0$ a representation of $G_\infty$ consistent with $\chi_0$. Let $\varphi_{\pi_0, f} = f \varphi_{\pi_0}$ be a generalized pseudocoefficient for $\pi_0$ and consider the partial average
\begin{align*}
\bar \varphi(g) &= \int_{A_{G, \rat}} \chi_0(a) f(ag) \varphi_{\pi_0}(ag) \, da \\
&= \int_{A_{G, \rat}} \chi_0(a) f(ag) \chi_0^{-1}(a) \varphi_{\pi_0}(g) \, da = \varphi_{\pi_0}(g) \int_{A_{G, \rat}} f(ag) \, da.
\end{align*}
This is an element of $C_c^\infty(G_\infty, \chi_0)$ and every function $f \in C_c^\infty(A_{G, \infty}/A_{G, \rat})$ arises as an integral this way. Finally, by a similar computation to \eqref{genpstops}, this has the same traces against representations $\pi$ consistent with $\chi_0$ as $\varphi_{\pi_0, f}$.

Therefore, for any function $f \in C_c^\infty(A_{G, \infty}/A_{G, \rat})$, we can construct analogues of generalized pseudocoefficients $\varphi_{\pi_0, f} = f\varphi_{\pi_0} \in C_c^\infty(G_\infty, \chi_0)$. For computations later, note that such $f$ have Fourier transforms defined on any character of $A_{G, \infty}$ trivial on $A_{G, \rat}$. The same discussion carries over to Euler-Poincar\'e functions. These are the functions we will actually be using. 

We fix $f$ to be dimensionless so these generalized pseudocoefficients have dimension $[G_\infty/A_{G, \rat}]^{-1}[A_{G, \infty}/A_{G, \rat}] = [G_\infty]^{-1}[A_{G, \infty}]$.

\subsection{Spectral Side with Central Character}
To get a simple trace formula with central character, we need two spectral side computations: one for $I'_\spec$ and one for $I_\spec$. Start with a lemma:
\begin{lem}\label{pstrace}
Let $\pi_0$ be a regular discrete series representation of $G_\infty$ with weight $\xi_0$ and character $\chi_0$ on $A_{G, \infty}$. Then for any real irreducible representation $\rho$ of $G_\infty$ with character $\chi_0$ on $A_{G, \infty}$, $\tr_{\rho}(\varphi_{\pi_0}) = \delta_{\pi_0}(\rho)$. 
\end{lem}

\begin{proof}
We thank David Vogan for this argument and note that all mistakes in this writeup are our own. 

The case $\rho = \pi_0$ follows immediately. Consider $\rho \neq \pi_0$. In the Grothendieck group, $\rho$ is a linear combination of basic representations with infinitesimal character matching $\pi_0$:
\[
\rho = \sum_{\rho' \text{ basic}} m_\rho(\rho') \rho'.
\]
Taking traces of both sides, $\tr_\rho(\varphi_{\pi_0}) = m_\rho(\pi_0)$. Now, taking the trace against an EP-function $\eta_{\xi_0}$:
\[
0 = \tr_\rho(\eta_{\xi_0}) =  \f1{|\Pi_\disc(\xi_0)|}\sum_{\rho' \in \Pi_\disc(\xi_0)} m_\rho(\rho')
\]
where $\Pi_\disc(\xi_0)$ is the $L$-packet for $\xi_0$. It therefore suffices to show that the $m_\rho(\rho')$ for $\rho' \in \Pi_\disc(\xi_0)$ all have the same sign. This would force them all to be $0$.  

The most direct way is to use the classification of all unitary representations with infinitesimal character of a discrete series from \cite{Sal99}. These are of the form of certain $A_\mf q(\lb)$ described in terms of Zuckerman functors. These have an explicit decomposition in the Grothendieck group through a version of Zuckerman's character formula proposition 9.4.16 in \cite{Vog81}: $\lb$ is a character on Levi $L_\infty$ so first get a character formula $\lb$ by twisting both sides of 9.4.16 for $L_\infty$ by $\lb$. Then cohomologically induce to get a character formula on $G_\infty$. Alternatively, by Kazhdan-Lusztig theory, the $m_\rho(\rho')$ are Euler characteristics of stalks of certain perverse sheaves. By theorem 1.12 in \cite{LV83} their cohomologies are either concentrated in even degree or odd degree. See the comments in the proof to corollary 4.6 in \cite{Vir15}, for example, for why this applies to $\C$ in addition to $\overline \F_p$. 
\end{proof}
Combining with computation \eqref{genpstops} (note that twisting by a character does not change the regularity of the discrete series) then gives:
\begin{cor}\label{genpstrace}
Let $\pi_0$ be a regular discrete series representation of $G_\infty$ with weight $\xi_0$. Let $f \in C_c^\infty(A_{G, \infty})$.  Then for any real representation $\rho$ of $G_\infty$, $\tr_{\rho}(\varphi_{\pi_0,f}) = f(\rho, \pi_0)$ where
\[
f(\pi, \pi_0) = \begin{cases} \wh f(\lb) & \pi = \pi_\lb \text{ for some } \lb \in \mf (a^*_{G_\infty})_\C \\ 0 & \text{else} \end{cases}.
\]
\end{cor}
A similar result holds for $f \in C_c^\infty(A_{G, \infty}/A_{G, \rat})$. 

This allows us to prove
\begin{prop}\label{ssidebig}
Let $\pi_0$ be a regular discrete series representation of $G_\infty$ with weight $\xi_0$ and character $\chi_0$ on $A_{G, \rat}$. Let $f \in C_c^\infty(A_{G, \infty}/A_{G, \rat})$. Then for all $\varphi^\infty \in \ms H(G^\infty)$:
\[
I^G_{\spec}(\varphi_{\pi_0, f} \otimes \varphi^\infty) = \sum_{\pi \in \mc{AR}_\disc(G, \chi_0)} m_\disc(\pi) f(\pi_\infty, \pi_0) \tr_{\pi^\infty}(\varphi^\infty)
\]
where
\[
f(\pi_\infty, \pi_0) = \begin{cases} \wh f(\lb) & \pi_\infty = (\pi_0)_\lb \text{ for some } \lb \in \mf (a^*_{G_\infty})_\C \\ 0 & \text{else} \end{cases}.
\]
\end{prop}

\begin{proof}
This is simply a due-diligence check that none of the steps in the derivation of formula 3.5 in \cite{Art89} break. First, $\varphi_{\pi_0,f}$ being cuspidal gives 
\begin{align*}
I^G_\spec(\varphi) &= \sum_{t \geq 0} I^G_{\disc, t}(\varphi)\\
&=\sum_{t \geq 0} \sum_{L \in \ms L(G)} \f{|\Om_M|}{|\Om_G|} \\
& \qquad \sum_{s \in W^G(\mf a_L)_\reg} |\det(s-1)|_{\mf a_L/\mf a_G}|^{-1} \tr(M_{Q|Q}(s,0) \rho_{Q, t}(0, (\varphi_{\pi_k} \varphi^\infty)^1)),
\end{align*}
using that $G$ is connected. This uses a lot of the notation from \cite{Art89}. In particular, $\ms L(G)$ is the set of Levi subgroups of $G$, $Q$ is a parabolic for $L$, $M_{Q|Q}(s,0)$ is some intertwining operator, $\rho_{Q,t}$ is a sum of parabolically-induced representations from $Q$ with Archimedean infinitesimal character having imaginary part of norm $t$, and $(\varphi_\pi \varphi^\infty)^1$ is the restriction of the function to $G(\A)^1$. 

The full definition of the rest of the terms in the inner sum is unnecessary: the only detail Arthur uses is that when $Q \neq G$ it is a sum
\[
\sum_{\pi \in \mc{AR}(G)} c_\pi \tr_\pi((\varphi_{\pi_0, f} \varphi^\infty)^1)
\]
where the $c_\pi$ vanish whenever the Archimedean infintesimal character of $\pi$ is regular. However, a property of the pseudocoefficient $\varphi_{\pi_0,f}$ is that it is only supported on representations which have the same infinitesimal character as ${\pi_0}$ (similar to the the proof of \cite{Clo86} lemma 1). This character minus $\rho$ has to be regular. Therefore the sum is $0$.

For the leftover term, $Q=G$ so $L=G$ and $M_{Q|Q}(s,0)$ is trivial. This gives
\[
I^G_\disc(\varphi_{\pi_0,f} \otimes \varphi^\infty) = \sum_{t \geq 0} \tr \rho_{G, t}(0, (\varphi_{\pi_0,f} \varphi^\infty)^1).
\]
By its definition, $\rho_{G,t}(0)$ is the sum of all irreducible, discrete subrepresentations of the space $L^2(G(\Q) \bs G(\A)^1)$ with Archimedean infinitesimal character having imaginary part with norm $t$. Arthur's original argument for the sum over discrete representations converging absolutely does not work since there are now potentially infinitely many $t$ on which this trace is supported. However, absolute convergence is now known in general by \cite{FLM11}. 

Finally, $(\varphi_{\pi_0,f} \varphi^\infty)^1$ acting on $L^2(G(\Q) \bs G(\A)^1)$ is the same operator as $\varphi_{\pi_0,f} \varphi^\infty$ acting on $L^2(G(\Q) \bs G(\A), \chi_0)$. Therefore, summing over the representations that are actually subrepresentations of $L^2$,
\begin{align*}
I^G_\disc(\varphi_{\pi_0,f} \otimes \varphi^\infty) &= \sum_{\pi \in \mc{AR}_\disc(G, \chi_0)} m_\disc(\pi) \tr_\pi(\varphi_{\pi_0,f} \varphi^\infty) \\
&=  \sum_{\pi \in \mc{AR}_\disc(G, \chi_0)} m_\disc(\pi) \tr_{\pi_\infty}(\varphi_{\pi_0,f})\tr_{\pi^\infty}(\varphi^\infty).
\end{align*}
Corollary \ref{genpstrace} gives that $\tr_{\pi_\infty}(\varphi_{\pi_0,f}) = f(\pi_\infty, \pi_0)$ finishing the argument.
\end{proof}

Next, let $\varphi = \varphi_{\pi_0,f} \otimes \varphi^\infty$. Then
\[
I'_{\spec, \chi}(\varphi) = \f1{\vol(\mf X_F \bs \mf X/ A_{G, \rat})} \int_{\mf X_F \bs \mf X/ A_{G, \rat}} \chi(z) I_{\spec, \chi_0} (\varphi_z) \, dz.
\]
Computing
\begin{align*}
 I_{\spec, \chi_0} (\varphi_z) &= \sum_{\pi \in \mc{AR}_\disc(G, \chi_0)} m_\disc(\pi) \tr_{\pi}(\varphi_z) \\
 &= \sum_{\pi \in \mc{AR}_\disc(G, \chi_0)} m_\disc(\pi) \om_\pi^{-1}(z) \tr_\pi(\varphi)
\end{align*}
where $\om_\pi$ is the central character of $\pi$. Substituting this in and factoring out the sum and constants from the integral gives
\[
\int_{\mf X_F \bs \mf X/ A_{G, \rat}} \chi(z) \om^{-1}_\pi(z) \, dz = \begin{cases} \vol(\mf X_F \bs \mf X/ A_{G, \rat})  & \chi = \om_\pi|_\mf X \\ 0 & \text{else} \end{cases}.
\]
Therefore, a lot of terms in the sum go to $0$. Finally, since $\pi^\infty$ has central character $\chi^\infty$, it can be traced against functions in $\ms H(G^\infty, \chi^\infty)$. By definition
\[
\tr_{\pi^\infty}(\varphi^\infty) = \tr_{\pi^\infty}(\overline{(\varphi^\infty)}_{\chi^\infty}).
\]
Putting it all together,
\begin{cor}
Let $\pi_0$ be a regular discrete series representation of $G_\infty$ with weight $\xi_0$ and character $\chi_0$ on $A_{G, \rat}$. Let $f \in C_c^\infty(A_{G, \infty}, \chi_0)$. Then for all $\varphi^\infty \in \ms H(G^\infty)$:
\[
I'_{\spec, \chi}(\varphi_{\pi_0,f} \otimes \varphi^\infty) = \sum_{\pi \in \mc{AR}_\disc(G, \chi)} m_\disc(\pi) f(\pi_\infty, \pi_0) \tr_{\pi^\infty}(\overline{(\varphi^\infty)}_{\chi^\infty})
\]
(where we only sum over automorphic representations with the correct central character on all of $\mf X$ instead of just $A_{G, \rat}$).
\end{cor}

Finally, the same arguments as in \autoref{ssidebig} again work for the terms in equation \eqref{Idisc} giving that for $\varphi^\infty \in \ms H(G^\infty, \chi^\infty)$,
\[
I_{\spec, \chi}(\varphi_{\pi_0} \otimes \varphi^\infty) = \f1{\vol(\mf X^1_\infty)} \sum_{\pi \in \mc{AR}_\disc(G, \chi)} m_\disc(\pi) \delta_{\pi_0,\pi_\infty} \tr_{\pi^\infty}(\varphi^\infty)
\]
where we factor $\mf X_\infty = \mf X^1_\infty \times A_{G, \infty}$. Sanity checking dimensions here, we need
\[
[G(\A)][\mf X]^{-1} [G_\infty]^{-1} [A_{G, \infty}] = [\mf X_\infty/A_{G, \infty}]^{-1} [G^\infty][\mf X^\infty]
\]
which holds. 

Putting everything together:
\begin{prop}\label{ssidegenps}
Let $\pi_0$ be a regular discrete series representation of $G_\infty$ with weight $\xi_0$ and that matches character $\chi$ on $\mf X$. Let $f \in C_c^\infty(A_{G, \infty}/A_{G, \rat})$ and $\varphi^{\infty_1} \in \ms H(G^\infty, \chi^\infty)$ such that $\overline{(\varphi^{\infty_1})}_\chi = \varphi^\infty$. Then:
\begin{multline*}
\vol(\mf X^1_\infty) I_{\spec, \chi}(\varphi_{\pi_0} \otimes \varphi^\infty) =   \\
\sum_{\pi \in \mc{AR}_\disc(G, \chi)} m_\disc(\pi) \delta_{\pi_0,\pi_\infty} \tr_{\pi^\infty}(\varphi^\infty) \\
=\f1{\wh f(0)}I'_{\spec, \chi}(\varphi_{\pi_0,f} \otimes \varphi^{\infty_1}).
\end{multline*}
\end{prop}
The second equality uses that for any $\pi_\lb \in \mc{AR}_\disc(G, \chi)$, $\lb = 0$. We fix $\varphi^\infty$ to be dimensionless and normalize $\varphi^{\infty_1}$ by it. Therefore, the dimensions are all $[G^\infty][\mf X^\infty]^{-1}$.

\subsection{Geometric Side with Central Character}
\subsubsection{Vanishing of $I^G_{M, \infty}(\gamma, \psi)$}\label{cuspvanishing}
We explicitly describe all the implicit vanishing arguments in \cite{Art89} for the ease of the reader. Assume $\psi$ is some cuspidal function. First, by lemma \ref{centertechnicality}, $I^G_M$ vanishes unless $M \in \ms L^\cusp$; i.e., unless $A_{M, \rat}/A_{G, \rat} = A_{M, \infty}/A_{G, \infty}$. Furthermore, in this case, $I^G_{M, \infty}(\gamma, \psi) = \td I^G_{M, \infty}(\gamma, \psi)$. 

Furthermore, as explained in the summary \cite[\S24]{Art05}, unless $\gamma$ is elliptic in $M$ over $\infty$, it is contained in a smaller Levi at $\infty$, so the descent formula to the smaller Levi shows that $\td I^G_{M, \infty}(\gamma, \psi)$ vanishes. The main result of \cite{Art76} also gives this.

\subsubsection{Computation of $I_{\geom, \chi_0}$}
Next, we compute the geometric side. Let $\Pi_\disc(\lb)$ be a regular discrete series $L$-packet for $G_\infty$ consistent with $\chi$ and $f \in C_c^\infty(A_{G, \infty}/A_{G, \rat})$. We again try to mimic Arthur's arguments. Cuspidality of $\eta_{\lb, f}$ and the splitting formulas reduce the geometric side to 
\[
I_{\geom, \chi_0}(\eta_{\lb, f} \otimes \varphi^\infty) = \sum_{M \in \ms L} \f{|\Om_M|}{|\Om_G|} \sum_{\gamma \in [M(\Q)]_{M,S}} a^M(S, \gamma) I^G_M(\gamma_\R, \eta_{\lb,f}) O^M_\gamma(\varphi^\infty_M).
\]
Define for $\psi \in C_c^\infty(G_\infty, \chi)$:
\[
\Phi_M(\gamma_\R, \psi) = |D^M(\gamma)|^{-1/2} \td I^G_M(\gamma, \psi),
\]
By the previous subsubsection, we can without loss of generality set $\Phi_M(\gamma, \psi) = 0$ if $M$ is not cuspidal over $\R$

For $L$-packet $\Pi_\disc(\lb)$, and elliptic regular $\gamma \in M_\infty$,
\[
\Phi_M(\gamma, \lb) = (-1)^{q(G)} |D^G_M|^{1/2} \sum_{\pi \in \Pi_\disc(\lb)} \Theta_\pi(\gamma).
\]
Arthur shows that $\Phi_M(\gamma, \lb)$ can be extended by continuity to all elements in elliptic maximal tori. Define it to be $0$ for other elements to extend it to all of $M_\infty$; in particular, to non-semisimple elements. 

Next, we need a defintion
\begin{dfn}
Let $\chi$ be a character on $A_{G, \infty}$. A cuspidal function $\psi \in C_c^\infty(G_\infty, \chi)$ is \emph{stable cuspidal} if its trace is supported on discrete series and constant on $L$-packets.
\end{dfn}
Note that Euler-Poincar\'e functions are stable cuspidal. Part of the main result of \cite{CD90} gives that Euler-Poincar\'e functions are also $K$-finite. 

As some notation for the next step, if $H$ is a reductive group over $\R$, let $\overline H$ be the compact form of $H$.  Any Haar measure on $H$ comes from a differential form on $H_\C$ and therefore induces a Haar measure on $\overline H$. Then:
\begin{thm}[{\cite[thm 5.1]{Art89} slightly rephrased}]
Let $\chi$ be a character on $A_{G, \infty}$ and $\varphi \in C_c^\infty(G_\infty, \chi)$ be stable cuspidal and $K$-finite. Then for any $\gamma \in M_\infty$
\[
\Phi_M(\gamma, \varphi) = (-1)^{\dim(A_M/A_G)} \nu(I^M_\gamma)^{-1} \sum_{\substack{\lb \chi^{-1} \in X^*_\C(T) \\ \lb \text{ matches } \chi}}  \Phi_M(\gamma, \lb) \tr_{\lb^\vee}(\varphi):
\]
where $\nu(M_\gamma) = (-1)^{q(G)}\vol(\bar I^M_{\gamma,\infty}/A_{I^M_\gamma,\infty}) |\Om(B_{K_{I^M_{\gamma,\infty}}})|^{-1}$.
\end{thm}
Note that there is a correction here changing $A_{I^M_\gamma,\rat}$ to $A_{I^M_\gamma,\infty}$ and using $\td I^G_M$ instead of $I^G_M$. (see the end of \cite[\S7]{GKM97}).   

Since lemma \ref{genpstops} gives that without loss of generality, $\eta_{\lb,f} = f\eta_\lb$, we recall the following rephrasing of a fact used in deriving the invariant trace formula:
\begin{lem}
Let $f = f_1 \circ H_{G_\infty}$ be a function on $G_\infty/A_{G, \rat}$ where $f_1$ is a function on $C_c^\infty(A_{G, \infty}/A_{G, \rat})$. Let $\varphi$ be any function on $G_\infty$ compactly supported mod center. Then for any $\gamma \in G_\infty$ and Levi $M$
\[
\td I^G_M(\gamma, f \varphi) = f(\gamma) \td I^G_M(\varphi).
\]
\end{lem}

\begin{proof}
Remark 4 after theorems 23.2 and 23.3 in \cite{Art05} gives that $\td I^G_M(\gamma, f \varphi)$ only depends on the values of $f\varphi$ on $g \in G_\infty$ with the same image as $\gamma$ under $H_{G_\infty}$. On this set $f$ is constant so the result follows. 
\end{proof}

In particular,  keeping in mind our normalization for EP-functions, for any $\gamma \in G_\infty$:
\[
|\Pi_\disc(\lb)|\Phi_M(\gamma, \eta_{\lb, f}) = f(\gamma) \Phi_M(\gamma, \eta_\lb) = (-1)^{\dim(A_M/A_G)} f(\gamma) \nu(M_\gamma)^{-1} \Phi_M(\gamma, \lb),
\]
so following the computation in \cite{Art89} section 6 gives:
\begin{cor}\label{geomsidebig}
Let $\lb_0$ be weight consistent with $\chi_0$ and $f \in C_c^\infty(A_{G, \infty}/A_{G, \rat})$. Then
\begin{multline*}
|\Pi_\disc(\lb_0)| I_{\geom, \chi_0}(\eta_{\lb_0, f} \otimes \varphi^\infty) = \sum_{M \in \ms L^\cusp} (-1)^{\dim(A_M/A_G)} \f{|\Om_M|}{|\Om_G|} \\
 \sum_{\gamma \in [M(F)]^\ssm} \chi(I^M_\gamma) |\iota^M(\gamma)|^{-1} f(\gamma) \Phi_M(\gamma, \lb_0) O^M_\gamma(\varphi^\infty_M)
\end{multline*}
where
\[
\chi(I^M_\gamma) = \f{\vol(I^M_\gamma(F)  \bs I^M_\gamma(\A)/ A_{I^M_\gamma, \rat})}{\vol(\bar I^M_{\gamma, \infty}/A_{I^M_\gamma, \infty})}|\Om(B_{K_{I^M_{\gamma,\infty}}})|
\]
and $\iota^M(\gamma)$ is the set of connected components of $M_\gamma$ that have an $F$-point.
\end{cor}
As explained on the top of page 19 in \cite{Shi12}, we can actually set 
\[
\chi(I^M_\gamma) = \bar \mu^{\can, EP}(I^M_\gamma(F)  \bs I^M_\gamma(\A)/ A_{I^M_\gamma, \rat})
\]
by picking measures appropriately on $I^M_\gamma$. Note a key change from Arthur's formula: the Levi's that appear are those for which $A_{M,\rat}/A_{G, \rat} = A_{M, \infty}/A_{G, \infty}$ instead of just those satisfying Arthur's notion of cuspidal. 

\subsubsection{Computation of $I'_{\geom,\chi}$}\label{geomaverage}
It remains to compute $I'_{\geom, \chi}(\eta_{\lb, f} \otimes \varphi^\infty)$ by averaging. To make the final formula more elegant, without loss of generality assume $\lb_0$ is consistent with $\chi$. We have
\[
I'_{\geom, \chi}(\eta_{\lb_0, f} \otimes \varphi^\infty) = \f1{\vol(\mf X_F \bs \mf X/ A_{G, \rat})} \int_{\mf X_F \bs \mf X/ A_{G, \rat}} \chi(z) I_{\geom, \chi_0} ((\varphi_{\pi_0,f} \otimes \varphi^\infty)_z) \, dz.
\]
Without loss of generality, taking $\eta_{\lb_0, f} = f\eta_{\lb_0}$ by lemma \ref{genpstops}:
\[
(\eta_{\lb_0,f} \otimes \varphi^\infty)_z = (\eta_{\lb_0,f})_{z_\infty} \otimes \varphi^\infty_{z^\infty} = \om^{-1}_{\lb_0}(z_\infty) \eta_{\lb_0, f_{z_\infty}} \otimes \varphi^\infty_{z^\infty}
\]
where $\om_{\lb_0}$ is the central character associated to $\lb_0$. Here, $\varphi_{\lb_0,f_{z_a}}$ is still a generalized Euler-Poincar\'e function so we substitute in corollary \ref{geomsidebig}. The terms that change are
\[
f(\gamma) \Phi_M(\gamma, \lb_0) \mapsto \om^{-1}_{\lb_0}(z_\infty) f_{z_\infty}(\gamma) \Phi_M(\gamma, \lb_0)
\]
and
\[
O^M_\gamma(\varphi^\infty_M) \mapsto O^M_\gamma((\varphi^\infty_{z^\infty})_M).
\]
By our simplifying assumptions, the $\om^{-1}_{\lb_0}(z_\infty)$ can be pulled out and partially cancelled against the $\chi(z)$. Finally, we use proposition \ref{ssidegenps}:
\[
\vol(\mf X^1_\infty)I_{\spec, \chi}(\eta_{\lb_0} \otimes \varphi^\infty) = \f1{\wh f(0)} I'_{\spec, \chi_0}(\eta_{\lb_0, f} \otimes \varphi^\infty) = \f1{\wh f(0)} I'_{\geom, \chi_0}(\eta_{\lb_0, f} \otimes \varphi^\infty)
\]
thereby getting the full formula we will use later:
\begin{prop}\label{geomsidegenps}
Let $\Pi_\disc(\lb_0)$ be a regular discrete series $L$-packet of $G_\infty$ with weight $\xi_0$ and central character $\chi$ on $\mf X$, $f$ a function pulled back through $H_{G_\infty}$ from $C_c^\infty(A_{G, \infty}/A_{G, \rat})$, and $\varphi^{\infty_1} \in \ms H(G^\infty, \chi_0)$ such that $\overline {(\varphi^{\infty_1})}_{\chi^\infty} = \varphi^\infty$.  Then we have geometric expansion
\begin{multline*}
\vol(\mf X^1_\infty) |\Pi_\disc(\lb_0)| I_{\spec, \chi}(\eta_{\lb_0} \otimes \varphi^\infty) = \\ \f1{\wh f(0)} \f1{\vol(\mf X_F \bs \mf X/ A_{G, \rat})}  \int_{\mf X_F \bs \mf X/ A_{G, \rat}} \chi(z^\infty)
 \sum_{M \in \ms L^\cusp} (-1)^{\dim(A_M/A_G)} \f{|\Om_M|}{|\Om_G|} \\
  \sum_{\gamma \in [M(F)]^\ssm}  \chi(I^M_\gamma)  |\iota^M(\gamma)|^{-1}  f(z_\infty \gamma) \Phi_M(\gamma, \lb_0) O^M_\gamma((\varphi^{\infty_1}_{z^\infty})_M) \, dz
\end{multline*}
where
\[
\chi(I^M_\gamma) = \f{\vol(I^M_\gamma(F)  \bs I^M_\gamma(\A)/ A_{I^M_\gamma, \rat})}{\vol(\bar I^M_{\gamma, \infty}/A_{I^M_\gamma, \infty})}|\Om(B_{K_{I^M_{\gamma,\infty}}})|
\]
and $\iota^M(\gamma)$ is the set of connected components of $M_\gamma$ that have an $F$-point.
\end{prop}

\subsubsection{Further Simplification} Mimicking some simplifications from \cite{KSZ}, the integral can be evaluated to remove $f$ and $\varphi^1$-dependence. This version of the formula and the method of its derivation are useful for some bounds later. 

$\mf X_F$ acts on  $[M(F)]^\ssm$ by multiplication. Let the set of orbits be $[M(F)]^\ssm_{\mf X}$. For any $\gamma$, let $\Stab_\mf X(\gamma)$ be the stabilizer of $\gamma$ under this action. This is finite by using a faithful representation (which always induces a finite-to-one map on semisimple conjugacy classes) to reduce to the case $G = \GL_n$. Here conjugacy classes are just sets of eigenvalues and the $\mf X$-action just scales each eigenvalue. Note also that since $\mf X$ is central, $\iota$ and $\nu$ are constant on $\mf X$-orbits.

We can therefore move the integral into the inner sum over $\gamma$ and break it up as
\begin{multline*}
\sum_{\gamma \in [M(F)]^\ssm_{\mf X}} \chi(I^M_\gamma) |\iota^M(\gamma)|^{-1} |\Stab_\mf X(\gamma)|^{-1} \\
\sum_{x \in \mf X_F} \int_{\mf X_F \bs \mf X/ A_{G, \rat}} \chi(z^\infty) f_{z_\infty}(x\gamma) \Phi_M(x \gamma, \lb_0) O^M_\gamma((\varphi^{\infty_1}_{xz^\infty})_M) \, dz.
\end{multline*}
Since $\chi$ is defined to be trivial on rational points, the innermost sum simplifies to
\begin{multline*}
\sum_{x \in \mf X_F} \int_{\mf X_F \bs \mf X/ A_{G, \rat}} \chi(z^\infty x) f(z_\infty x\gamma) \om^{-1}_{\lb_0}(x) \Phi_M(\gamma, \lb_0) O^M_\gamma((\varphi^{\infty_1}_{xz^\infty })_M) \, dz \\
=\Phi_M(\gamma, \lb_0) \lf( \int_{\mf X_\infty/A_{G, \rat}} f(z \gamma) \, dz \ri) \lf(\int_{\mf X^\infty} \chi(z) O^M_\gamma((\varphi^{\infty_1}_{z})_M) \, dz \ri).
\end{multline*}
Recalling
\[
(\varphi^{\infty_1}_{z})_M = \delta_{P_M}(\gamma^\infty)^{1/2} \int_{K^\infty} \int_{N_M(\A^\infty)} \varphi^{\infty_1}(k^{-1} \gamma^\infty z n k) \, dn \, dk,
\]
a bunch of Fubini's steps gives that the non-Archimedean integral is $O^M_\gamma((\overline{(\varphi^{\infty_1})}_\chi)_M) = O^M_\gamma((\varphi^\infty)_M)$ where we recall
\[
\overline{\varphi_\chi}(g) = \int_{\mf X^\infty} \varphi(gz) \chi(z) \, dz
\]
for any $\varphi$. 

For the Archimedean integral, let the $G_\infty = G_\infty^1 \times A_{G, \infty}$ components of any $g$ be $g_1 \times g_a$. Then $f(z \gamma) = f(z_a \gamma_a)$. This factorization gives a corresponding one $\mf X^\infty/A_{G, \rat} = \mf X_\infty^1 \times A_{G, \infty}/A_{G, \rat}$. Then the integral becomes
\[
\int_{\mf X_\infty^1} \int_{A_{G, \infty}/A_{G, \rat}} f(z_a \gamma_a) \, dz_a \, dz_1 = \vol(\mf X_\infty^1)\wh f(0).
\]
Putting it all together:
\begin{multline*}
|\Pi_\disc(\lb_0)|I'_{\geom, \chi}(\eta_{\lb_0, f} \otimes \varphi^\infty) = \f{\vol(\mf X_\infty^1)\wh f(0)}{\vol(\mf X_F \bs \mf X/ A_{G, \rat})}\sum_{M \in \ms L^\cusp} (-1)^{\dim(A_M/A_G)} \f{|\Om_M|}{|\Om_G|} \\
 \sum_{\gamma \in [M(F)]^\ssm_\mf X} \chi(I^M_\gamma) |\iota^M(\gamma)|^{-1}|\Stab_\mf X(\gamma)|^{-1} \Phi_M(\gamma, \lb_0) O^M_\gamma((\varphi^\infty)_M).
\end{multline*}
Using proposition \ref{ssidegenps} as before finally gives:
\begin{prop}\label{sgeomsidegenps}
Let $\Pi_\disc(\lb_0)$ be a regular discrete series $L$-packet of $G_\infty$ with weight $\xi_0$ and central character $\chi$ on $\mf X$. Then for any $\varphi^\infty \in \ms H(G^\infty, \chi^\infty)$, we have geometric expansion:
\begin{multline*}
I_{\spec, \chi}(\eta_{\lb_0} \otimes \varphi^\infty) = \f1{|\Pi_\disc(\lb_0)|}\f1{\vol(\mf X_F \bs \mf X/ A_{G, \rat})}\sum_{M \in \ms L^\cusp} (-1)^{\dim(A_M/A_G)} \f{|\Om_M|}{|\Om_G|} \\
 \sum_{\gamma \in [M(F)]^\ssm_\mf X} \chi(I^M_\gamma) |\iota^M(\gamma)|^{-1}|\Stab_\mf X(\gamma)|^{-1} \Phi_M(\gamma, \lb_0) O^M_\gamma((\varphi^\infty)_M).
\end{multline*}
\end{prop}
The dimensions on both sides are $[G^\infty][\mf X^\infty]^{-1}[\mf X^1_\infty]^{-1} = [G^\infty][\mf X/A_{G, \infty}]^{-1}$. We state again that the Levi's that appear are those for which $A_{M,\rat}/A_{G, \rat} = A_{M, \infty}/A_{G, \infty}$ instead of just those satisfying Arthur's notion of cuspidal.

\subsection{Irregular Discrete Series}
When $\lb_0$ is not regular, $\tr_{\pi_\infty} \eta_{\lb_0}$ does not simply test if $\pi_\infty$ is in a given $L$-packet. However, it can be interpreted as a cohomology as in \cite[\S2]{Art89}. While we will not use this more general result, we state it here in case it is useful in other applications.

Even with irregular $\lb_0$, we still have
\[
|\Pi_\disc(\lb_0)|I_{\spec, \chi}(\eta_{\lb_0} \otimes \varphi^\infty) = \f1{\vol(\mf X^1_\infty)}\sum_{\pi \in \mc{AR}_\disc(G, \chi)} m_\disc(\pi) \tr_{\pi_\infty}(\eta_{\lb_0}) \tr_{\pi^\infty}(\varphi^\infty).
\]
The Euler-Poincar\'e function $\eta_{\lb_0}$ always satisfies $\tr_{\pi_\infty}(\eta_{\lb_0}) = \chi_{\lb_0}(\pi_\infty)$ where $\chi_{\lb_0}$ is the Euler characteristic
\[
\chi_{\lb_0}(\pi_\infty) = \sum_q (-1)^q \dim H^q(\mf g(\R), K_\infty, \pi_\infty \otimes \pi_{\lb_0}).
\]
Here, $H^q$ is the $(\mf g, K)$-cohomology: $K_\infty$ is a maximal compact of $G_\infty$ and $\pi_{\lb_0}$ is the finite dimensional complex representation with highest weight $\lb_0$. The equality holds in general because it holds on basic representations which generate the Grothendieck group. 

In particular, if we define the $L^2$-Lefschetz number
\[
\ms L_{\lb_0}(\varphi^\infty) = \sum_{\pi \in \mc{AR}_\disc(G, \chi)} m_\disc(\pi) \chi_{\lb_0}(\pi_\infty) \tr_{\pi^\infty}(\varphi^\infty),
\]
we get
\[
|\Pi_\disc(\lb_0)| I_{\spec, \chi}(\eta_{\lb_0} \otimes \varphi^\infty) = \f1{\vol(\mf X^1_\infty)}\ms L_{\lb_0}(\varphi^\infty).
\]
Combining with the calculations before proposition \ref{sgeomsidegenps} gives the formula:
\begin{cor}
Let $\pi_0$ be a possibly irregular discrete series representation of $G_\infty$ with weight $\xi_0$ matching character $\chi$ on $\mf X$. Then for any $\varphi^\infty \in \ms H(G^\infty, \chi^\infty)$:
\begin{multline*}
\ms L_{\lb_0}(\varphi^\infty) = \f{\vol(\mf X^1_\infty)}{\vol(\mf X_F \bs \mf X/ A_{G, \rat})}\sum_{M \in \ms L^\cusp} (-1)^{\dim(A_M/A_G)} \f{|\Om_M|}{|\Om_G|} \\
 \sum_{\gamma \in [M(F)]^\ssm_\mf X} \chi(I^M_\gamma) |\iota^M(\gamma)|^{-1}|\Stab_\mf X(\gamma)|^{-1} \Phi_M(\gamma, \lb_0) O^M_\gamma((\varphi^\infty)_M).
\end{multline*}
\end{cor}
The dimensions on both sides are $[G^\infty][\mf X^\infty]^{-1}$.

\section{Trace Formula Computation Set-Up}
Now we can finally set up our main computation.
\subsection{Conditions on $G$ and Defining Families}\label{cndtns}
Let $G$ be a reductive group over a number field $F$ with discrete series at $\infty$. By instead looking at $\Res_\Q^F G$, we could without loss of generality take $F = \Q$ since $\Res_\Q^F G(\Q) = G(F)$ and $\Res_\Q^F G(\A) = G(\A_F)$ as topological groups. Fix central character datum $(\mf X, \chi)$. Assume $G$ is connected.

Let:
\begin{itemize}
\item
$\pi_0$ be a regular real discrete series representation for $G$ with weight $\xi_0$ and character $\chi$ on $A_{G, \infty}$. 
\item
$\varphi_{\pi_0}$ be its pseudocoefficient.
\item
$S_0$ be a finite set of finite places and choose $\varphi_{S_0} \in \ms H (G_{S_0}, \chi_{S_0})$.
\item
$S_1$ be another finite set of finite places disjoint from $S_0$ such that $\chi_{S_1}$ is unramfied.
\item
$S = S_0 \sqcup S_1$.
\item
$U^{S, \infty} \subset G(\A^{S, \infty})$ an open compact subset on which $\chi^{S, \infty}$ is trivial. 
\item
$S_\bad$ is a set of places that $S_1$ needs to be disjoint from that will be defined in section \ref{final}.
\end{itemize}

Define a family of automorphic representations $\mc F$ in $\mc{AR}_\disc(G, \chi)$ through discrete multiplicities
\[
a_{\mc F}(\pi)  = m_\disc(\pi)  \delta_{\pi_0, \pi_\infty}  \dim(\pi^{S, \infty})^{U^{S, \infty}} \f{\wh{\m 1}_{K_{S_1}}(\pi_{S_1})}{\vol(K_{S_1})}.
\]
Note that the second-to-last term is just checking if $\pi_{S_1}$ is unramified. The coefficient $a_{\mc F}(\pi)$ is dimensionless.

Define function
\[
\1_{U^{S, \infty}, \chi} = \vol(U^{S, \infty} \cap \mf X^{S, \infty})^{-1} \overline{(\m 1_{U^{S, \infty}})}_\chi.
\]
This is normalized so that $\1_{U^{S, \infty}, \chi}(1) = 1$. For any test function $\varphi_{S_1} \in \mc H^\ur(G_{S_1}, \chi_{S_1})$ let 
\[
\varphi = \varphi_{\pi_0, f, \varphi_{S_0}} = \varphi_{\pi_0} \otimes \varphi^\infty = \varphi_{\pi_0} \otimes \1_{U^{S, \infty}, \chi} \otimes \varphi_{S_0} \otimes \varphi_{S_1} 
\]
where as before, $\varphi_\pi$ is the pseudocoefficient for $\pi$. Test function $\varphi$ will momentarily be shown to pick out the family $a_{\mc F}$. 

Intuitively, the test function is
\begin{itemize}
\item
putting weight restrictions on the infinite place,
\item
putting level restrictions on finite places away from $S$,
\item
forcing $S_1$ parts to be unramified,
\item
counting possible components at $S$ according to test function $\varphi_S$ with $\varphi_{S_1}$ unramified.
\end{itemize}

To make all the traces well-defined, we fix Haar measures on factors of $G(\A_F)$:
\begin{itemize}
\item
Use the normalization from \cite[\S6.6]{ST16} of Gross' canonical measure from \cite{Gro97} on $G_S$ and the $\mf X_S$. 
\item
Use Euler-Poincar\'e measure on $G_\infty$, $A_{G, \infty}$, $A_{G, \rat}$, and $\mf X^1_\infty$. 
\end{itemize}
This determines all appropriate Plancherel measures. We call the product measure $\mu^{\can, EP}$ and the volume of the adelic quotient under it the modified Tamagawa number $\tau'(G)$. 
 
\subsection{Spectral Side}\label{SSide}
We can now directly compute the spectral expansion of $I_{\spec, \chi}(\varphi)$:
\begin{cor}\label{spectral}
Let $\pi_0$ be a regular discrete series representation of $G$ with weight $\xi_0$. Then:
\[
I^G_{\spec, \chi}(\varphi_{\pi_0} \otimes \varphi^\infty) = \bar \mu^\can(U^{S, \infty}_\mf X) \sum_{\pi \in \mc{AR}_\disc(G, \chi)} a_{\mc F}(\pi) \wh \varphi_S(\pi)
\]
where $U^{S, \infty}_\mf X = U^{S, \infty}/\mf X^{S, \infty} \cap U^{S, \infty}$. 
\end{cor}

\begin{proof}
By proposition \ref{ssidegenps} and using that $\vol(\mf X^1_\infty) = 1$, 
\[
I^G_{\spec, \chi}(\varphi_{\pi_0} \otimes \varphi^\infty) = \f1{\vol(\mf X^1_\infty)} \sum_{\pi \in \mc{AR}_\disc(G, \chi)} m_\disc(\pi) \delta_{\pi_0, \pi_\infty} \tr_{\pi^\infty}(\varphi^\infty).
\]
Factoring the finite trace into its $S_0, S_1$ and other components gives that
\[
\tr_{\pi^\infty}(\varphi^\infty) = \wh \varphi_{S_0}(\pi_{S_0}) \f{\wh{\m 1}_{K_{S_1}}(\pi_{S_1})}{\vol(K_{S_1})}\wh \varphi_{S_1}(\pi)  \mu^\can(U^{S, \infty}_\mf X) \dim(\pi^{S, \infty})^{U^{S, \infty}} 
\]
so we are done.
\end{proof}

\subsection{Geomteric Side Outline}
We get a geometric expansion $I_{\spec, \chi}(\varphi_{\pi_0} \otimes \varphi^\infty)$ by using the hyperendoscopy formula (proposition \ref{Hform}). Since Euler-Poincar\'e functions and pseudocoefficients have the same stable orbital integrals:
\begin{multline*}
I_{\spec,\chi}^G(\varphi_{\pi_0} \otimes \varphi^\infty) \\
= I_{\spec,\chi}^G(\eta_{\lb_0} \otimes \varphi^\infty) + \sum_{\mc H \in \mc{HE}_\el(G)} \iota(G, \mc H)I_{\spec, \chi_\mc H}^{\mc H}((\eta_{\xi_0} - \varphi_{\pi_\infty})^{\mc H} \otimes (\varphi^\infty)^\mc H).
\end{multline*}
Simplifying and bounding this takes a few steps:
\begin{enumerate}
\item
Notice that transfers $(\eta_{\xi_0} - \varphi_{\pi_\infty})^{\mc H}$ through hyperendoscopic paths can be chosen to be linear combinations of regular Euler-Poincar\'e functions.
\item
Substitute in proposition \ref{geomsidegenps} for each hyperendoscopic group.
\item
The result will have a main term consisting of central elements of $G$ and an error term consisting of non-central elements, Levi terms, and terms from the hyperendoscopic groups.
\item
Use a Poisson summation argument to compute the main term.
\item
Bound the error term using bounds on non-Archimedean transfers and small generalizations of the results of \cite{ST16}.
\end{enumerate} 
For sanity checks later, note that both sides of our computation have dimension $[G^\infty][\mf X/A_{G, \infty}]^{-1}$.

\section{Geometric Side Details}
We are eventually going to use the hyperendoscopic formula with $f_1$ of the form
\[
f_1 = \eta_\xi  \otimes \varphi^\infty.
\]
All transfers appearing will have linear combinations of Euler-Poincar\'e functions as infinite parts so we only need to analyze the geometric side with test functions of the form $\eta_\xi \otimes \varphi^\infty$. This is similar to what was done in \cite{ST16}.

\subsection{Original Bounds}
Recall the notation and conditions from \ref{cndtns}. We state the main bounds from \cite{ST16} for reference.
$G$ determines a finite set of places $S_{\bad', G}$ in a complicated, uncontrolled manner. We assume three conditions:
\begin{itemize}
\item
$S$ does not intersect $S_{\bad', G}$.
\item
$G$ is cuspidal.
\item 
$\mf X$ is trivial. 
\end{itemize}
Then we get the following bounds (changing to our normalization of EP-functions):
\begin{thm}[Weight-aspect bound {\cite[thm 9.19]{ST16}}]\label{owtbound}
Consider the case where $Z_G =1$. Let $f_{S_1} \in \ms H^\ur(G(F_{S_1}))^{\leq \kappa}$ such that $\|f_{S_1}\|_\infty \leq 1$. Let $\xi$ be a dominant weight. Then
\[
\f{|\Pi_\disc(\lb_0)|}{\tau'(G) \dim(\xi) \wh \mu^\pl_{S_0}(\wh \varphi_{S_0})} I_\spec(\eta_\xi \otimes \varphi^\infty) = \wh \mu^\pl_{S_1}(\wh f_{S_1}) + O_{G, \varphi_{S_0}}(q_{S_1}^{A_\wt + B_\wt \kappa} m(\xi)^{-C_\wt})
\]
for some constants $A_\wt,B_\wt,C_\wt$ depending only on $G$.
\end{thm}

\begin{thm}[Level-aspect bound {\cite[thm 9.16]{ST16}}]\label{lvbound}
Consider the case where $U^{S, \infty}$ is a level subgroup $K^{S, \infty}(\mf n)$ for some ideal $\mf n$ relatively prime to $S_{\bad', G}$. Let $f_{S_1} \in \ms H^\ur(G(F_{S_1}))^{\leq \kappa}$ such that $\|f_{S_1}\|_\infty \leq 1$. Let $\xi$ be a dominant weight. Then, if $\N(\mf n)$ is large enough,
\[
\f{|\Pi_\disc(\lb_0)|}{\tau'(G) \dim(\xi) \wh \mu^\pl_{S_0}(\wh \varphi_{S_0})} I_\spec(\eta_\xi \otimes \varphi^\infty) = \wh \mu^\pl_{S_1}(\wh f_{S_1}) + O_{G, \varphi_{S_0}}(q_{S_1}^{A_\lv + B_\lv \kappa} \N(\mf n)^{-C_\lv})
\]
for some constants $A_\lv,B_\lv,C_\lv$ depending only on $G$.
\end{thm}

For clarity later, we emphasize that the implied constants in the big $O$ depend on $G$ and $\varphi_{S_0}$. As noted in errata on the authors' websites, there is a mistake in \cite[\S7]{ST16} so the alternate argument in \cite[B]{ST16} must be used for the orbital integral bounds that go into the results. This alternate argument does not provide any control on the constants or $S_{\bad'}$. 

\subsubsection{Clarifying a minor detail}
As another note, there is a small detail assumed in the bound for $a_{\gamma, M}$ used in proving the weight aspect bound: corollary 6.16 used to bound the $L$ function in the formula for $\bar \mu^{\can, EP}(G(F) \bs G(\A)/A_{G, \rat})$  only applies to groups with anisotropic center. However 6.17 uses it for centralizers of elements and these can have arbitrary center. We can use the following lemma to get an alternate bound for $\bar \mu^{\can, EP}(G(F) \bs G(\A)/A_{G, \rat})$ in general in terms of the bound for groups with anisotropic center:
\begin{lem}\label{aboundgap}
Let $G$ be a connected reductive group over $F$ and $G' = G/A_G$. Then
\begin{multline*}
\bar \mu^{\can, EP}(G(F) \bs G(\A)/A_{G, \infty}) \\
= \bar \mu^{\can, EP}(G'(F) \bs G'(\A)) \bar \mu^{\can, EP}(A_G(F) \bs A_G(\A) / A_{A_G, \rat}).
\end{multline*}
Note that the factor $\mu^{\can, EP}(A_G(F) \bs A_G(\A) / A_{A_G, \rat})$ is a constant depending only on the field $F$ and the dimension of $A_G$.
\end{lem}

\begin{proof}
If $G$ is quasisplit at finite $v$, there is a special model $\underline G$ over $F_v$. Then $\underline G(\mc O_v) \cap A_G(F_v)$ is a maximal (a bigger subgroup times $G(\mc O_v)$ is otherwise a bigger compact) connected compact subgroup and therefore corresponds to a model $\underline A_G$ consistent with the inclusion. Consider the quotient model $\underline{G/A_G}$. By Lang's theorem, $\underline{G'}(k_v) = \underline G(k_v)/\underline{A_G}(k_v)$, so by Hensel's lemma and smoothness of quotient maps by smooth subgroups, $\underline{G/A_G}(\mc O_v) = \underline G(\mc O_v)/\underline{A_G}(\mc O_v)$. By Hilbert 90, $G'(F_v) = G(F_v)/A_G(F_v)$ for any local $F_v$. This gives that $G'(\A) = G(\A)/A_G(\A)$ implying $G'(\A)^1 = G'(\A) = G(\A)^1/A_G(\A)^1$. 

Using  $G'(F) = G(F)/A_G(F)$, we then get an isomorphism of topological spaces
\[
G(F) \bs G(\A)^1 \cong G'(F) \bs G'(\A) \times A_G(F) \bs A_G(\A)^1.
\]
Next, $\mu^{\can, EP}$ on $G'(\A)$ and $G(\A)$ induces a measure $\mu_A$ on $A_G(\A)$. By the above factorization, it suffices to show that this equals $\mu_A^{\can, EP}$ place by place. At the infinite place, they are the same by definition (see \cite[\S6.5]{ST16}). 

If $G$ is quasisplit at finite $v$, then $\mu^\can$ is characterized by giving any special subgroup volume $1$.   As before, $\underline{G/A_G}(\mc O_v) = \underline G(\mc O_v)/\underline A_G(\mc O_v)$. In particular, $\underline{G/A_G}(\mc O_v)$ also needs to be maximal connected so it is special. Since these are all special subgroups, this forces $\mu_A = \mu_A^\can$ at $v$. 

If $G$ is not quasisplit at $v$, then $\mu^\can$ is determined by the transfer of a top-form $\om_{G^\qs}$ from $G^\qs$ (since the normalization factor  $\Lambda$ in \cite{ST16} depends only on the motive for $G$ which depends only on the quasisplit form of $G$). The isomorphism $G_{\overline k} \iso G^\qs_{\overline k}$ carries $(A_G)_{\overline k}$ to $(A_{G^\qs})_{\overline k}$ since centers are identified between inner forms. This means that $G'^\qs = G^\qs/A_{G^\qs}$ through the isomorphism over $\overline k$. By the previous paragraph, the defining top-forms for $G'_\qs$ and $A_{G^\qs}$ wedge together to that of $G^\qs$. Therefore, this same property holds for $G$ and $A_G$ which is what we want. 
\end{proof}
The previous lemma is implicit in later sections of \cite{ST16} but not explained in detail.

\subsection{New Bounds Set-up}\label{setup} 
For our use, we will need a generalization of these bounds that works when $Z_G \neq 1$ and when $G$ is not necessarily cuspidal. We will also need the big $O$, choices of $S_{\bad,H}$, and the constants $A,B,C$ to be uniform over all groups $H$ appearing in hyperendoscopic paths of $G$. The final statement requires some notation and will be in Theorem \ref{mainresult}. 

Let $\xi$ be a dominant weight and choose central character datum $(\mf X, \chi)$ where $A_{G, \infty} \subseteq \mf X$ and $\chi$ is consistent with $\xi$. Let $\chi_0$ be its restriction to $A_{G, \rat}$. We start similar to \cite[thm 4.11]{Shi12} and \cite[thm 9.19]{ST16}, instead trying to apply proposition \ref{geomsidegenps}. This requires making some choices:
\begin{itemize}
\item
a cutoff function $f \in C_c^\infty(A_{G, \infty}/A_{G, \rat})$,
\item
a $\varphi^{\infty_1} \in \ms H(G^\infty, \chi_0)$ such that $\overline{(\varphi^{\infty_1})}_\chi = \varphi^\infty$, 
\item
lots of Haar measures: fix them to be $\mu^{\can \times EP}$ whenever necessary.
\end{itemize}
We need to bound the term for all endoscopic groups. Considering all the previous lemmas on transfers, we are interested in the case where:
\begin{itemize}
\item
$\varphi$ and $\chi$ are unramified outside of $S_0$ and $\infty$. 
\item
$\chi$ extends to a character on $G_v$. 
\item
$(\varphi^{S, \infty})^1$ can be chosen to be $\vol(\mf X^{S, \infty} \cap U^{S, \infty})^{-1}\1_{U^{S, \infty}}$. For endoscopic groups we will without loss of generality expand $S_0$ so that $U^{S, \infty} = K^{S, \infty}$. Then this follows from the computation of  transfers in section \ref{zextarchtransfer}. 
\item
$\varphi_s \in \ms H(G_s, K_s, \chi_s)^{\leq \kappa}$ and $\|\chi_s \varphi_s\|_\infty \leq 1$ for all $s \in S_1$. 
\end{itemize}
We choose a specific $\varphi_s^1$ for $s \in S_1$ according to the following lemma. 
\begin{lem}
Pick unramified character datum $(\mf X_v, \chi_v)$ such that $\chi_v$ extends to a character on $G$. Let $\varphi_v \in \ms H(G_v, K_v, \chi_v)^{\leq \kappa}$ such that $\| \chi_v \varphi_v\|_\infty \leq 1$. Fix the canonical measure on $\mf X_v$ so that $\vol(K \cap \mf X_v) = 1$. Then there exists $\varphi^1_v \in \ms H(G_v, K_v)^{\leq \kappa}$ such that $\overline{(\varphi^1_v)}_{\chi_v} = \varphi_v$ and $\|\chi_v \varphi^1_v\|_\infty \leq 1$. 
\end{lem}

\begin{proof}
Let
\[
\varphi_v = \sum_{\lb \in X_*(A)} a_\lb \tau_\lb.
\]
Let $A_{\mf X_v}$ be the split part of $\mf X_v$. Then for any $\zeta \in X_*(A_{\mf X_v})$, $a_{\lb + \zeta} = \chi(\zeta(\varpi))^{-1} a_\lb$. For each $\lb$ such that $a_\lb \neq 0$, there is a representative $\lb'$ of its class $[\lb] \in X_*(A)/X_*(A_{\mf X_v})$ such that $\|\lb'\| \leq \kappa$. Let $\Lambda$ be the set of all these chosen representatives. Then
\[
\varphi_v^1 = \varphi_v = \sum_{\lb \in \Lambda} a_\lb \tau_\lb
\]
satisfies $\overline{(\varphi^1_v)}_{\chi_v} = \varphi_v$ The $L^\infty$ bound on $\varphi_v$ gives that $|\chi_v(\lb(\varpi)) a_\lb| = 1$ implying the needed bound on $\varphi^1_v$. 
\end{proof}

\begin{note}
There is a small technicality here. The original $\chi_v$ chosen on the subgroup of $G_v$ may not necessarily extend to $G_v$. However, section \ref{hyptransferchar} still gives that $\chi_{\mc H,v}$ on any $\mc H_v$ is a character $\lb$ that extends to $\mc H_v$ times $\chi_v$. Since $Z_{G^\der}$ is finite, $\chi_v$ can be factored as a unitary character times a character on $G_v$. Since the bounds here are only up to absolute value, this does not matter.
\end{note}
 
 Beginning the computation:
\begin{multline*}
 \f{|\Pi_\disc(\lb_0)|}{\tau'(G) \dim(\xi)} I_{\spec, \chi}(\eta_\xi \otimes \varphi^\infty) = 
\f1{\wh f(0)} \f1{\vol(\mf X_F \bs \mf X/ A_{G, \rat})}  \int_{\mf X_F \bs \mf X/ A_{G, \rat}} \chi(z^\infty) \\
 \sum_{M \in \ms L^\cusp} \sum_{\gamma \in [M(F)]^\ssm}  a_{M,\gamma}  |\iota^M(\gamma)|^{-1}  f(z_\infty \gamma) \f{\Phi_M(\gamma, \xi)}{\dim \xi} O^M_\gamma((\varphi^{\infty_1}_{z^\infty})_M) \, dz.
\end{multline*}
Here
\[
a_{M,\gamma} = \tau'(G)^{-1} \f{|\Om_M|}{|\Om_G|} \f{\overline \mu^{\can, EP}(I^M_\gamma(F) \bs I^M_\gamma(\A_F) / A_{I^M_\gamma, \Q})}{\overline \mu^{EP}(\bar I^M_{\gamma,\infty}/A_{I^M_\gamma, \infty})}
\]
(see the top of page 19 in \cite{Shi12}).

This double sum breaks into three pieces: $M=G$ and $\gamma \in Z_G$, $M=G$ otherwise, and $M \neq G$. For $M=G$, $\Phi_M(\gamma, \xi) = \tr \xi(\gamma_\infty)$. For central $\gamma$, the centralizer is everything so $|\iota^G(\gamma)| = 1$. In addition, the measure on the quotient is just counting measure on a point so $O^M_\gamma(\varphi^{\infty_1}_{z^\infty}) = \varphi^{\infty_1}(z^\infty \gamma)$. Finally,
\[
a_{G, \gamma}  = \tau'(G)^{-1} \f{\overline \mu^{\can, EP}(G(F) \bs G(\A_F) / A_{G, \rat})}{\overline \mu^{EP}(\bar G_\infty/A_{G, \infty})} = \overline \mu^{EP}(\bar G_\infty/A_{G, \infty})^{-1} = 1
\]
since existence of a discrete series requires that the last group is compact and therefore has EP-measure $1$.This leaves us with
\[
\f1{\wh f(0)} \f1{\vol(\mf X_F \bs \mf X/ A_{G, \rat})}  \int_{\mf X_F \bs \mf X/ A_{G, \rat}} \chi(z^\infty)  \sum_{\gamma \in Z_G(F)} \varphi^{\infty_1}(\gamma) f(z \gamma) \f{\tr \xi(z_\infty \gamma)}{\dim \xi}  \, dz.
\]
Next, note that by a Fourier inversion formula
\[
\f{\tr \xi(\gamma)}{\dim \xi} = \om_\xi^{-1}(\gamma) = \om_\xi(z_\infty) \om_\xi^{-1}(z_\infty \gamma) = \om_\xi(z_\infty) \eta_\xi(z_\infty \gamma) \eta_\xi(1)^{-1}
\]
where $\om_\xi$ is the central character for $\xi$. Therefore, the term inside the sum is simply $\om_\xi(z_\infty)f(z \gamma) \varphi^1(z \gamma)$ where $\varphi^1 = \eta_\xi \varphi^{\infty_1}$.

Combining the $\om_\xi(z_\infty)$ factor with the $\chi$, we get a main term
\begin{equation}\label{main}
\f1{\wh f(0) \eta_\xi(1)} \f1{\vol(\mf X_F \bs \mf X/ A_{G, \rat})}  \int_{\mf X_F \bs \mf X/ A_{G, \rat}} \chi(z)  \sum_{\gamma \in Z_G(F)} f(z_\infty \gamma) \varphi^1(z\gamma)  \, dz
\end{equation}
The leftovers form an error term
\begin{multline}\label{error}
\f1{\wh f(0)} \f1{\vol(\mf X_F \bs \mf X/ A_{G, \rat})}  \int_{\mf X_F \bs \mf X/ A_{G, \rat}} \chi(z^\infty) \\
\lf( \sum_{\substack{\gamma \in [G(F)]^\ssm \\ \gamma \notin Z_G}} a_{G,\gamma} |\iota^G(\gamma)|^{-1} f(z_\infty \gamma) \f{\tr \xi(\gamma_\infty)}{\dim \xi} O_{z^\infty\gamma}^G(\varphi^{\infty_1})  \ri. \\
\lf. + \sum_{\substack{M \in \ms L^\cusp \\ M \neq G}} \sum_{\gamma \in [M(F)]^\ssm} a_{M,\gamma} |\iota^M(\gamma)|^{-1} f(z_\infty \gamma) \f{\Phi_M(\gamma, \xi)}{\dim \xi} O_{z^\infty\gamma}^M(\varphi_M^{\infty_1}) \ri) \, dz
\end{multline}
We compute these separately since they require pretty different ideas to understand.

\subsection{The Main Term}\label{themainterm}
\subsubsection{Central Fourier transforms}\label{cft}
This section uses material on Fourier analysis on non-abelian groups. See \cite{Fol16} chapter 7 for a good reference. That $p$-adic reductive groups are type I is a classic result from \cite{Ber74}.

The main term initially simplifies in terms of the Fourier transform $\bar f_S$ of $f_S$ with respect to $(Z_G)_S$. To actually get a reasonable interpretation, we need to relate $\bar f_S$ to $\wh f_S$. Therefore, for this subsection only, redefine $G = G_S$, $Z = (Z_G)_S$ and consider arbitrary $f \in \mc H(G_S)$. Note that the following results probably hold for general type I unimodular groups with an appropriate modification of $\mc H(G)$ to a more complicated function space; the case of $p$-adic groups just makes the analytic issues a lot nicer.

There is a map from $P : \wh G \to \wh Z_G$ taking $\pi$ to its central character $\om_\pi$. 
\begin{lem}
$P$ is measurable with respect to the usual sigma algebras on $\wh G$ and $\wh Z$. 
\end{lem}
\begin{proof}
Fix a Hilbert space $H_i$ of dimension $i$ for $i \in \N$ or countable infinity. Let $\Pi$ be the set of irreducible unitary representations of $G$ on some $H_i$. Consider the functions on $\Pi$ defined by $\pi \mapsto \langle \pi(g)v, w \rangle$ for $g \in G$ and $v,w$ in the appropriate Hilbert space. Since $G$ is type I, the $\sigma$-algebra on $\wh G$ is the quotient of the smallest one on $\Pi$ that makes these functions continuous. An analogous statement holds for $\wh Z$. 

Then, since central elements act by central characters, the functions defined by $z \in Z$ on $\wh G$ are exactly the pullbacks by $P$ of the analogous functions on $\wh Z$.
\end{proof}
Denote the Fourier transform of $f|_{Z_G}$ by $\bar f$. 
\begin{lem}\label{centerftl}
For any functions $\varphi \in \mc H(\wh Z)$ and $f \in \mc H(G)$
\[
\int_{\wh Z} \varphi \bar f \, d\mu^\pl = \int_{\wh G} (\varphi \circ P) \wh f \, d\mu^\pl.
\]
\end{lem}
\begin{proof}
Using both Fourier inversion theorems, for any $z \in Z$
\[
\int_{\wh Z} \om(z) \bar f(\om)  \, d\om = f(z) = \int_{\wh G} \om_\pi(z) \wh f(\pi)  \, d\pi.
\]
For a general $\varphi$
\begin{align*}
\int_{\wh G} \varphi(\om_\pi) \wh f(\pi) \, d\pi &= \int_{\wh G} \int_Z  \bar \varphi(z) \om_\pi^{-1}(z) \wh f(\pi) \, dz \, d\pi \\
&= \int_Z \bar \varphi(z) \int_{\wh G}   \om_\pi^{-1}(z) \wh f(\pi) \, d\pi \, dz \\
&= \int_Z \bar \varphi(z) \int_{\wh Z}   \om^{-1}(z) \bar f(\om) \, d\om \, dz \\
&= \int_{\wh Z} \int_Z  \bar \varphi(z) \om^{-1}(z) \bar f(\om) \, dz \, d\om  = \int_{\wh Z} \varphi(\om) \bar f(\om) \, d\om
\end{align*}
so we are done.
\end{proof}

Intuitively, we can therefore think of $\bar f(\om)$ as an average of $\wh f$ over representations with central character $\om$. To make this notion precise, push $\wh f \, d\mu^\pl$ forward to a measure $\mu_{\wh f}$ on $\wh Z_G$. 
\begin{lem}
$\mu_{\wh f}$ is absolutely continuous with respect to Haar measure on $\wh Z_G$. 
\end{lem}
\begin{proof}
Let $X \subset \wh Z$ have measure $0$. By $\sigma$-finiteness, outer regularity, and continuity of $\bar f$, for any $\eps > 0$, $X$ is contained in a union $X_\eps$ of countably many compact open sets such that $\int_{X_\eps} \bar f d\mu^\pl < \eps$. Then
\[
\mu_{\wh f}(X) \leq \mu_{\wh f}(X_\eps) = \int_{\wh G} \1_{P^{-1}(X_\eps)} \wh f \, d\mu^\pl = \int_{\wh Z} \1_{X_\eps} \bar f \, d\mu^\pl < \eps.
\]
Since this is true for every $\eps > 0$, $\mu_{\wh f}(X) = 0$, so we are done.
\end{proof}
Therefore we can define:
\begin{dfn}
The \emph{conditional Plancherel expectation} is the Radon-Nikodym derivative
\[
E^\pl(\wh f | \om) := \dif{\mu_{\wh f}}{\mu^\pl_{Z_G}}(\om).
\]
\end{dfn}
This is defined up to a set of measure $0$. However, note that the measures $E^\pl(\wh f|\om) \, d\mu^\pl$ and $\bar f d\mu^\pl$ are the same on $\wh Z$ so:
\begin{cor}\label{centerft}
$E^\pl(\wh f|\om)$ can be taken to be continuous. If so $E^\pl(\wh f|\om) = \bar f(\om)$. 
\end{cor}
We borrow the notation of conditional expectation from probability theory to emphasize first, the same definition in terms of Radon-Nikodym derivatives and second, the analogous intuition as an average over the measure-zero set of representations with central character $\om$. Beware that under this analogy, $E^\pl$ is an unnormalized expectation since $E^\pl(\wh f |\om) = \bar f$ and the operation $f \mapsto \bar f$ multiplies in a factor of $[Z]$ to the dimensions of $f$.

\subsubsection{Main term computation}
\begin{prop}\label{wtbound}
The main term \eqref{main} simplifies to
\[
\f1{|X|} \f\mu{\vol(Z'_{S, \infty}/L)}  \sum_{\om_S \in \wh Z_{S, L, \xi, \chi}} E^\pl(\wh \varphi_S | \om_S),
\]
where $Z'_{S, \infty} = Z_{G_{S,\infty}}/A_{G, \rat}$, $L = Z_G(F) \cap U^{S,\infty}$, and $\wh Z_{S, L, \xi, \chi}$ is the set of $\om_S \in \wh Z_S$ such that $\om_S|_L = \om_\xi|_L$ and $\om_S|_{\mf X_S} = \chi_S$. The normalizing factors are
\begin{itemize}
\item
$\mu = \mu_{Z'_\infty}/\mu^{EP}_{Z'_\infty}$ where $\mu_{Z'_\infty}$ is the measure chosen on  $Z'_\infty$ to compute the other terms.
\item
$X$ is the finite group $\mf X^{S, \infty}/\mf X^{S, \infty} \cap \overline{Z_G(F)}Z_{U^{S, \infty}}$ where the closure is taken in $Z^{S, \infty}$. 
\end{itemize}
For shorthand, we denote this sum $E(\wh \varphi_S | \om_\xi, L, \chi_S)$.
\end{prop}

\begin{proof}
Start with \eqref{main}:
\[
\f1{\wh f(0) \eta_\xi(1)} \f1{\vol(\mf X_F \bs \mf X/ A_{G, \rat})}  \int_{\mf X_F \bs \mf X/ A_{G, \rat}} \chi(z)  \sum_{\gamma \in Z_G(F)} f(z_\infty \gamma) \varphi^1(z\gamma)  \, dz.
\]
$Z_G(F)$ is cocompact and discrete inside $Z^1 = Z_G(\A)/A_{G, \rat}$. Then by Poisson summation, the inner sum becomes
\[
\f1{\vol(Z/Z_G(F))}\sum_{\substack{\om \in \wh{Z^1} \\ \om(Z_G(F)) = 1}} \om^{-1}(z) \overline{f\varphi^1}(\om)
\]
since if $\varphi_z : x \mapsto \varphi(zx)$, then $\bar \varphi_z(\om) = \om^{-1}(z) \bar \varphi(\om)$. Integrating over $z$, all terms with $\om \neq \chi$ vanish so \eqref{main} becomes
\[
\f1{\wh f(0)} \f1{\vol(Z^1/Z_G(F))}\sum_{\substack{\om \in \wh{Z^1} \\ \om(Z_G(F)) = 1 \\ \om|_\mf X = \chi}} \overline{f\varphi}(\om).
\]
Here we use that $\varphi^\infty$ has Fourier transforms on any $\om^\infty$ in the sum and $\overline{\varphi^\infty} = \overline{\varphi^{\infty_1}}$ on these characters. We next break this up into local components to make it more interpretable. First,
\[
\bar \varphi(\om) = \overline{f\eta_\xi}(\om_\infty) \bar \varphi_S(\om_S) \bar \varphi^{S, \infty} (\om^{S, \infty})
\]
after choosing Haar measures on the components of $Z^1$. Let $\om_\xi$ be the central character associated to $\xi$. For any test function $\psi$ compactly supported on $Z'_\infty = Z_{G, \infty}/A_{G, \rat}$, 
by lemma \ref{centerftl} applied to $G_{\infty}/A_{G, \rat}$,
\begin{multline*}
\int_{\wh{Z'_\infty}} \psi(\om) \overline{f \eta_\xi}(\om) \, d\om^\pl = \int_{(G_\infty/A_{G, \rat})^\vee} \psi(\om_\pi) \wh{f \eta_\xi}(\pi) \, d\pi^\pl = \\
\int_{\wh A} \int_{\wh{G^1_\infty}} \psi(\om \om_\pi) \wh{f \eta_\xi}(\pi \otimes \om) \, d\pi^\pl \, d\om^\pl = \vol_{\wh{G^1_\infty}}(\Pi_\disc(\xi)) \int_{\wh A} \psi(\om_\xi \om) \wh f(\om) \, d\om^\pl
\end{multline*}
where $A = A_{G, \infty}/A_{G, \rat}$. We want to change the integral to be over $\wh{Z'_\infty}$. The measure chosen on on $Z'_\infty$ induces Plancherel measure on $\wh{Z'_\infty}$ which restricts to a measure on $\wh A$ lying co-discretely inside. This corresponds to the quotient measure on $A$ coming from setting $\vol(Z^1_\infty) = 1$. Therefore, if we had EP-measure on $Z'_\infty$, our fixed EP-measure on $\wh A$ would have matched that on $\wh{Z'_\infty}$. The choices of measures we made also fix EP-measure on $G^1_\infty$ so the volume factor becomes $1$. 

In general, let $\mu = \mu_{Z'_\infty}/\mu^{EP}_{Z'_\infty}$. Then the identity finally simplifies to
\[
\int_{\wh{Z'_\infty}} \psi(\om) \overline{f \eta_\xi}(\om) \, d\om^\pl = \mu  \int_{\wh{Z'_\infty}}\psi(\om_\xi \om)  \1_{\wh A}(\om)  \wh f(\om) \, d\om^\pl
\]
for any test function $\psi$. Therefore we get
\[
\overline{f \varphi_\infty}(\om) = \mu \delta_{\om|_{Z^1_\infty} = \om_\xi|_{Z^1_\infty}} \wh f(\om \om_\xi^{-1}).
\] 
In our case $A_{G, \infty} \subseteq \mf X_\infty$ so for $\om|_{\mf X_\infty} = \om_\xi|_{\mf X_\infty}$, this simplifies to
\[
\overline{f \varphi_\infty}(\om) = \mu \delta_{\om_\infty = \om_\xi} \wh f(0).
\]

Next, let $Z_{U^{S, \infty}}  = U^{S, \infty} \cap Z^1$. Since it is an integral over a subgroup
\[
\bar \varphi^{S, \infty} (\om^{S, \infty}) = \begin{cases} \vol(Z_{U^{S, \infty}} ) & \om^{S, \infty} |_{Z_{U^{S, \infty}} } = 1 \\ 0 & \text{else} \end{cases}.
\] 
In total, cancelling the $\wh f(0)$ factors, the terms that do not vanish are 
\[
\f{ \mu \vol(Z_{U^{S, \infty}} )}{\vol(Z^1/Z_G(F))} \bar \varphi_S(\om_S)
\]
for every character $\om$ satisfying
\begin{enumerate}
\item $\om(Z_G(F)) = 1$,
\item $\om|_{\mf X} = \chi$, 
\item $\om_\infty = \om_\xi$,
\item $\om^{S, \infty}(Z_{U^{S, \infty}} ) = 1$.
\end{enumerate}

We try to characterize such $\om$. Consider $\om = \om_\infty \om_S \om^{S, \infty}$. Let $L = Z_G(F) \cap U^{S, \infty}$. These conditions require that $\om_S \om_\infty = 1$ on $L$ and that $\om_S \chi_S^{-1} = 1$ on $\mf X_S$. Given $\om_S$ satisfying this, the conditions determine $\om^{S, \infty} = \om_S^{-1} \om_\infty^{-1}$ on $Z_G(F)$. Since the determined $\om^{S, \infty}$ is trivial on $Z_G(F) \cap U^{S, \infty}$, it extends to a continuous character on $\overline{Z_G(F)} \subseteq Z^{S, \infty}$. The character $\om^{S, \infty}$ is also determined on $U^{S, \infty}$ and $\mf X^{S, \infty}$ so in total, the possible choices of $\om^{S, \infty}$ are those that restrict to a particular value on $E^{S, \infty} = \overline{Z_G(F)}Z_{U^{S, \infty}}\mf X^{S, \infty}$. 

 Since quotient maps of groups are open, $Z_{U^{S, \infty}}$ is open mod $\overline{Z_G(F)}$. Therefore, since $Z^{S,\infty}/\overline{Z_G(F)}$ is compact, $Z^{S, \infty}/E^{S, \infty}$ is a finite group. Therefore the choices are in bijection with $Z^{S, \infty}/E^{S, \infty}$.

By comparing $U^{S, \infty}$ times a fundamental domain for $Z^1/Z_G(F)$ to a fundamental domain for $Z^1_{S, \infty}/L$, we get
\[
\f{\vol(Z_{U^{S, \infty}})}{\vol(Z^1/Z_G(F))} = \f1{\vol(Z'_{S, \infty}/L)|Z^{S, \infty}/\overline{Z_G(F)}Z_{U^{S, \infty}}|}.
\]
Therefore, pulling out just the non-zero terms in the sum gives 
\[
\f1{|X|} \f\mu{\vol(Z'_{S, \infty}/L)} \sum_{\substack{\om_S \in \wh Z_S \\ \om_S \om_\xi^{-1}(L) = 1 \\ \om_S \chi_S^{-1}(\mf X_S) = 1}} \bar \varphi_S(\om_S)
\]
where
\[
|X|^{-1} = \f{|Z^{S, \infty}/\overline{Z_G(F)}Z_{U^{S, \infty}}\mf X^{S, \infty}|}{|Z^{S, \infty}/\overline{Z_G(F)}Z_{U^{S, \infty}}|} = |\overline{Z_G(F)}Z_{U^{S, \infty}}\mf X^{S, \infty}/\overline{Z_G(F)}Z_{U^{S, \infty}}|^{-1}.
\]
An application of corollary \ref{centerft} to $G_S/\mf X_S$ then finishes the argument.
\end{proof}

The formula here is complicated and requires some discussion. First, $\om_\xi$ determines a character on $L$ consistent with $\chi_S$. Therefore, $\om_\xi$ and $\chi_S$ together determine a character $\lb$ on $L\mf X_S$. The term $E(\wh \varphi_S | \om_\xi, L, \chi_S)$ can be thought of as some sort of normalized average of $\hat \varphi_S$ along representations with central character extending $\lb$. 

Note that if $Z_G$ is compact and $\mf X = A_{G, \rat} = 1$, we can choose a measure so that $\mu(Z_v) = 1$ for all $v$. This gives $\mu = 1$ so 
\[
\f1{|X|} \f\mu{\vol(Z_{S, \infty}/L)} = \f1{\mu(Z_{S,\infty}/L)} = |L| = |Z_G(F) \cap U^{S, \infty}|
\]
and $\wh Z_S$ has the counting measure. Therefore, $E^\pl(\wh f | \om)$ is the literal integral of $f \, d\mu^\pl$ over representations with character $\om$. This is in line with the result in \cite{KSZ16}. 

This computation can be compared to the very short argument at the beginning of \cite[\S2]{FL18}.  Reconciling notation, $\Theta$ in that paper is the same as $L$ here and $S$ there is $S \cup \infty$ here. Our argument is much longer since we are factoring out the infinite part of $\mu_{\Theta, S}$ requiring a sum over a complicated set of $\om_S$ instead of just a term for $E^\pl(\varphi_{S, \infty} | 1)$.  In addition, issues involving $\mf X$ appear.

\subsubsection{Main term bound}
It will also be useful to have a very rough bound on the magnitude of this main term.
\begin{prop}\label{roughwt}
Let  $\varphi_{S_1} \in \ms H(G_{S_1}, K_{S_1}, \chi_{S_1})^{\leq \kappa}$ such that $|\chi_{S_1}(x) \varphi_{S_1}(x)| \leq 1$ for all $x$. Then for some constant $C$ depending only on $G$, the main term \eqref{main} is $O_{\varphi_{S_0}}(q_{S_1}^{C \log \kappa})$
where the implied constant is independent of $\varphi_{S_1}$ and the $L$-packet weight $\xi$. \red{LEVEL ASPECT: The implied constants can in addition be chosen uniformly over all hyperendoscopic groups of $G$.}
\end{prop}

\begin{proof}
Start with the expression \eqref{main}:
\[
\f1{\wh f(0) \eta_\xi(1)} \f1{\vol(\mf X_F \bs \mf X/ A_{G, \rat})}  \int_{\mf X_F \bs \mf X/ A_{G, \rat}} \chi(z)  \sum_{\gamma \in Z_G(F)} f(z_\infty \gamma) \varphi^1(z\gamma)  \, dz.
\]
Here it is actually convenient to evaluate the integral, giving the central terms in \ref{sgeomsidegenps}:
\[
\f1{\vol(\mf X_F \bs \mf X/ A_{G, \rat})} \sum_{\gamma \in [Z_G(F)]^\ssm_{\mf X}} \om_\xi^{-1}(\gamma) \varphi(\gamma).
\]
The sum becomes 
\[
\sum_{\gamma \in [Z_G(F)]^\ssm_{\mf X}} \chi_{S_1} (\gamma) \varphi_{S_1} (\gamma) \chi_{S_0}(\gamma) \varphi_{S_0}(\gamma) \varphi^{S, \infty}.
\]
By construction, $\varphi^1_{S_1}$ and $(\varphi^1)^{S, \infty}$ intersect every $\mf X$-class in $Z_G(F)$ that $\varphi_{S_1}$ does. Pick a $\varphi^1_{S_0}$ with the same property. Finally let $U_\infty \subset Z_\infty$ be such that every point with non-zero summand can be translated into it. We will choose specific $U_\infty$ later. 

We may then instead bound
\[
\sum_{\gamma \in [L]^\ssm_{\mf X}} \1_{U_\infty}(\gamma) \chi_{S_1} (\gamma) \varphi^1_{S_1} (\gamma) \chi_{S_0}(\gamma) \varphi^1_{S_0}(\gamma)
\]
where $L = Z_G(F) \cap U^{S, \infty}$. 
We will do this by first bounding the number of terms in this sum by the size of $L \cap U_\infty \Supp \varphi_S$. 

If $K_s$ are the chosen maximal compacts, for each $s \in S_1$, $\varphi^1_s \in \ms H(G_s, K_s)^{\leq \kappa}$ so $\varphi^1_s$ is a linear combination of indicator functions $\1_{K_s \lb(\om) K_s}$ for a number of possible $\om$ that is polynomial in $\kappa$. \red{LEVEL ASPECT: The polynomial here is a count of roots with bounded norm so it can be chosen uniformly over hyperendsocopic groups.} Therefore, for some constant $C$, $\varphi_{S_1}$ is supported on a union of $O(\kappa^{C |S_1|})$ double cosets of $K_{S_1}$. Since $\varphi^1_{S_0}$ is compactly supported, this gives that $\varphi^1_S$ is supported on a union of $O_{\varphi^1_{S_0}}(\kappa^{C |S_1|})$ double cosets of $K_S$. Note that $\kappa^{C |S_1|} \leq \kappa^{C \log q_{S_1}} = q_{S_1}^{C \log \kappa}$. 

Let $Z_{K_S} = Z_S \cap K_S$ be the maximal compact for abelian $Z_S$. Consider the double coset $D = K_S \alpha K_S$. If $D \cap Z_S \neq \emptyset$, without loss of generality let $\alpha$ be in the intersection. Then $D = \alpha K_S$ and $D \cap Z_S$ is a union of cosets of $Z_{K_S}$ in $Z_S$. Consider two of these cosets $xZ_{K_S}$ and $yZ_{K_S}$. Then there exists $k \in K_S$ such that $x = ky \implies k = xy^{-1} \implies k \in Z_S$. Therefore, $x \in Z_{K_S}$ and the two cosets are equal. In total, $D \cap Z_S$ is either empty or a coset of $K_S$. This finally implies that $\Supp \varphi \cap Z_S$ is contained in a union of $O_{\varphi^1_{S_0}}(q_{S_1}^{C \log \kappa})$ cosets of $Z_{K_S}$.

To continue, we need to choose a particular $U_\infty$. First, $Z_\infty$ factors as $A_{G, \infty}/A_{G, \rat}$ times a compact real torus $Z_c$. Let $U'_\infty$ be some subset of $A_{G, \infty}/A_{G, \rat}$ and choose $f$ to be the pullback of the characteristic function of $U'_\infty$ through $H_{G_\infty}$  (we are not technically allowed to do this due to the smoothness restriction but we can take a close enough approximation in $L^1$). Then $f$ has support on $U_\infty = U'_\infty \times Z_c$. 

Let $c_1 = |L \cap Z_{K_S} U_\infty|$ and assume for now that this is finite. If coset $C = \alpha_S Z_{K_S} U_\infty$ contains an element of $L$, without loss of generality let this element be $\alpha_S$. Multiplying by $\alpha_S^{-1}$ bijects $L \cap C$ to $L \cap Z_{K_S} U_\infty$ so $|L \cap C| = c_1$. Counting all possible cosets, $|L \cap \Supp(f\varphi_S)| = O_{\varphi_{S_0}}(c_1q_{S_1}^{C \log \kappa})$. By a similar argument, $|L \cap \Supp(f\varphi_S)_z| = O_{\varphi_{S_0}}(c_zq_{S_1}^{C \log \kappa})$ where $c_z = |L \cap Z_{K_S} z_\infty^{-1} U_\infty|$. 

It remains to bound
\[
c_z = |Z_G(F) \cap Z_{K_S} Z_{U^{S, \infty}} z_\infty^{-1} U_\infty| \leq |Z_G(F) \cap Z_{K_S} Z_{K^{S, \infty}} z_\infty^{-1} U_\infty| 
\]
where $K^{S, \infty}$ is the maximal compact (since $Z^{S, \infty}$ is abelian). This is finite since $Z_G(F)$ is discrete inside $Z/A_{G, \rat}$. Then, $Z_G(F) \cap  Z_{K_S} Z_{K^{S, \infty}}$ is a co-compact lattice inside $Z_\infty$. It is still so when projecting down to $A_{G, \infty}/A_{G,\Q}$. Choose $U'_\infty$ to be a fundamental domain for this lattice. Then $c_z = 1$ for all $z$ and $\wh f(0) = \vol(U'_\infty)$ which depends only on $G$. 

Finally, the terms in the sum all have norm $1$ up to the factor $\chi_{S_0} \varphi_{S_0}^1$ that depends on $\varphi_{S_0}$. Therefore the sum is $O_{\varphi_{S_0}}( q_{S_1}^{C \log \kappa})$ for all $z$. The factor in front depends only on $(G, \mf X)$ so the entire term is $O_{\varphi_{S_0}, G}(q_{S_1}^{C \log \kappa})$ \red{LEVEL ASPECT: this part still needs uniformizing over hypernendoscopic groups} . 
\end{proof}

\subsection{The Error Term}
We need to do a few things to bound the error term.  First, the orbital integral bounds used only apply to elements in $\ms H(H_v, K_{H,v})^{\leq \kappa}$ so we need to extend them to spaces like $\ms H(H_v, K_{H,v}, \chi)^{\leq \kappa}$. 

Second, a given group has infinitely many endoscopic groups. Unfortunately, the alternate proof of orbital integral bounds in \cite[\S B]{ST16} gives no control over constants and $S_{\bad'}$. Therefore, it is useful to have some result that allows the use of the same constants and a choice of uniform $S_\bad$. 

\red{LEVEL ASPSECT: Next, we add a new orbital integral bound to combine both the weight and level aspect bounds.}
Finally, we need to do another due-diligence check that one, all the lemmas used in the proofs of theorems \ref{owtbound} and \ref{lvbound} still hold over to the non-trivial center case, and two, all the constants from those lemmas can also be uniformly bounded over all hyperendoscopic groups that contribute a non-zero term. This in particular uses the correction to \cite[cor 6.17]{ST16}.

\subsubsection{Uniform bounds for orbital integrals}

The model-theoretic method for bounding orbital integrals gives the following
\begin{thm}[{\cite[thm B.2]{ST16}}]\label{ointbound}
Let $\Xi$ be the root datum for an unramified group over some non-Archimedean local field (so the Galois action is determined by the Frobenius action). Choose a norm of the form $\|\cdot\|_\mc B$ on $X_*(A)$. Then there exist $T_\Xi, a_\Xi, b_\Xi$ depending only on $(\Xi, \|\cdot\|_\mc B)$ such that for all non-Archimedean local fields $F$ (including ones of positive characteristic) with residue field degree $q \geq T_\Xi$ the following holds: 

Let $G^F$ be the unramified group over $F$ with root datum $\Xi$, $K$ a hyperspecial of $G^F$, $A$ a maximal split torus, and $\varpi$ a uniformizer for $F$. Then for all $\lb \in X_*(A)$ with $\|\lb\| \leq \kappa$ and semisimple $\gamma \in G^F(F)$:
\[
|O_\gamma(\tau^{G^F}_\lb)| \leq q^{a_\Xi \kappa + b_\Xi} D^{G^F}(\gamma)^{-1/2}
\]
where as before, $\tau^{G^F}_\lb = \1_{K \lb(\varpi) K}$.
\end{thm}

\begin{note}
Elements of $\ms H(G_{S_1}, K_{S_1})^{\leq \kappa}$ are linear combinations of a number of $\1_{K \lb(\varpi) K}$ that is bounded by a polynomial in $\kappa$. Therefore, this can be used to get a bound of shape
\[
|O_\gamma(\varphi_{S_1})| = O(\|\varphi_{S_1}\|_\infty q_{S_1}^{a_\Xi \kappa + b_\Xi} D^{G}_{S_1}(\gamma)^{-1/2})
\]
for $\varphi_{S_1} \in \ms H(G_{S_1}, K_{S_1})^{\leq \kappa}$ and where we slightly increase $a_\Xi$ to absorb the $\kappa^{d|S_1|}$ factor from the polynomials. 
\end{note}

\begin{note}
We actually need a bound on  $\varphi \in \ms H(G_{S_1}, K_{S_1}, \chi_{S_1})^{\leq \kappa}$. By shifting double cosets by central elements, we can extend it at the cost of a factor of $|\chi^{-1}_{S_1}(\gamma_{S_1})|$. (Recall that this is well defined by the note in section \ref{setup}). 
\end{note}

By the following lemma, we can choose $a_\Xi, b_\Xi$, and $T$ uniformly over all $H$ appearing in an endoscopic path of $G$ and all places $v$ where $H$ is unramified:
\begin{lem}\label{finnumgroups}
Let $H$ be a group appearing in a hyperendoscopic path for $G$, $M_H$ a Levi of $H$, $v$ a place where $H$ is unramified, and $\Xi$ the unramified root data for $(M_H)_v$. Then $\Xi$ is an element of a finite set depending only on $G$. 
\end{lem}

\begin{proof}
The (co)root spaces of $M_H$ are isomorphic to those of $G$ and the (co)roots of $M_H$ are a subset of those of $G$ so there are only finitely many possibilities for the root system of $M_H$ (without Galois action) since its rank is bounded by a finite number through iteratively applying lemma \ref{zextbound}. Then, there are only finitely many ways for Frobenius to map into the automorphisms of this root system.
\end{proof}

This bound is absurdly inefficient---in particular, it involves factorials nested to the degree of the semisimple rank of $G$. In any application one should use properties of the exact group being studied to describe the set more explicitly.

\subsubsection{Error term bound for weight aspect}
We can now show
\begin{prop}\label{weakuniform}
Assume that $\varphi_{S_1} \in \ms H(G_{S_1}, K_{S_1}, \chi_{S_1})^{\leq \kappa}$ with $\|\chi_{S_1} \varphi_{S_1}\|_\infty \leq 1$. Consider error term \eqref{error} for any group $H$ unramified on $S_1$ and appearing in an endoscopic path of $G$ with induced central character datum $(\mf X, \chi)$ such that $A_{H, \infty} \subseteq \mf X$ and $\chi$ is unramified on $S_1$. It is $O_{\varphi_0, H}(q_{S_1}^{A_{\wt,H}+ B_{\wt,H} \kappa} m(\xi)^{-C_{\wt, H}})$ for some constants $A_{\wt, H},B_{\wt, H}$, and $C_{\wt, H}$ as long as $S_1$ contains no fields with residue degree less than some $M_G$ that is uniform over all $H$. 
\end{prop}

\begin{proof}
Let $M_G$ be the maximum of the $T_\Xi$ from \ref{ointbound} over all root data $\Xi$ from lemma \ref{finnumgroups}. This is then a due-diligence check that all the steps in \cite[thm 9.19]{ST16} still hold. We start by evaluating the integral in \eqref{error} getting term
\begin{multline*}
\f{\vol(\mf X_\infty^1)}{\vol(\mf X_F \bs \mf X/ A_{H, \rat})} \lf( \sum_{\substack{\gamma \in [H(F)]^\ssm \\ \gamma \notin Z_H}} a_{H,\gamma} |\iota^H(\gamma)|^{-1} |\Stab_\mf X(\gamma)|^{-1} \f{\tr \xi(\gamma_\infty)}{\dim \xi} O^M_\gamma(\varphi^\infty_M) \ri. \\
+ \lf.\sum_{\substack{M \in \ms L^\cusp_H \\ M \neq H}} \sum_{\gamma \in [M(F)]^\ssm} a_{M,\gamma} |\iota^H(\gamma)|^{-1} |\Stab_\mf X(\gamma)|^{-1}\f{\Phi_M(\gamma, \xi)}{\dim \xi} O^M_\gamma((\varphi^\infty)_M) \ri).
\end{multline*}
Without loss of generality, expand $S_0$ so that $\varphi$ is the characteristic function of a hyperspecial $K^{S, \infty}$ away from $S \cup \infty$ and that the places less than $M_G$ are contained in $S_0$. If a conjugacy class intersects the support of $\varphi_{S_1}$, then we can scale it by an element of $\mf X_S$ so that it intersects the support of $\varphi_{S_1}^1$. The same holds for $\varphi^{S, \infty}$ which has support $K^{S, \infty}$. Choose $\varphi_{S_0}$ and $\varphi_S$ similarly and let their supports after taking constant terms to $M$ be $U_{S_0,M}$ and $U_{\infty, M}$. We can then replace terms in the sum through the rule
\[
\f{\Phi_M(\gamma, \xi)}{\dim \xi} O^M_\gamma((\varphi^\infty)_M) \mapsto \1_{U_{\infty, M}} \f{\Phi_M(\gamma, \xi)}{\dim \xi} O^M_\gamma((\varphi^\infty)^1_M).
\]
Let $U_{S_1,M} = \Supp \ms H^\ur(M_{S_1})^{\leq \kappa}$. 
Let $Y_{M}$ be the set of semisimple rational conjugacy classes intersecting the set $U_{S_1,M} U_{S_0,M} U_{\infty,M} K^{S, \infty}_M$. The number of terms in the sum is less than or equal to $|Y_M|$. 

We check that each of factors can be bounded as in the proof of \cite[thm 9.19]{ST16}. The finite set of places $S_{M, \gamma}$ disjoint from $S$ can be defined in the same way. Then:
\begin{itemize}
\item
\cite[cor 6.17]{ST16} still applies to $a_{M, \gamma}$,  (see the missing step lemma \ref{aboundgap} for why this works for general center).
\item
The bound in \cite[lem 6.11]{ST16} still applies to the $\Phi_M(\gamma, \xi)$ terms. There is an extra factor of $\chi_\infty^{-1}(\gamma_\infty)$. 
\item
A version of \cite[thm A.1]{ST16} modified to work on functions with central character still applies to bound $O_{\gamma}^M(\varphi_{S_0,M})$. There is an extra factor of $|\chi_{S_0}^{-1}(\gamma_{S_0})|$. \item
Proposition \ref{ointbound} still bounds $O_{\gamma}^M(\varphi_{S_1,M})$. There is an extra factor of $|\chi^{-1}_{S_1}(\gamma_{S_1})|$. 
\item
Proposition \ref{ointbound} still gives the same bound for $O_{\gamma}^M(\varphi_{v,M})$ for $v \in S_{M, \gamma}$ since $M_H \leq M_G$.There is again an extra factor of $|\chi|$. 
\item
\cite[lem 2.18]{ST16} and \cite[lem 2.21]{ST16} still provide a bound on the $D^M$ terms since we can still construct the embedding from \cite[prop 8.1]{ST16}.
\item
$|Y_{M}|$ can still be bounded by \cite[cor 8.10]{ST16} (this also applies to groups with general center).
\item
$|\Stab_\mf X(\gamma)|^{-1} \leq 1$.
\end{itemize} 
Since $\chi$ is trivial at rational elements, all the $\chi_v$ terms cancel. Therefore, the entire term can similarly be bounded by 
\[
O(q_{S_1}^{A_{\wt,H} + B_{\wt,H} \kappa} m(\xi)^{-c_{\wt,H}}),
\]
folding in the constant that only depends on $H$ and $\mf X$. 
\end{proof}
This very weak uniformity in just $M_G$ is all we will need for the weight aspect.

\red{\subsubsection{A new orbital integral bound}
LEVEL ASPECT}

\red{\subsubsection{More uniform error term bound}
LEVEL ASPECT}

\section{Final Computation}\label{final}
\subsection{Weight Aspect}
Assume the previous conditions on $(G, \mf X, \chi)$ from section \ref{cndtns}. Let $\pi_k$ be a sequence of discrete series representations of $G_\infty$ such that their corresponding finite-dimensional representations $\xi_k$ have regular weights $m(\xi_k) \to \infty$. Let $S_1$ be disjoint from $S_{\bad, G}$: the set of places with residue degree less than the uniform $M_G$ from proposition \ref{weakuniform}. Choose constant $S_0$, $\varphi_{S_0}$ and $U^{S, \infty}$ to define a sequence of families $\mc F_k$ for each $\xi_k$. 
\begin{thm}\label{mainresult}
There are constants $A'_{G, \wt}$ and $B'_{G, \wt}$ such that for any $\varphi_{S_0}$ and $\varphi_{S_1} \in \ms H(G_{S_1}, K_{S_1}, \chi_{S_1})^{\leq \kappa}$, 
\begin{multline*}
\f{\bar \mu^\can(U^{S, \infty}_\mf X)|\Pi_\disc(\xi_k)|}{\tau'(G) \dim(\xi_k)} \sum_{\pi \in \mc{AR}_\disc(G, \chi)} a_{\mc F_k}(\pi) \wh \varphi_S(\pi) \\
= E(\wh \varphi_S | \om_\xi, L, \chi_S) + O(q_{S_1}^{A'_{\wt} + B'_{\wt} \kappa} m(\xi_k)^{-1})
\end{multline*}
(using notation from corollary \ref{spectral} and theorem \ref{wtbound}). The constants in the error depend on $(G, \mf X, \chi)$, $\varphi_{S_0}$, and $U^{S, \infty}$. 
\end{thm}

\begin{proof}
For $\|\varphi_{S_1} \chi_{S_1} \|_\infty \leq 1$, let
\[
\varphi_k = \varphi_{\pi_k} \otimes \1_{U^{S, \infty}, \chi} \otimes \varphi_{S_1} \otimes \varphi_{S_0}
\]
as in section \ref{cndtns}. Let $\varphi^1_k = \eta_{\xi_k} \otimes \varphi_k^\infty$. Then $\varphi_k$ and $\varphi^1_k$ are unramified outside of $S$. Let $\mc A$ be the set of hyperendoscopic tuples that contribute a non-zero value to the hyperendoscopy formula as in lemma \ref{conditions}.
 
 Then using the hyperendoscopy formula:
\begin{multline*}
\f{|\Pi_\disc(\xi_k)|}{\tau'(G) \dim(\xi_k)} I_\disc(\varphi_k)  \\
= \f{|\Pi_\disc(\xi_k)|}{\tau'(G) \dim(\xi_k)}   \lf(I_{\disc}^G(\varphi_k^1) + \sum_{\mc H \in \mc A} \iota(G, \mc H)I_{\disc}^{H_{n_\mc H}}((\varphi_k^1 - \varphi_k)^{\mc H})
\ri).
\end{multline*}

 We choose arbitrary transfers of $\varphi_0$. 
 Choose $(\1_{K_G^{S, \infty}})^{\mc H}$ according to lemma \ref{zextfunlemma} since by lemma \ref{conditions}, $\mc H$ stays unramified away from $S, \infty$. Let $\Pi_\disc(\xi_k)$ be the $L$-packet containing $\pi_k$ and let its size be $X_k$. Then
\[
(\varphi_k^1 - \varphi_k)^\infty = \varphi_k^\infty \qquad (\varphi_k^1 - \varphi_k)_\infty = \f1{X_k} \sum_{\pi_k \neq \pi \in \Pi_\disc(\xi_k)} \varphi_\pi - \f{X_k - 1}{X_k} \varphi_{\pi_k}.
\]
By proposition \ref{archtransfer}, we can choose the infinite part transfer to be a linear combination of EP-functions  
\[
\sum_{\xi \in \Xi_{\xi_k, \mc H}} c_\xi \eta_\xi
\]
for some constants
\[
|c_\xi| \leq (X_k - 1) \f1{X_k} + \f{X_k - 1}{X_k} \leq 2.
\]
Now, checking some conditions:
\begin{itemize}
\item
All groups in the hyperndoscopic paths are unramified on $S_1$ and cuspidal at infinity. In addition, each $\mf X_{\mc H} \supseteq A_{\mc H, \infty}$ by lemma \ref{conditions}.
\item
Let $\chi_{\mc H}$ be the character determined by $\mc H$ as in section \ref{hyptransferchar}. The transfer $\chi_{\mc H,S_1} \varphi_{S_1}^\mc H$ can be chosen to be in $\ms H(G_{S_1}, K_{S_1}, \chi_{\mc H, S_1})^{\leq \kappa}$ and have $L^\infty$-norm bounded by some $q_{S_1}^{E_\mc H \kappa} \kappa^{|S_1|}$ by repeated application of proposition \ref{valuebound}. We can apply this due to the above.
\item
The $\xi$ are regular by lemma \ref{regbound}.
\item
Without loss of generality, enlarge $S_0$ so that $U^{S, \infty} = K^{S, \infty}$. Then $\1_{U^{S, \infty}}^\mc H$ is still the indicator function of an open compact subgroup averaged over $\chi_\mc H^{S, \infty}$. 
\end{itemize}

We can therefore apply the main term bound in proposition \ref{roughwt} and the error term bound in propostion \ref{weakuniform} to each term in the sum and get
\begin{multline*}
I_{\disc}^{\mc H}((\varphi_k^1 - \varphi_k)^{\mc H}) = I_{\spec}^{\mc H}((\varphi_k^1 - \varphi_k)^{\mc H}) = \\
\sum_{\xi \in \Xi_{\xi_k, \mc H}} \iota(G, \mc H) \f{\tau'(\mc H) \dim(\xi)}{|\Pi^{\mc H}_\disc(\xi)|} O_{\varphi_0^\mc H, U^{S, \infty}, \mc H}(q_{S_1}^{(A_{\wt, \mc H} + E_\mc H + \eps) \kappa + B_{\wt, \mc H} })
\end{multline*}
for some constant $U_{\xi, \mc H}$. We use here that $O(\kappa^{C|S_1|}) O(q_{S_1}^{(A + E) \kappa + B}) = O(q_{S_1}^{(A + E+ \eps) \kappa + B})$.

By the computation in \ref{wtbound} and the error term bound \ref{weakuniform},
\[
\f{|\Pi_\disc(\xi_k)|}{\tau'(G) \dim(\xi_k)} I_{\disc}^G(\varphi_k^1) = E + O(q_{S_1}^{A_{\wt, G} + B_{\wt, G} \kappa_1} m(\xi_k)^{-C_{\wt, G}})
\]
where we shorthand $E = E(\wh \varphi_S | \om_\xi, L, \chi_S)$. Multiplying through,
\begin{multline*}
\f{|\Pi_\disc(\xi_k)|}{\tau'(G) \dim(\xi_k)} I_\disc(\varphi_k)  =  \\
E + \sum_{\mc H \in \mc A} \sum_{\xi \in \Xi_{\xi_k, \mc H}} W_{\xi, \mc H} \f{\dim(\xi)}{\dim(\xi_k)} O_{\mc H, \varphi_0^\mc H}(q_{S_1}^{(A_{\wt, \mc H} + E_\mc H+ \eps) \kappa + B_{\wt, \mc H} }) \\+  O(q_{S_1}^{A_{\wt, G}\kappa + B_{\wt, G} } m(\xi_k)^{-C_{\wt}})
\end{multline*}
where
\[
\lf| W_{\xi, \mc H} = \iota(G, \mc H) \f{\tau'(\mc H)}{\tau'(G)} \f{{|\Pi^{G}_\disc(\xi_k)|}}{{|\Pi^{\mc H}_\disc(\xi)|}} \ri| \leq W
\] 
for some constant $W$ independent of $\xi_k$, $k$, and $q_{S_1}$. Finally, by lemma \ref{dimbound} the ratio of dimensions is $O(m(\xi_k)^{-1})$. 

In total, the inner sum has $|\Om_\mc H|$ elements so the entire double sum has finite size independent of $S_1$ and $\xi$. Therefore, it can be bounded to
\[
\f{|\Pi_\disc(\xi_k)|}{\tau'(G) \dim(\xi_k)} I_\disc(\varphi_k) = E + O(q_{S_1}^{A'_{\wt} + B'_{\wt} \kappa} m(\xi_k)^{-1}).
\]
where $A'_\wt, B'_\wt$ are anything bigger than the maxima over all groups appearing in $\mc A$ (Note that $C_{\wt}$ can be chosen to be $\geq 1$). Finally, plug in corollary \ref{spectral}.
\end{proof}

\red{\subsection{Level aspect}
For this, will need transfers of level subgroup indicators, fix Ferrari!}

\red{will also need much stronger uniformity and a bound on how many groups you get in $\mc A$}

\red{some notes: only finitely many possible Levis of endoscopic groups possible at places $v \in S_0$ so can use uniform constants in Shalika germ bound}

\red{Will also need to bound ratios
\[
\f{\mu^\can(K_G(\mf p^n))}{\mu^\can(K_{\mc H}(\mf p^n))}.
\]
}

\section{Corollaries}\label{corollaries} 
Theorem \ref{mainresult} can be substituted in for \cite{ST16}'s 9.19 to most of the same corollaries. We leave the result on zeros of $L$-functions for the future because the computations are complicated---the term $\beta^\pl_v$ gets replaced by something far more complicated in the case with central character.

Recall the notation from last section and for brevity define
\[
\mu_{\mc F_k}(\wh \varphi_S) = \f{\bar \mu^\can(U^{S, \infty}_\mf X)|\Pi_\disc(\xi_k)|}{\tau'(G) \dim(\xi_k)} \sum_{\pi \in \mc{AR}_\disc(G, \chi)} a_{\mc F_k}(\pi) \wh \varphi_S(\pi) 
\]
for any $\wh \varphi_S$ on $\wh G_S$. Theorem \ref{mainresult} computes this when $\varphi_S \in \ms H(G_S)$ and $\varphi_{S_1}$ is unramified.

\subsection{Plancherel Equidistribution}
First, we get a version of \cite[cor 9.22]{ST16} using a similar Sauvageot density argument. We thank the reviewers for pointing out that \cite[pg. 103]{NV19} discusses a discovered gap in \cite[pg. 181]{Sau97}. Here, Bernstein components are conflated with ``$\mf l$-components'' (defined in \cite[\S 23.7]{NV19}) when arguing that a certain algebra separates points. Most of this section is dependent on that gap being fixed.

We phrase things as in \cite{FL18}. Restrict to the case where all the $\xi_k$ have the same central character $\om_\xi$ and $S_1$ is trivial. Let $\Theta = L \mf X_S$ and let $\psi$ be the character on $\Theta$ induced by $\om_\xi$ and $\chi_S$. Let $\wh G_{S,\psi} \subseteq \wh G_S $ be all representations with central character extending $\psi$. We can define a measure $\mu^\pl_\psi$ on $\wh G_{S,\psi}$ by $\mu^\pl_\psi(f) = E(f^* | \om_\xi, L, \chi_S)$ where $f^*$ is a continuous extension of $f$ to $\wh G_S$. 

The non-conditional result is a quick consequence of \ref{mainresult} restated in the terminology of this section:
\begin{cor}[Unconditional Plancherel equidistribution up to central character]
For any $f \in \ms H(G_S, \chi_S)$, 
\[
\lim_{k \to \infty} \mu_{\mc F_k}(\wh f) = \mu^\pl_\psi(\wh f).
\]
\end{cor}
Note that there is no uniformity in this result---the rate of convergence depends heavily on the exact $f$. 

{\color{gray}The following is all conditional on \cite{Sau97}: 
 When $\psi$ is trivial, $\mu^\pl_\psi = \mu_{\Theta, \pl}$ from \cite{FL18} up to some constant. The lemma in the middle of the proof of \cite[thm 2.1]{FL18} extends to our case of non-trivial $\psi$ and $\Theta$ a general subgroup of $Z_{G_S}$. 
\begin{lem}
Let $\eps > 0$:
\begin{enumerate}
\item
For any bounded $A \subseteq \wh G_S \setminus \wh G_S ^\temp$, there exists $h \in \ms H(G_S)$ such that $\wh h \geq 0$ on $\wh G_S $, $\wh h \geq 1$ on $A$, and $\mu^\pl_\psi(\wh h) \leq \eps$.
\item
For any Riemann integrable function $\wh f$ on $\wh G_{S,\psi}^\temp$, there exist $h_1, h_2 \in \ms H(G_S)$ such that $|\wh f - \wh h_1| \leq \wh h_2$ on $\wh G_{S,\psi}$ and $\mu^\pl_\psi(\wh h_2) \leq \eps$. 
\end{enumerate}
Both these results depend on 
\end{lem} 

\begin{proof}
We try to mimic the argument in \cite[thm 2.1]{FL18}. Let $\Theta_f = \Theta \cap Z_{G_\der}(F_S)$ and $\overline \Theta = \Theta/\Theta_f$. Then $\Theta_f$ is finite. In addition, if we denote by $X(\cdot)$ taking complex-valued characters, the map $X(G_S) \to X(Z_{G,S}/Z_{G_\der}(F_S)) \to X(\overline \Theta)$ is surjective. Choose a set-theoretic section $s$ of this map. 

We can ignore normalization constants by, without loss of generality, changing $\eps$. Then this result for $\Theta$ trivial follows from the main result of \cite{Sau97}. If $\overline \Theta$ is trivial, then the various $\wh G_{S, \psi}$ are positive-measure clopen subsets of $\wh G_S$ so we can use the $h_i$ for either $A$ or the extension of $f$ by $0$ on $\wh G_S$. 

For the general case, given $f$ on $\wh G_{S, \psi}$ define $F$ on $\wh G_{S, \psi|_{\Theta_f}}$ by $F(\pi) = f(\pi \otimes s(\om_\pi^{-1} \psi))$. Choose $H_1$ and $H_2$ satisfying the condition for $F$. For any finite subset $T_0\subseteq X(\overline \Theta)$, the averages
\[
h_i = \f1{|T_0|} \sum_{\lb \in T_0}  s(\lb)H_i
\]
satisfy $|\wh f - \wh h_i| \leq \wh h_2$ (each individual term in the sum does) so we simply need to find an $T_0$ such that $\mu^\pl_\psi(h_2) \leq \eps$. 

Up to some constants
\[
\mu^\pl_\psi(h_2) = \int_{\Theta} \psi(z) h_2(z) \, dz  = \f1{|T_0|} \sum_{\lb \in T_0}  \overline H_2(\lb \psi) \qquad \mu^\pl_{\psi|_{\Theta_f}}(H_2) = \sum_{z \in \Theta_f} \psi(z) H_2(z)
\]
by variations of the arguments in section \ref{cft}. Choose a sparse enough lattice $M$ so that the support of $H_2$ intersect $M$ is $\{1\}$. Then, up to constants,
\[
\eps > \mu^\pl_{\psi|_{\Theta_f}}(H_2) = \sum_{z \in \Theta_fM} \psi(z) H_2(z) = \sum_{\lb \in T} \overline H_2(\lb \psi),
\] 
where $T$ is some subset of $X(\overline \Theta)$. The last step was Poisson summation on $\Theta$. Therefore, since the last sum converges, choosing $T_0$ to be a large enough finite subset of $T$ suffices. 

The argument for subsets $A$ is the same averaging trick---in place of the function $F$, we use set $A' = \{\pi \otimes \lb: \pi \in A, \lb \in X(\overline \Theta)\}$.  
\end{proof}
The same ``$3 \eps$''-argument as \cite[cor 9.22]{ST16} then gives:
\begin{cor}[Plancherel equidistribution up to central character]\label{plancherel}
Recall the conditions and notation from the above discussion. Then
\begin{enumerate}
\item
For any bounded $A \subseteq \wh G_S \setminus \wh G_S ^\temp$,
\[
\lim_{k \to \infty} \mu_{\mc F_k}(\1_A) = 0.
\]
\item
For any Riemann integrable $\wh f$ on $\wh G_{S,\psi}^\temp$,
\[
\lim_{k \to \infty} \mu_{\mc F_k}(\wh f) = \mu^\pl_\psi(\wh f).
\]
\end{enumerate}
\end{cor}
Beware that part (1) does not give a Ramanujan conjecture at $S$  on average; it cannot count that the total number of $\pi$ in $\mc F$ with non-tempered $\pi_S$ is $O(m(\xi_k))^{-1}$ since $A$ needs to be bounded. It is nevertheless somewhat close.}

\subsection{Sato-Tate Equidistribution}
For this section we need to slightly modify our notation. Allow $S_1$ to be infinite and define modified measure
\[
\mu_{\mc F_k,v}^\natural(\wh \varphi_v)  = \f{\bar \mu^\can(U^{S, \infty}_\mf X)|\Pi_\disc(\xi_k)|}{\tau'(G) \dim(\xi_k)} \sum_{\pi \in \mc{AR}_\disc(G, \chi)} a_{\mc F_k}(\pi) \wh \varphi_{S_0}(\pi_{S_0}) \wh \varphi_v(\pi_v) 
\]
for any $v \in S_1$. Then $\mu_{k,v}^\natural(\wh \varphi_v)$ can still be picked out by a test function $\varphi$ of the form we have been considering by setting $\varphi_w = \1_{K_w}$ for all $w \in S_1 \setminus v$. 
\subsubsection{Sato-Tate measures}
We recall the definition of the Sato-Tate measure from \cite[\S3,\S5]{ST16}. Recall the Satake isomorphism $\ms H(G_v, K_v) \to \C[X_*(A)]^{\Om_F}$ in the notation of section \ref{funlem} and how it identifies $\wh G_v^{\ur, \temp}$ with $\Om_{F_v} \bs \wh A_c$. 

We can find a maximal compact $\wh K$ of $\wh G$ invariant under $\Frob_v$. Then since $G_v$ is unramified, $\Om_{F_v} \bs \wh A_c$ can be identified with the $\wh G$ classes in $\wh K \rtimes \Frob_v \subseteq \Ld G$ and also $\wh T_{c, v} = \Om_{F_v} \bs \wh T_c/(\Frob_v - \id) \wh T_c$ (see \cite[lem 3.2]{ST16}).   

In general, let $G$ split over $F_1$ and let $\Gamma_1 = \Gal(F_1/F)$. Given $\Theta \in \Gamma_1$, define
\[
\wh T_{c, \Theta} = \Om_G^\Theta \bs \wh T_c/(\Theta - \id) \wh T_c.
\]
Given $\tau \in \Gamma_1$, $t \mapsto \tau t$ canonically identifies $T_{c, \Theta}$ with $T_{c, \tau \Theta \tau^{-1}}$. All these identifications are consistent with each other so $T_{c, \Theta}$ depends only on the $\Gamma_1$-conjugacy class of $\Theta$. Note then that $\wh T_{c, \Frob_v} = \wh T_{c, v}$ since $G_v$ is quasisplit.

Choose the Haar measure on $\wh K$ with total volume $1$. This induces a quotient measure on the set of conjugacy classes in $\wh K \rtimes \Theta$ and therefore on $\wh T_{c, \Theta}$. Call this $\mu^\ST_\Theta$.
Finally, let $\mc V_F(\Theta)$ be the set of places $v$ such that $F_1$ is unramified at $v$ and $\Frob_v$ is in the conjugacy class of $\Theta$. For such a $v$, we get a measure $\mu_v^{\pl,\ur}$ from the identification $T_{c, \Theta}$ with $\wh G_v^{\ur, \temp}$. Normalize this to also have total volume $1$. 

\begin{prop}[{\cite[prop 5.3]{ST16}}]
For any $\Theta \in [\Gamma_1]$, let $v \to \infty$ in $\mc V_F(\Theta)$. Then there is weak convergence $\mu_v^{\pl,\ur} \to \mu^\ST_\Theta$. 
\end{prop}

\begin{proof}
by the explicit formulas \cite[prop 3.3]{ST16} and \cite[lem 5.2]{ST16}
\end{proof}

\subsubsection{Central character issues}
Recall all the notation from proposition \ref{wtbound}. Our result is in terms of $E(\wh \varphi | \om_\xi, L, \chi_S)$ instead of $\mu_v^{\pl, \ur}$ so we need to define an alternate Sato-Tate measure in terms of this. First, we need to understand $E^{\pl, \ur}_v$ better.   

There is a central character map $T_{c, \Theta} \to \wh Z_{G_v}$. This lets us define $E^{\ST, \Theta}(\wh \varphi | \om)$ for any $\wh \varphi$ on $T_{c, \Theta}$ similar to $E^{\pl, \ur}_v(\wh \varphi | \om)$ from section \ref{cft}. Now Langlands for tori gives that $\wh Z_{G_v}$ is the set of $L$-parameters $\varphi : W_{F_v}^\ur \into \Ld{(Z_G)}_{F_v}^\ur$. If $\Frob_v$ and $\Frob_w$ are conjugate in $\Gamma_1$, we can identify the set of these parameters and therefore $\wh Z_{G_v}$ and $\wh Z_{G_w}$. For $v \in \mc V_F(\Theta)$, call this common set $\wh Z_\Theta$. Note that these identifications commute with the identifications of $\wh T_{c, v}$ and the map taking central characters.
\begin{lem}
Fix a common measure on $\wh Z_\Theta$. Choose $\wh \varphi_\Theta$ on $\wh T_{c, \Theta}$. Then $E^{\pl, \ur}_v(\wh \varphi | \om) \to E^{\ST, \Theta}(\wh \varphi | \om)$ pointwise for $\om \in \wh Z_\Theta$.
\end{lem}
\begin{proof}
The previous result gives weak convergence $E^{\pl, \ur}_v(\wh \varphi_\Theta | \om) \to E^{\ST, \Theta}(\wh \varphi_\Theta | \om)$ in $L^2(\wh Z_\Theta)$. By the formula \cite[prop 3.3]{ST16}, the $E^{\pl, \ur}_v(\wh \varphi | \om)$ are equicontinuous so this implies pointwise convergence.
\end{proof} 

To understand the more complicated $E(\wh \varphi | \om_\xi, L, \chi_S)$, we now have to parametrize $Z_{S, \xi, L, \chi}$ in terms of local components. Assume $\om_S = \om_{S_1} \om_{S_0} \in Z_{S, \xi, L, \chi}$: i.e. $\om_S \om_\xi = 1$ on $L$ and $\om_S|_{\mf X_S} = \chi_S$. Assume also that $\om_{S_1}$ is unramified. Let $L_0 = L \cap K_{S_1}$. It is a cocompact lattice in $Z_{S_0}$. Then we always have that $\om_{S_0} \om_\xi = 1$ on $L_0$ and that $\om_{S_0} |_{\mf X_{S_0}} = \chi_{S_0}$. 

Given such $\om_{S_0}$, it forces $\om_{S_1} = \om_{S_0}^{-1} \om_\xi^{-1}$ on $L$. The determined $\om_{S_1}$ is trivial on $L \cap K_{S_1}$ and therefore extends to a continuous character on $\overline L \subseteq Z_{S_1}$. Therefore, the possible choices for $\om_{S_1}$ are those that restrict to $\om_{S_0}^{-1} \om_\xi^{-1}$ on $L$, restrict to $\chi_{S_1}$ on $\mf X_{S_1}$, and are unramified. 

Let $E_{S_1}$ be the group $\overline L K_{S_1} \mf X_{S_1}$. Since $Z_{S_1}/E_{S_1}$ is finite, there are finitely many choices for $\om_{S_1}$ and we can factor
\[
\sum_{\substack{\om_S \in \wh Z_S \\ \om_S \om_\xi(L) = 1 \\ \om_S|_{\mf X_S} = \chi_S}} E^\pl(\wh \varphi_S | \om_S) = \sum_{\substack{\om_{S_0} \in \wh Z_{S_0} \\ \om_{S_0} \om_\xi(L_0) = 1 \\ \om_{S_0}|_{\mf X_{S_0}} = \chi_{S_0}}} E^\pl(\wh \varphi_{S_0} | \om_{S_0}) \sum_{\substack{\om_S \in \wh Z_{S_1}^\ur \\ \om_{S_1} \om_{S_0} \om_\xi(L) = 1 \\ \om_{S_1}|_{\mf X_{S_1}} = \chi_{S_1}}} E^\pl(\wh \varphi_{S_1} | \om_{S_1}).
\]
To compute $\mu_{k,v}^\natural$, we consider $\varphi_{S_1} = \1_{K_{S_1 \setminus v}} \varphi_v$ so
\[
E^\pl(\wh \varphi_{S_1} | \om_{S_1}) = E^\pl(\wh \varphi_v | \om_v) \prod_{w \in S_1 \setminus v} \vol(Z_w \cap K_w) = \f{\vol(Z_{S_1} \cap K_{S_1})}{\vol(Z_v \cap K_v)} E^\pl(\wh \varphi_v | \om_v).
\]

Let the set of summands for the second sum be $\wh Z_{v, \om_{S_0}, \chi_v} \subseteq \wh Z_v$ and let $\om_S \in \wh Z_{v, \om_{S_0}, \chi_v}$. The possible $\om_v$ components are those satisfying two conditions: $\om_v \om_{S_0} \om_\xi$ extends continuously to $\overline L \subseteq Z_{S \setminus v}$, and $\om_v|_{\mf X_v} = \chi_v$. The first condition is equivalent to $\om_v$ being the $F_v$-component of a character $\om$ on $Z_G(\A)/Z_G(F)$ trivial on $U^{S, \infty}$ that also has $F_{S_0, \infty}$-component $\om_{S_0} \om_\xi$. 

Next, by global Langlands for tori, this is equivalent to its parameter $\psi_{\om_v} : W_{F_v} \to \Ld{(Z_G)}_{F_v}$ being a restriction of a global parameter $\psi_\om : W_F \to \Ld Z_G$ satisfying certain conditions. However, if $\Frob_w$ is conjugate to $\Frob_w$, then $\psi_\om|_{W_{F_w}}$ is the transport of $\psi_\om|_{W_{F_v}}$ through the identification before. In particular, if we identify all the $\wh Z_v$ for $v \in \mc V_F(\Theta) \cap S_1$, $\wh Z_{v, \om_{S_0}, \chi_v}$ depends on $v$ only through $\Theta$. Call the common value $\wh Z_{\Theta, \om_{S_0}, \chi_v} \subseteq \wh Z_\Theta$. 

In total, if we set $\varphi_{S_1} = \1_{K_{S_1 \setminus v}} \varphi_v$ for some $v \in \mc V_F(\Theta) \cap S_1$,
\begin{multline*}
E(\wh \varphi_S | \om_\xi, L, \chi_S) = \f1{|X|} \f{\mu}{\vol(Z'_{S, \infty}/L)} \f{\vol(Z_{S_1} \cap K_{S_1})}{\vol(Z_v \cap K_v)} \\
 \sum_{\substack{\om_{S_0} \in \wh Z_{S_0} \\ \om_{S_0} \om_\xi(L_0) = 1 \\ \om_{S_0}|_{\mf X_{S_0}} = \chi_{S_0}}} E^\pl(\wh \varphi_{S_0} | \om_{S_0}) \sum_{\om_v \in \wh Z_{\Theta, \om_{S_0}, \chi_v}} E^\pl(\wh \varphi_v | \om_v).
\end{multline*}
This allows us to define an $E_{\ST, \Theta}(\wh \varphi_v | \om_\xi, L, \chi_S, \wh \varphi_{S_0})$ analogously:
\begin{multline*}
E_{\ST, \Theta}(\wh \varphi_S | \om_\xi, L, \chi_S, \wh \varphi_{S_0}) = \f1{|X|} \f{\mu}{\vol(Z'_{S, \infty}/L)} \f{\vol(Z_{S_1} \cap K_{S_1})}{\vol(Z_v \cap K_v)} \\
 \sum_{\substack{\om_{S_0} \in \wh Z_{S_0} \\ \om_{S_0} \om_\xi(L_0) = 1 \\ \om_{S_0}|_{\mf X_{S_0}} = \chi_{S_0}}} E^\pl(\wh \varphi_{S_0} | \om_{S_0}) \sum_{\om_v \in \wh Z_{\Theta, \om_{S_0}, \chi_v}} E^{\ST, \Theta}(\wh \varphi_v | \om_v).
\end{multline*}
 Then we get:
\begin{prop}
Choose a sequence $v \to \infty$ in $\mc V_F(\Theta) \cap S_1$ such that the characters $\chi_v$ all correspond in $\wh{\mf X}_\Theta$. Choose $\wh \varphi_\Theta$ on $\wh T_{c, \Theta}$. Then
\[
E(\wh \1_{K_{S_1 \setminus v}} \wh \varphi_\Theta \wh \varphi_{S_0} | \om_\xi, L, \chi_S, ) \to E_{\ST, \Theta}(\wh \varphi_\Theta | \om_\xi, L, \chi_S, \wh \varphi_{S_0}).
\]
\end{prop}
\begin{proof}
Use the above formula for $E_\ST$ and $E$ together with the previous lemma. We can compute both sides by fixing a common measure on $\wh Z_\Theta$ which makes $\vol(Z_v \cap K_v)$ constant on $v \in \mc V_F(\Theta)$. 
\end{proof}
This is a replacement for \cite[prop 5.3]{ST16} in our case.

\subsubsection{Final Statment}
Arguing as in \cite[thm 9.26]{ST16}, we get the full corollary. Note that remark 9.5 in \cite{ST16} removes the dependence on Sauvageot's result. 
\begin{cor}[Sato-Tate equidistribution up to central character]\label{SatoTate}
Choose a sequence $v_j \to \infty$ in $\mc V_F(\Theta) \cap S_1$ such that the characters $\chi_v$ all correspond in $\wh{\mf X}_\Theta$. Choose a Riemann integrable function $\wh f_\Theta$  on $\wh T_{c, \Theta}$. Then
\[
\lim_{(j,k) \to \infty} \mu_{\mc F_k, v_j}^\natural(\wh f_\Theta) = E_{\ST, \Theta}(\wh f_\Theta | \om_\xi, L, \chi_S, \wh \varphi_{S_0})
\]
where the limit is over any sequence of pairs $(j,k)$ such that $q_{v_j}^N m(\xi_k)^{-1} \to 0$ for all integers $N$. 
\end{cor}
This can be thought of as sort of a ``diagonal'' equidistribution as opposed to the ``vertical'' Plancherel equidistribution involving $\lim_{k \to \infty} \mu_{\mc F_k, v_j}^\natural(\wh f_\Theta)$ or the conjectural ``pure horizontal'' Sato-Tate equidistribution involving $\lim_{j \to \infty} \mu_{\mc F_k, v_j}^\natural(\wh f_\Theta)$.

\red{\subsection{Zeroes of $L$-functions} 
$\beta_v^\pl$ replaced by something nasty, is this worth doing??}.

\bibliography{/Users/Rahul/Documents/Rahul/Tbib}

\begin{thebibliography}{KWY18}

\bibitem[Art76]{Art76}
James Arthur.
\newblock The characters of discrete series as orbital integrals.
\newblock {\em Invent. Math.}, 32(3):205--261, 1976.

\bibitem[Art88a]{Art88I}
James Arthur.
\newblock The invariant trace formula. {I}. {L}ocal theory.
\newblock {\em J. Amer. Math. Soc.}, 1(2):323--383, 1988.

\bibitem[Art88b]{Art88}
James Arthur.
\newblock The invariant trace formula. {II}. {G}lobal theory.
\newblock {\em J. Amer. Math. Soc.}, 1(3):501--554, 1988.

\bibitem[Art89]{Art89}
James Arthur.
\newblock The {$L^2$}-{L}efschetz numbers of {H}ecke operators.
\newblock {\em Invent. Math.}, 97(2):257--290, 1989.

\bibitem[Art02]{Art02}
James Arthur.
\newblock A stable trace formula. {I}. {G}eneral expansions.
\newblock {\em J. Inst. Math. Jussieu}, 1(2):175--277, 2002.

\bibitem[Art05]{Art05}
James Arthur.
\newblock An introduction to the trace formula.
\newblock In {\em Harmonic analysis, the trace formula, and {S}himura
  varieties}, volume~4 of {\em Clay Math. Proc.}, pages 1--263. Amer. Math.
  Soc., Providence, RI, 2005.

\bibitem[Art13]{Art13}
James Arthur.
\newblock {\em The endoscopic classification of representations}, volume~61 of
  {\em American Mathematical Society Colloquium Publications}.
\newblock American Mathematical Society, Providence, RI, 2013.
\newblock Orthogonal and symplectic groups.

\bibitem[Ber74]{Ber74}
I.~N. Bernshtein.
\newblock All reductive p-adic groups are tame.
\newblock {\em Functional Analysis and Its Applications}, 8(2):91--93, Apr
  1974.

\bibitem[Bor79]{Bor79}
A.~Borel.
\newblock Automorphic {$L$}-functions.
\newblock In {\em Automorphic forms, representations and {$L$}-functions
  ({P}roc. {S}ympos. {P}ure {M}ath., {O}regon {S}tate {U}niv., {C}orvallis,
  {O}re., 1977), {P}art 2}, Proc. Sympos. Pure Math., XXXIII, pages 27--61.
  Amer. Math. Soc., Providence, R.I., 1979.

\bibitem[CCH19]{CCH19}
William Casselman, Jorge~E. Cely, and Thomas Hales.
\newblock The spherical {H}ecke algebra, partition functions, and motivic
  integration.
\newblock {\em Trans. Amer. Math. Soc.}, 371(9):6169--6212, 2019.

\bibitem[CD90]{CD90}
Laurent Clozel and Patrick Delorme.
\newblock Le th\'eor\`eme de {P}aley-{W}iener invariant pour les groupes de
  {L}ie r\'eductifs. {II}.
\newblock {\em Ann. Sci. \'Ecole Norm. Sup. (4)}, 23(2):193--228, 1990.

\bibitem[Clo86]{Clo86}
Laurent Clozel.
\newblock On limit multiplicities of discrete series representations in spaces
  of automorphic forms.
\newblock {\em Invent. Math.}, 83(2):265--284, 1986.

\bibitem[Fer07]{Fer07}
Axel Ferrari.
\newblock Th\'eor\`eme de l'indice et formule des traces.
\newblock {\em Manuscripta Math.}, 124(3):363--390, 2007.

\bibitem[FL18]{FL18}
Tobias Finis and Erez Lapid.
\newblock An approximation principle for congruence subgroups {II}: application
  to the limit multiplicity problem.
\newblock {\em Math. Z.}, 289(3-4):1357--1380, 2018.

\bibitem[FLM11]{FLM11}
Tobias Finis, Erez Lapid, and Werner M\"{u}ller.
\newblock On the spectral side of {A}rthur's trace formula---absolute
  convergence.
\newblock {\em Ann. of Math. (2)}, 174(1):173--195, 2011.

\bibitem[FLM15]{FLM15}
Tobias Finis, Erez Lapid, and Werner M\"uller.
\newblock Limit multiplicities for principal congruence subgroups of {${\rm
  GL}(n)$} and {${\rm SL}(n)$}.
\newblock {\em J. Inst. Math. Jussieu}, 14(3):589--638, 2015.

\bibitem[Fol16]{Fol16}
Gerald~B. Folland.
\newblock {\em A course in abstract harmonic analysis}.
\newblock Textbooks in Mathematics. CRC Press, Boca Raton, FL, second edition,
  2016.

\bibitem[GG19]{Ger19}
Mathilde Gerbelli-Gauthier.
\newblock Growth of cohomology of arithmetic groups and the stable trace
  formula: the case of $ u (2, 1) $.
\newblock {\em arXiv preprint arXiv:1910.06900}, 2019.

\bibitem[GKM97]{GKM97}
M.~Goresky, R.~Kottwitz, and R.~MacPherson.
\newblock Discrete series characters and the {L}efschetz formula for {H}ecke
  operators.
\newblock {\em Duke Math. J.}, 89(3):477--554, 1997.

\bibitem[Gro97]{Gro97}
Benedict~H. Gross.
\newblock On the motive of a reductive group.
\newblock {\em Invent. Math.}, 130(2):287--313, 1997.

\bibitem[Gro98]{Gro98}
Benedict~H. Gross.
\newblock On the {S}atake isomorphism.
\newblock In {\em Galois representations in arithmetic algebraic geometry
  ({D}urham, 1996)}, volume 254 of {\em London Math. Soc. Lecture Note Ser.},
  pages 223--237. Cambridge Univ. Press, Cambridge, 1998.

\bibitem[GW96]{GW96}
Benedict~H. Gross and Nolan~R. Wallach.
\newblock On quaternionic discrete series representations, and their
  continuations.
\newblock {\em J. Reine Angew. Math.}, 481:73--123, 1996.

\bibitem[Hai18]{Hai18}
Thomas~J. Haines.
\newblock Dualities for root systems with automorphisms and applications to
  non-split groups.
\newblock {\em Represent. Theory}, 22:1--26, 2018.

\bibitem[Hal95]{Hal95}
Thomas~C. Hales.
\newblock On the fundamental lemma for standard endoscopy: reduction to unit
  elements.
\newblock {\em Canad. J. Math.}, 47(5):974--994, 1995.

\bibitem[HL04]{HL04}
Michael Harris and Jean-Pierre Labesse.
\newblock Conditional base change for unitary groups.
\newblock {\em Asian J. Math.}, 8(4):653--683, 2004.

\bibitem[Kal16]{Kal16}
Tasho Kaletha.
\newblock The local {L}anglands conjectures for non-quasi-split groups.
\newblock In {\em Families of automorphic forms and the trace formula}, Simons
  Symp., pages 217--257. Springer, [Cham], 2016.

\bibitem[Kat82]{Kat82}
Shin-ichi Kato.
\newblock Spherical functions and a {$q$}-analogue of {K}ostant's weight
  multiplicity formula.
\newblock {\em Invent. Math.}, 66(3):461--468, 1982.

\bibitem[Kos61]{Kos61}
Bertram Kostant.
\newblock Lie algebra cohomology and the generalized {B}orel-{W}eil theorem.
\newblock {\em Ann. of Math. (2)}, 74:329--387, 1961.

\bibitem[Kot84]{Kot84}
Robert~E. Kottwitz.
\newblock Stable trace formula: cuspidal tempered terms.
\newblock {\em Duke Math. J.}, 51(3):611--650, 1984.

\bibitem[Kot86]{Kot86}
Robert~E. Kottwitz.
\newblock Stable trace formula: elliptic singular terms.
\newblock {\em Math. Ann.}, 275(3):365--399, 1986.

\bibitem[Kot90]{Kot90}
Robert~E. Kottwitz.
\newblock Shimura varieties and {$\lambda$}-adic representations.
\newblock In {\em Automorphic forms, {S}himura varieties, and {$L$}-functions,
  {V}ol. {I} ({A}nn {A}rbor, {MI}, 1988)}, volume~10 of {\em Perspect. Math.},
  pages 161--209. Academic Press, Boston, MA, 1990.

\bibitem[KS99]{KS99}
Robert~E. Kottwitz and Diana Shelstad.
\newblock Foundations of twisted endoscopy.
\newblock {\em Ast\'{e}risque}, (255):vi+190, 1999.

\bibitem[KST16]{KSZ16}
Ju-Lee Kim, Sug~Woo Shin, and Nicolas Templier.
\newblock Asymptotic behavior of supercuspidal representations and sato-tate
  equidistribution for families, 2016.

\bibitem[KSZ]{KSZ}
Mark Kisin, Sug-Woo Shin, and Yihang Zhu.
\newblock The stable trace formula for {S}himura varieties of abelian type.
\newblock draft.

\bibitem[KWY18]{KWY18}
Henry~H. Kim, Satoshi Wakatsuki, and Takuya Yamauchi.
\newblock An equidistribution theorem for holomorphic siegel modular forms for
  $\mathit{GSp}_{4}$ and its applications.
\newblock {\em Journal of the Institute of Mathematics of Jussieu}, pages
  1--69, 2018.

\bibitem[Lab11]{Lab11}
Jean-Pierre Labesse.
\newblock Introduction to endoscopy: {S}nowbird lectures, revised version,
  {M}ay 2010 [revision of mr2454335].
\newblock In {\em On the stabilization of the trace formula}, volume~1 of {\em
  Stab. Trace Formula Shimura Var. Arith. Appl.}, pages 49--91. Int. Press,
  Somerville, MA, 2011.

\bibitem[Lan79]{Lan79}
R.~P. Langlands.
\newblock Stable conjugacy: definitions and lemmas.
\newblock {\em Canad. J. Math.}, 31(4):700--725, 1979.

\bibitem[LS87]{LS87}
R.~P. Langlands and D.~Shelstad.
\newblock On the definition of transfer factors.
\newblock {\em Math. Ann.}, 278(1-4):219--271, 1987.

\bibitem[LV83]{LV83}
George Lusztig and David~A. Vogan, Jr.
\newblock Singularities of closures of {$K$}-orbits on flag manifolds.
\newblock {\em Invent. Math.}, 71(2):365--379, 1983.

\bibitem[Mar14]{Mar14}
Simon Marshall.
\newblock Endoscopy and cohomology growth on {$U(3)$}.
\newblock {\em Compos. Math.}, 150(6):903--910, 2014.

\bibitem[MS82]{MS82}
James~S. Milne and Kuang-yen Shih.
\newblock Conjugates of shimura varieties.
\newblock In {\em Hodge cycles, motives, and {S}himura varieties}, volume 900
  of {\em Lecture Notes in Mathematics}, pages 280--356. Springer-Verlag,
  Berlin-New York, 1982.

\bibitem[MS19]{MS19}
Simon Marshall and Sug~Woo Shin.
\newblock Endoscopy and cohomology in a tower of congruence manifolds for
  {$U(n,1)$}.
\newblock {\em Forum Math. Sigma}, 7:e19, 46, 2019.

\bibitem[Ng{\^{o}}10]{Ngo10}
Bao~Ch\^{a}u Ng{\^{o}}.
\newblock Le lemme fondamental pour les alg\`ebres de {L}ie.
\newblock {\em Publ. Math. Inst. Hautes \'{E}tudes Sci.}, (111):1--169, 2010.

\bibitem[NV19]{NV19}
Paul~D. Nelson and Akshay Venkatesh.
\newblock The orbit method and analysis of automorphic forms, 2019.

\bibitem[Pen19]{Pen19}
Zhifeng Peng.
\newblock Multiplicity formula and stable trace formula.
\newblock {\em Amer. J. Math.}, 141(4):1037--1085, 2019.

\bibitem[Sar05]{Sar05}
Peter Sarnak.
\newblock Notes on the generalized {R}amanujan conjectures.
\newblock In {\em Harmonic analysis, the trace formula, and {S}himura
  varieties}, volume~4 of {\em Clay Math. Proc.}, pages 659--685. Amer. Math.
  Soc., Providence, RI, 2005.

\bibitem[Sau97]{Sau97}
Fran\c{c}ois Sauvageot.
\newblock Principe de densit\'{e} pour les groupes r\'{e}ductifs.
\newblock {\em Compositio Math.}, 108(2):151--184, 1997.

\bibitem[Sch17]{Sch17}
Ralf Schmidt.
\newblock Archimedean aspects of siegel modular forms of degree 2.
\newblock {\em Rocky Mountain J. Math.}, 47(7):2381--2422, 11 2017.

\bibitem[She82]{She82}
D.~Shelstad.
\newblock {$L$}-indistinguishability for real groups.
\newblock {\em Math. Ann.}, 259(3):385--430, 1982.

\bibitem[She10]{She10}
D.~Shelstad.
\newblock A note on real endoscopic transfer and pseudo-coefficients, 2010.

\bibitem[Shi10]{Shi10}
Sug~Woo Shin.
\newblock A stable trace formula for {I}gusa varieties.
\newblock {\em J. Inst. Math. Jussieu}, 9(4):847--895, 2010.

\bibitem[Shi12a]{Shi12}
Sug~Woo Shin.
\newblock Automorphic {P}lancherel density theorem.
\newblock {\em Israel J. Math.}, 192(1):83--120, 2012.

\bibitem[Shi12b]{Shi12a}
Sug~Woo Shin.
\newblock On the cohomology of {R}apoport-{Z}ink spaces of {EL}-type.
\newblock {\em Amer. J. Math.}, 134(2):407--452, 2012.

\bibitem[SR99]{Sal99}
Susana~A. Salamanca-Riba.
\newblock On the unitary dual of real reductive {L}ie groups and the
  {$A_g(\lambda)$} modules: the strongly regular case.
\newblock {\em Duke Math. J.}, 96(3):521--546, 1999.

\bibitem[ST16]{ST16}
Sug~Woo Shin and Nicolas Templier.
\newblock Sato-{T}ate theorem for families and low-lying zeros of automorphic
  {$L$}-functions.
\newblock {\em Invent. Math.}, 203(1):1--177, 2016.
\newblock Appendix A by Robert Kottwitz, and Appendix B by Raf Cluckers, Julia
  Gordon and Immanuel Halupczok.

\bibitem[Vir15]{Vir15}
R.~Virk.
\newblock Some geometric facets of the {L}anglands correspondence for real
  groups.
\newblock {\em Bull. Lond. Math. Soc.}, 47(2):225--232, 2015.

\bibitem[Vog81]{Vog81}
David~A. Vogan, Jr.
\newblock {\em Representations of real reductive {L}ie groups}, volume~15 of
  {\em Progress in Mathematics}.
\newblock Birkh\"{a}user, Boston, Mass., 1981.

\bibitem[Wal97]{Wal97}
J.-L. Waldspurger.
\newblock Le lemme fondamental implique le transfert.
\newblock {\em Compositio Math.}, 105(2):153--236, 1997.

\bibitem[Wal10]{Wal10}
J.-L. Waldspurger.
\newblock Les facteurs de transfert pour les groupes classiques: un formulaire.
\newblock {\em Manuscripta Math.}, 133(1-2):41--82, 2010.

\end{thebibliography}
\bibliographystyle{alpha}

\end{document}